\documentclass{amsart}
\usepackage{amssymb}
\usepackage{amscd}
\usepackage{verbatim}
\usepackage{epsfig}
\begin{document}

\newcommand\Mand{\ \text{and}\ }
\newcommand\Mwith{\ \text{with}\ }
\newcommand\Mwhere{\ \text{where}\ }
\newcommand\Mfor{\ \text{for}\ }
\newcommand\Mst{\ \text{such that}\ }
\newcommand\Mor{\ \text{or}\ }
\newcommand\Mon{\ \text{on}\ }
\newcommand\Min{\ \text{in}\ }
\newcommand\Mif{\ \text{if}\ }
\newcommand\Miff{\ \text{iff}\ }
\newcommand\Mthen{\ \text{then}\ }
\newcommand\nin{\notin}
\newcommand\identity{\operatorname{id}}
\newcommand\Id{\operatorname{Id}}
\newcommand\Real{\mathbb{R}}
\newcommand\RR{\mathbb{R}}
\newcommand\pos{\Real^+}
\newcommand\Rnp{\Real\setminus\{0\}}
\newcommand\nzero{\setminus\{0\}}
\newcommand\Cx{\mathbb{C}}
\newcommand\Cxp{\Cx^+}
\newcommand\Cxm{\Cx^-}
\newcommand\NN{\mathbb{N}}
\newcommand\Nat{\mathbb{N}}
\newcommand\halfNat{{\frac{1}{2}}\mathbb{N}}
\newcommand\intgr{\mathbb{Z}}
\newcommand\BB{\mathbb{B}}
\newcommand\HH{\mathbb{H}}
\newcommand\im{\operatorname{Im}}
\newcommand\re{\operatorname{Re}}
\newcommand\sign{\operatorname{sign}}
\newcommand\codim{\operatorname{codim}}
\newcommand\End{\operatorname{End}}
\newcommand\Ker{\operatorname{Ker}}
\newcommand\Hom{\operatorname{Hom}}
\newcommand\tr{\operatorname{tr}}
\newcommand\Tr{\operatorname{Tr}}
\newcommand\ideal{{\mathcal I}}
\newcommand\Span{\operatorname{span}}
\newcommand\image{\operatorname{image}}
\newcommand\Range{\operatorname{Ran}}
\newcommand\Graph{\operatorname{graph}}
\newcommand\Card{\operatorname{Card}}
\newcommand\Hess{\operatorname{Hess}}
\newcommand\slim{\operatornamewithlimits{s-lim}}
\newcommand\spp{\operatorname{sp}}
\newcommand\sll{\operatorname{sl}}
\newcommand\sol{\operatorname{so}}
\newcommand\SL{\operatorname{SL}}
\newcommand\SO{\operatorname{SO}}
\newcommand\On{\operatorname{O}}
\newcommand\pa{\partial}
\newcommand\eff{\mathrm{eff}}
\newcommand\Rn{\Real^n}
\newcommand\Rm{\Real^m}
\newcommand\RN{\Real^N}
\newcommand\RtN{\Real^{2N}}
\newcommand\RM{\Real^M}
\newcommand\sphere{\mathbb{S}}
\newcommand\Sn{\sphere^{n-1}}
\newcommand\Sm{\sphere^{m-1}}
\newcommand\Snp{\sphere^n_+}
\newcommand\Smp{\sphere^m_+}
\newcommand\SN{\sphere^{N-1}}
\newcommand\SNp{\sphere^N_+}
\newcommand\circlep{\sphere^1_+}
\newcommand\Phom{P_{h}}
\newcommand\Shom{S_{h}}
\newcommand\distance{\operatorname{dist}}
\newcommand\cl{\operatorname{cl}}
\newcommand\interior{\operatorname{int}}
\newcommand\Fa{\operatorname{Fa}}
\newcommand\ff{\operatorname{ff}}
\newcommand\mf{\operatorname{mf}}
\newcommand\cf{\operatorname{cf}}
\newcommand\scf{\operatorname{sf}}
\newcommand\lf{\operatorname{lf}}
\newcommand\rf{\operatorname{rf}}
\newcommand\indfam{{\mathcal K}}
\newcommand\fraka{{\mathfrak a}}
\newcommand\cB{{\mathcal B}}
\newcommand\calA{{\mathcal A}}
\newcommand\calB{{\mathcal B}}
\newcommand\calR{{\mathcal R}}
\newcommand\calO{{\mathcal O}}
\newcommand\calJ{{\mathcal J}}
\newcommand\calL{{\mathcal L}}
\newcommand\cL{{\mathcal L}}
\newcommand\calM{{\mathcal M}}
\newcommand\calN{{\mathcal N}}
\newcommand\calX{{\mathcal X}}
\newcommand\calU{{\mathcal U}}
\newcommand\calV{{\mathcal V}}
\newcommand\calF{{\mathcal F}}
\newcommand\calG{{\mathcal G}}
\newcommand\calT{{\mathcal T}}
\newcommand\calC{{\mathcal C}}
\newcommand\calCt{{\tilde {\mathcal C}}}
\newcommand\calCL{{\mathcal C}_{\text L}}
\newcommand\calCR{{\mathcal C}_{\text R}}
\newcommand\cF{{\mathcal F}}
\newcommand\cE{{\mathcal E}}
\newcommand\cH{{\mathcal H}}
\newcommand\cG{{\mathcal G}}
\newcommand\cU{{\mathcal U}}
\newcommand\cC{{\mathcal C}}
\newcommand\cX{{\mathcal X}}
\newcommand\cY{{\mathcal Y}}
\newcommand\cS{{\mathcal S}}
\newcommand\CI{{\mathcal C}^{\infty}}
\newcommand\Cinf{{\mathcal C}^{\infty}}
\newcommand\dist{{\mathcal C}^{-\infty}}
\newcommand\dCinf{{\dot{\mathcal C}}^{\infty}}
\newcommand\dCI{{\dot{\mathcal C}}^{\infty}}
\newcommand\ddist{\dot\dist}
\newcommand\Cj{{\mathcal C}^j}
\newcommand\Linf{L^{\infty}}
\newcommand\phg{{\text{phg}}}
\newcommand\comp{{\text{comp}}}
\newcommand\loc{{\text{loc}}}
\renewcommand{\Box}{{\square}}
\newcommand\bcon{{\mathcal A}}
\newcommand\bconc{{\mathcal A}_{\text{phg}}}
\newcommand\Sch{{\mathcal S}}
\newcommand\temp{\Sch^{\prime}}
\newcommand\Diff{\operatorname{Diff}}
\newcommand\Diffb{\operatorname{Diff}_{\text{b}}}
\newcommand\Diffc{\operatorname{Diff}_{\text{c}}}
\newcommand\Diffsc{\operatorname{Diff}_{\text{sc}}}
\newcommand\DiffI{\operatorname{Diff}_{\text{I}}}
\newcommand\DiffIq{\operatorname{Diff}_{\text{I},q}}
\newcommand\sing{\text{sing}}
\newcommand\reg{\text{reg}}
\newcommand\supp{\operatorname{supp}}
\newcommand\ssupp{\operatorname{sing\ supp}}
\newcommand\csupp{\operatorname{cone\ supp}}
\newcommand\esupp{\operatorname{ess\ supp}}
\newcommand\Fr{{\mathcal F}}
\newcommand\Frinv{\Fr^{-1}}
\newcommand\bop{{\mathcal B}}
\newcommand\spec{\operatorname{spec}}
\newcommand\pspec{\spec_{pp}}
\newcommand\cspec{\spec_{c}}
\newcommand\FIO{{\mathcal I}}
\newcommand\SP{\operatorname{RC}}
\newcommand\RC{\operatorname{RC}}
\newcommand\Symc{S_c}
\newcommand\Symca{S_c^{\alpha}}
\newcommand\Symczero{S_c^{0,...,0}}
\newcommand\zi{{}^{0}}
\newcommand\sci{{}^{\text{sc}}}
\newcommand\sct{\sci T^*}
\newcommand\scT{\sci T}
\newcommand\zT{\zi T}
\newcommand\zS{\zi S}
\newcommand\scdt{\sci \dot T^*}
\newcommand\dS{\dot S^*}
\newcommand\dT{\dot T^*}
\newcommand\dSreg{\dot\Sigma_{\text reg}}
\newcommand\scct{\sci\bar{T}^*}
\newcommand\zcT{\zi\bar{T}}
\newcommand\Csc{C_{\text{sc}}}
\newcommand\SNpscd{(\SNp)^2_{\text{sc}}}
\newcommand\scdiag{\Delta_{\text{sc}}}
\newcommand\projscl{\pi^L_{\text{sc}}}
\newcommand\projscr{\pi^R_{\text{sc}}}
\newcommand\scHL{\sci H^{2,0}_{|\zeta|^2-\lambda^2}}
\newcommand\scHrg{\sci H^{2,0}_{\sqrt{g}}}
\newcommand\Hsc{H_{\text{sc}}}
\newcommand\Char{\operatorname{Char}}
\newcommand\dChar{\operatorname{\dot Char}}
\newcommand\WF{\operatorname{WF}}
\newcommand\zWF{\operatorname{WF}_0}
\newcommand\WFb{\operatorname{WF}_{\bl}}
\newcommand\WFbz{\operatorname{WF}_{\bl}}
\newcommand\WFbd{\operatorname{WF}_{\bl}}
\newcommand\WFp{\operatorname{WF^{\prime}}}
\newcommand\WFsc{\operatorname{WF}_{\text{sc}}}
\newcommand\WFscp{\operatorname{WF_{sc}^{\prime}}}
\newcommand\WFC{\operatorname{WF}_C}
\newcommand\WFCi{\operatorname{WF}_{C_i}}
\newcommand\elliptic{\operatorname{ell}}
\newcommand\Psop{\operatorname{\Psi}}
\newcommand\zPs{\operatorname{\Psi}_0}
\newcommand\Psib{\operatorname{\Psi}_{\bl}}
\newcommand\Psibc{\operatorname{\Psi}_{\text{bc}}}
\newcommand\Psiscrs{\operatorname{\Psi_{sc}^{-2,\infty}}}
\newcommand\Psiscr{\operatorname{\Psi_{sc}^{-2,0}}}
\newcommand\Psiscrm{\operatorname{\Psi_{sc}^{0,2}}}
\newcommand\PsiscHam{\operatorname{\Psi_{sc}^{2,0}}}
\newcommand\Psisci{\operatorname{\Psi_{sc}^{*,*}}}
\newcommand\Psiscid{\operatorname{\Psi_{sc}^{0,0}}}
\newcommand\Psiscis{\operatorname{\Psi_{sc}^{0,\infty}}}
\newcommand\Psiscsi{\operatorname{\Psi_{sc}^{-\infty,0}}}
\newcommand\Psiscs{\operatorname{\Psi_{sc}^{-\infty,\infty}}}
\newcommand\Psiscalg{\operatorname{\Psi_{sc}^{\infty,-\infty}}}
\newcommand\nullHam{{\mathcal N}}
\newcommand\charD{\Sigma_{\Delta-\lambda^2}}
\newcommand\charLap{\Sigma_{\Delta-\lambda}}
\newcommand\Snl{\Sn_{\lambda}}
\newcommand\SNl{\SN_{\lambda}}
\newcommand\gammat{\tilde\gamma}
\newcommand\gammasc{\gamma}
\newcommand\Tau{\mathcal{T}}
\newcommand\taut{\tilde\tau}
\newcommand\taub{\bar\tau}
\newcommand\Nout{N^+_{\lambda}}
\newcommand\Nin{N^-_{\lambda}}
\newcommand\Nio{N^{\pm}_{\lambda}}
\newcommand\El{E_{\lambda}}
\newcommand\Elt{\tilde E_{\lambda}}
\newcommand\Eil{E^i_{\lambda}}
\newcommand\Ejl{E^j_{\lambda}}
\newcommand\Eajl{E^{\alpha_j}_{\lambda}}
\newcommand\Eilt{\tilde E^i_{\lambda}}
\newcommand\Np{N^+}
\newcommand\Nm{N^-}
\newcommand\Npm{N^{\pm}}
\newcommand\Fin{F^-(\lambda)}
\newcommand\Fini{F^-_i(\lambda)}
\newcommand\Fout{F^+(\lambda)}
\newcommand\Fouti{F^+_i(\lambda)}
\newcommand\Foutj{F^+_j(\lambda)}
\newcommand\Rout{R^+_{\lambda}}
\newcommand\Routl{R^+_{\lambda^2}}
\newcommand\Routsgnl{R^{\sign\lambda}_{\lambda^2}}
\newcommand\Rin{R^-_{\lambda}}
\newcommand\Rinl{R^-_{\lambda^2}}
\newcommand\Rinsgnl{R^{-\sign\lambda}_{\lambda^2}}
\newcommand\Rio{R^{\pm}_{\lambda}}
\newcommand\Riol{R^{\pm}_{\lambda^2}}
\newcommand\Roi{R^{\mp}_{\lambda}}
\newcommand\Roil{R^{\mp}_{\lambda^2}}
\newcommand\Riob{R^{\pm}}
\newcommand\Roib{R^{\mp}}
\newcommand\Tio{T^{\pm}}
\newcommand\Tiob{T^{\pm}_{\ff}}
\newcommand\Toi{T^{\mp}}
\newcommand\Toib{T^{\mp}_{\ff}}
\newcommand\TIiob{T_I^{\pm}}
\newcommand\Rinb{R^-}
\newcommand\Rinbsgnl{R^{-\sign\lambda}}
\newcommand\Tin{T^-}
\newcommand\Tinb{T^-_{\ff}}
\newcommand\TIinb{T^-_I}
\newcommand\Routb{R^+}
\newcommand\Routbsgnl{R^{\sign\lambda}}
\newcommand\Tout{T^+}
\newcommand\Toutb{T^+_{\ff}}
\newcommand\TIoutb{T^+_I}
\newcommand\Rlkf{(|\xib|^2-(\lambda-i0)^2)^{-1}}
\newcommand\Rlk{\rho_0(\lambda)}
\newcommand\Rmlk{\rho_0(-\lambda)}
\newcommand\Rpmlk{\rho_0(\pm\lambda)}
\newcommand\Rlka{\rho_1(\lambda)}
\newcommand\Rlkb{\rho_2(\lambda)}
\newcommand\Rilk{\rho_i(\lambda)}
\newcommand\reduced{\natural}
\newcommand\Rlf{R_0(\lambda)}
\newcommand\Rla{R_1(\lambda)}
\newcommand\Rlb{R_2(\lambda)}
\newcommand\Ril{R_i(\lambda)}
\newcommand\Rlj{R_j(\lambda)}
\newcommand\Rlft{R_0(\lambda)}
\newcommand\Rflambda{R_0^{\reduced}(\sigma)}
\newcommand\RV{R^{\reduced}_V}
\newcommand\Rfsigma{R_0^{\reduced}(\sigma)}
\newcommand\Rfsigmah{R_0^{\reduced}(\sigma^{1/2})}
\newcommand\Rfzero{R_0^{\reduced}(0)}
\newcommand\RlV{R^{\reduced}_V(\sigma)}
\newcommand\RlVi{R^{\reduced}_{V_i}(\sigma)}
\newcommand\RlVt{R_V(\lambda)}
\newcommand\RlVtL{{R}_V^L(\lambda)}
\newcommand\RlVtR{{R}_V^R(\lambda)}
\newcommand\RlVit{{R}_{V_i}(\lambda)}
\newcommand\RlVta{{R}_V^{(1)}(\lambda)}
\newcommand\RlVtk{{R}_V^{(k)}(\lambda)}
\newcommand\RlVatV{{R}_{V_{\alpha}}(\lambda)V_{\alpha}}
\newcommand\RlVatVa{{R}_{V_{\alpha_1}}(\lambda)V_{\alpha_1}}
\newcommand\RlVatVb{{R}_{V_{\alpha_2}}(\lambda)V_{\alpha_2}}
\newcommand\RlVatVk{{R}_{V_{\alpha_k}}(\lambda)V_{\alpha_k}}
\newcommand\RlVatVkk{{R}_{V_{\alpha_{k+1}}}(\lambda)V_{\alpha_{k+1}}}
\newcommand\RlVaptV{{R}_{V_{\alpha'}}(\lambda)V_{\alpha'}}
\newcommand\RlVapptV{{R}_{V_{\alpha''}}(\lambda)V_{\alpha''}}
\newcommand\RlVajtV{{R}_{V_{\alpha_j}}(\lambda)V_{\alpha_j}}
\newcommand\RlVaktV{{R}_{V_{\alpha_k}}(\lambda)V_{\alpha_k}}
\newcommand\RlVakktV{{R}_{V_{\alpha_{k+1}}}(\lambda)V_{\alpha_{k+1}}}
\newcommand\Tl{T(\lambda)}
\newcommand\Tlt{\tilde\Tl}
\newcommand\Tltp{\tilde T'(\lambda)}
\newcommand\Tltpp{\tilde T''(\lambda)}
\newcommand\Tli{T_i(\lambda)}
\newcommand\Tlit{\tilde\Tli}
\newcommand\Tlip{T_i'(\lambda)}
\newcommand\Tlipp{T_i''(\lambda)}
\newcommand\Tlj{T_j(\lambda)}
\newcommand\Tla{T_{\alpha}(\lambda)}
\newcommand\Tlaa{T_{\alpha_1}(\lambda)}
\newcommand\Tlab{T_{\alpha_2}(\lambda)}
\newcommand\Tlak{T_{\alpha_k}(\lambda)}
\newcommand\Tlakt{\tilde\Tlak}
\newcommand\Tlaj{T_{\alpha_j}(\lambda)}
\newcommand\Tlajj{T_{\alpha_{j+1}}(\lambda)}
\newcommand\Tlajp{T_{\alpha_j}'(\lambda)}
\newcommand\Tlajpt{\tilde\Tlajp}
\newcommand\Tlajt{\tilde\Tlaj}
\newcommand\Tlakk{T_{\alpha_{k+1}}(\lambda)}
\newcommand\Tlakkp{T_{\alpha_{k+1}}'(\lambda)}
\newcommand\Tlap{T_{\alpha'}(\lambda)}
\newcommand\Tlapt{\tilde\Tlap}
\newcommand\Tlapp{T_{\alpha''}(\lambda)}
\newcommand\Tkl{T^{(k)}(\lambda)}
\newcommand\Tcl{T^{\flat}(\lambda)}
\newcommand\Fl{F(\lambda)}
\newcommand\BlVt{\tilde B_V(\lambda)}
\newcommand\KBlVt{K_{\BlVt}}
\newcommand\BlVaat{B_{V_{\alpha_1}}(\lambda)}
\newcommand\BV{B_V}
\newcommand\Bone{B_1}
\newcommand\Btwo{B_2}
\newcommand\Bthree{B_3}
\newcommand\Banyj{B_j}
\newcommand\PlV{P_V(\lambda)}
\newcommand\PlVc{P_V^{\flat}(\lambda)}
\newcommand\Pl{P_0(\lambda)}
\newcommand\SVl{S_V(\lambda)}
\newcommand\Sjr{S_j^{\reduced}}
\newcommand\Rkp{{\mathcal R}^k_+}
\newcommand\Rkm{{\mathcal R}^k_-}
\newcommand\Rkpm{{\mathcal R}^k_{\pm}}
\newcommand\Phys{{\mathcal P}}
\newcommand\Pc{\overline{\mathcal P}}
\newcommand\pip{\pi^{\perp}}
\newcommand\pipa{\pi_\partial}
\newcommand\gammapa{\gamma_\partial}
\newcommand\pipah{\hat\pi_\partial}
\newcommand\pit{\tilde\pi}
\newcommand\xit{\tilde\xi}
\newcommand\zetat{\tilde\zeta}
\newcommand\etat{\tilde\eta}
\newcommand\sigmat{\tilde\sigma}
\newcommand\sigmahat{\hat\sigma}
\newcommand\thetat{\tilde\theta}
\newcommand\psit{\tilde\psi}
\newcommand\phit{\tilde\phi}
\newcommand\chit{\tilde\chi}
\newcommand\rhot{\tilde\rho}
\newcommand\xib{\bar\xi}
\newcommand\zetab{\bar\zeta}
\newcommand\thetab{\bar\theta}
\newcommand\etab{\bar\eta}
\newcommand\iotal{\iota_{\lambda}}
\newcommand\rhoat{\rhot_{\alpha_1}}
\newcommand\Lambdat{\tilde\Lambda}
\newcommand\Lambdati{\tilde\Lambda^{\text{in}}}
\newcommand\Lambdato{\tilde\Lambda^{\text{out}}}
\newcommand\Lambdatp{\tilde\Lambda^{\text{prop}}}
\newcommand\Lambdai{\Lambda^{\text{in}}}
\newcommand\Lambdao{\Lambda^{\text{out}}}
\newcommand\poles{\Lambda'}
\newcommand\rpoles{\Lambda_p}
\newcommand\thresholds{\Lambda}
\newcommand\Vt{\tilde V}
\newcommand\It{\tilde I}
\newcommand\half{{\frac{1}{2}}}
\newcommand\sigmah{\sigma^{1/2}}
\newcommand\bX{\partial X}
\newcommand\bXb{\partial \Xb}
\newcommand\Deltabt{\tilde\Delta_0}
\newcommand\strip{\Omega_T}
\newcommand\Kf{K^{\flat}}
\newcommand\Gs{G^{\sharp}}
\newcommand\Gt{\tilde G}
\newcommand\Osc{\sci\Omega}
\newcommand\OSc{{}^\Scl\Omega}
\newcommand\Osch{\sci\Omega^{\half}}
\newcommand\Oscmh{\sci\Omega^{-\half}}
\newcommand\Isc{I_{sc}}
\newcommand\os{{\text{os}}}
\newcommand\Qzl{Q^0_{-\lambda}}
\newcommand\Lie{{\mathcal L}}
\newcommand\bl{{\text b}}
\newcommand\scl{{\text{sc}}}
\newcommand\sccl{{\text{scc}}}
\newcommand\Scl{{\text{sc}}}
\newcommand\ScLl{{\text{Sc,L}}}
\newcommand\ScRl{{\text{Sc,R}}}
\newcommand\Sccl{{\text{Scc}}}
\newcommand\sus{{\text{sus}}}
\newcommand\ssl{{\text{ss}}}
\newcommand\XXb{X^2_\bl}
\newcommand\XXbt{\Xt^2_\bl}
\newcommand\XXsc{X^2_\scl}
\newcommand\XXsct{\Xt^2_\scl}
\newcommand\XXSc{X^2_\Scl}
\newcommand\XXSct{\Xt^2_\Scl}
\newcommand\XXScL{X^2_\ScLl}
\newcommand\XXScR{X^2_\ScRl}
\newcommand\MMsc{M^2_\scl}
\newcommand\Deltab{\Delta_\bl}
\newcommand\Deltasc{\Delta_\scl}
\newcommand\DeltaSc{\Delta_\Scl}
\newcommand\DeltaScL{\Delta_\ScLl}
\newcommand\DeltaScR{\Delta_\ScRl}
\newcommand\prs{\sigma}
\newcommand\Nsc{N_\scl}
\newcommand\Nscp{N_{\scl,p}}
\newcommand\Nff{N_{\ff}}
\newcommand\Nffz{N_{\ff,0}}
\newcommand\Nffzp{N_{\ff,0,p}}
\newcommand\Nffl{N_{\ff,l}}
\newcommand\Nffml{N_{\ff,-l}}
\newcommand\Nmf{N_{\mf}}
\newcommand\Nmfz{N_{\mf,0}}
\newcommand\Nmfl{N_{\mf,l}}
\newcommand\Nmfml{N_{\mf,-l}}
\newcommand\ffb{\operatorname{bf}}
\newcommand\Ffb{\operatorname{bf'}}
\newcommand\ffsc{\operatorname{sf}}
\newcommand\ffSc{\operatorname{sf_C}}
\newcommand\Ffsc{\operatorname{sf'}}
\newcommand\rff{\rho_{\ff}}
\newcommand\rmf{\rho_{\mf}}
\newcommand\rffb{\rho_{\ffb}}
\newcommand\rffsc{\rho_{\ffsc}}
\newcommand\rFfsc{\rho_{\Ffsc}}
\newcommand\rffSc{\rho_{\ffSc}}
\newcommand\rinf{\rho_{\infty}}
\newcommand\CL{C_L}
\newcommand\CR{C_R}
\newcommand\betab{\beta_\bl}
\newcommand\betasc{\beta_\scl}
\newcommand\betaSc{\beta_\Scl}
\newcommand\BetaSc{\bar\beta_\Scl}
\newcommand\betaScL{\beta_\ScLl}
\newcommand\betaScR{\beta_\ScRl}
\newcommand\ScT{{}^\Scl T^*}
\newcommand\SccT{{}^\Scl \bar T^*}
\newcommand\ScS{{}^\Scl S^*}
\newcommand\Sb{{}^\bl S}
\newcommand\Tb{{}^\bl T}
\newcommand\Tsc{{}^\scl T}
\newcommand\TSc{{}^\Scl T}
\newcommand\CSc{C_\Scl}
\newcommand\Lambdasc{{}^\scl\Lambda}
\newcommand\XXXb{X^3_\bl}
\newcommand\XXXsc{X^3_\scl}
\newcommand\XXXSc{X^3_\Scl}
\newcommand\XXXScO{X^3_{\Scl,O}}
\newcommand\XXXScF{X^3_{\Scl,F}}
\newcommand\XXXScS{X^3_{\Scl,S}}
\newcommand\XXXScC{X^3_{\Scl,C}}
\newcommand\KDsc{\operatorname{KD^{\half}_\scl}}
\newcommand\KDSc{\operatorname{KD^{\half}_\Scl}}
\newcommand\KDScEF{\operatorname{KD^{E,F}_\Scl}}
\newcommand\Oh{\operatorname{\Omega^{\half}}}
\newcommand\WFSc{\WF_\Scl}
\newcommand\WFtSc{\WF_{\text 3sc}}
\newcommand\WFScmf{\WF_{\Scl,\mf}}
\newcommand\WFScff{\WF_{\Scl,\ff}}
\newcommand\WFScs{\WF_{\Scl,\prs}}
\newcommand\WFScp{\WF'_\Scl}
\newcommand\WFScmfp{\WF'_{\Scl,\mf}}
\newcommand\WFScffp{\WF'_{\Scl,\ff}}
\newcommand\WFScsp{\WF'_{\Scl,\prs}}
\newcommand\Diffscc{\Diff_\sccl}
\newcommand\DiffSc{\Diff_\Scl}
\newcommand\DiffScc{\Diff_\Sccl}
\newcommand\DiffscI{\Diff_{\scl,\text{I}}}
\newcommand\VscI{\Vf_{\scl,\text{I}}}
\newcommand\DiffsV{\operatorname{Diff}_{\sus(V)}}
\newcommand\DiffsVsc{\operatorname{Diff}_{\sus(V),\scl}}
\newcommand\DiffsVCsc{\operatorname{Diff}_{\sus(V)-C,\scl}}   
\newcommand\Psisc{\Psop_\scl}
\newcommand\Psiscc{\Psop_\sccl}
\newcommand\Psiss{\Psop_\ssl}
\newcommand\Psisch{\Psop_{\scl,h}}
\newcommand\Psiscch{\Psop_{\sccl,h}}
\newcommand\PsiSc{\Psop_\Scl}
\newcommand\PsiScph{\Psop_{\Scl,\phi}}
\newcommand\PsiScra{\Psop_{\Scl,\rho^\sharp_a}}
\newcommand\PsiScc{\Psop_\Sccl}
\newcommand\PsiSccml{\Psop^{m,l}_\Sccl}
\newcommand\PsiScxx{\Psop^{*,*}_\Scl}
\newcommand\PsiScml{\Psop^{m,l}_\Scl}
\newcommand\PsiScmz{\Psop^{m,0}_\Scl}
\newcommand\PsiScmmz{\Psop^{-m,0}_\Scl}
\newcommand\PsiSckz{\Psop^{k,0}_\Scl}
\newcommand\PsiScmmml{\Psop^{-m,-l}_\Scl}
\newcommand\Psiscmkk{\Psop^{-k,k}_\scl}
\newcommand\Psiscmmmkk{\Psop^{-m-k,k}_\scl}
\newcommand\Psiscmoo{\Psop^{-1,1}_\scl}
\newcommand\Psiscmz{\Psop^{m,0}_\scl}
\newcommand\Psiscmmz{\Psop^{-m,0}_\scl}
\newcommand\PsiSckmkl{\Psop^{km,kl}_\Scl}
\newcommand\PsiScmplp{\Psop^{m',l'}_\Scl}
\newcommand\PsiScmmpllp{\Psop^{m+m',l+l'}_\Scl}
\newcommand\Psiscml{\Psop^{m,l}_\scl}
\newcommand\PsiScid{\Psop^{0,0}_\Scl}
\newcommand\PsiSczo{\Psop^{0,1}_\Scl}
\newcommand\PsiScmii{\Psop^{-\infty,\infty}_\Scl}
\newcommand\PsiScmiz{\Psop^{-\infty,0}_\Scl}
\newcommand\PsiScmoo{\Psop^{-1,1}_\Scl}
\newcommand\PsisCid{\Psop^{0,0}_{\scl-C}}
\newcommand\PsisC{\Psop_{\scl-C}}
\newcommand\Psiinf{\Psop_{\infty}}
\newcommand\Psiinfid{\Psop_{\infty}^0}
\newcommand\PsiFinf{\Psop_{\infty-\Fr}}
\newcommand\PsisVscml{\Psop^{m,l}_{\sus(V),\scl}}
\newcommand\PsisVsc{\Psop_{\sus(V),\scl}}
\newcommand\PsisVpsc{\Psop_{\sus(V_p),\scl}}
\newcommand\PsisVCSc{\Psop_{\sus(V)-C,\scl}}
\newcommand\SFinf{S_{\infty-\Fr}}
\newcommand\YsVC{Y^2_{\sus(V)-C,\scl}}
\newcommand\ffYsc{\ffsc_{\sus(V)}}
\newcommand\SXC{S(X;C)}
\newcommand\Ios{I_{\text{os}}}
\newcommand\pbL{\pi^2_{\bl,{\text L}}}
\newcommand\pbR{\pi^2_{\bl,{\text R}}}
\newcommand\pscL{\pi^2_{\scl,{\text L}}}
\newcommand\pscR{\pi^2_{\scl,{\text R}}}
\newcommand\PbO{\pi^3_{\bl,{\text O}}}
\newcommand\PscO{\pi^3_{\scl,{\text O}}}
\newcommand\PScO{\pi^3_{\Scl,{\text O}}}
\newcommand\PScF{\pi^3_{\Scl,{\text F}}}
\newcommand\PScC{\pi^3_{\Scl,{\text C}}}
\newcommand\PScS{\pi^3_{\Scl,{\text S}}}
\newcommand\pScL{\pi^2_{\Scl,{\text L}}}
\newcommand\pScR{\pi^2_{\Scl,{\text R}}}
\newcommand\CLF{\CL^F}
\newcommand\CLO{\CL^O}
\newcommand\CLS{\CL^S}
\newcommand\CLC{\CL^C}
\newcommand\DeltaYb{\Delta_{\bl,Y}}
\newcommand\DeltaYsc{\Delta_{\sus-\scl}}
\newcommand\diag{\operatorname{diag}}
\newcommand\Vf{{\mathcal V}}
\newcommand\Vb{{\mathcal V}_{\bl}}
\newcommand\Vsc{{\mathcal V}_{\scl}}
\newcommand\VSc{{\mathcal V}_{\Scl}}
\newcommand\VfI{\Vf_{\text{I}}}
\newcommand\VfIq{\Vf_{\text{I},q}}
\newcommand\scH{{}^\scl H}
\newcommand\scHg{\scH_g}
\newcommand\Hss{H_\ssl}
\newcommand\xh{\hat x}
\newcommand\yh{\hat y}
\newcommand\sh{\hat s}
\newcommand\rh{\hat r}
\newcommand\Yh{\hat Y}
\newcommand\Zh{\hat Z}
\newcommand\Yb{\bar Y}
\newcommand\hb{\bar h}
\newcommand\xih{\hat\xi}
\newcommand\etah{\hat\eta}
\newcommand\muh{\hat\mu}
\newcommand\mub{\bar\mu}
\newcommand\nub{\bar\nu}
\newcommand\mubh{\widehat{\bar\mu}}
\newcommand\yb{\bar y}
\newcommand\zb{\bar z}
\newcommand\ub{\bar u}
\newcommand\Qb{\bar Q}
\newcommand\Wbp{{\bar W}^\perp}
\newcommand\Wp{W^\perp}
\newcommand\Kt{\tilde K}
\newcommand\Wt{\tilde W}
\newcommand\Ut{\tilde U}
\newcommand\yt{\tilde y}
\newcommand\ut{\tilde u}
\newcommand\vt{\tilde v}
\newcommand\ft{\tilde f}
\newcommand\htil{\tilde h}
\newcommand\St{\tilde S}
\newcommand\Pt{\tilde P}
\newcommand\Rt{\tilde R}
\newcommand\qt{\tilde q}
\newcommand\Qt{\tilde Q}
\newcommand\Xb{\bar X}
\newcommand\lambdat{\tilde\lambda}
\newcommand\betat{\tilde\beta}
\newcommand\Phit{\tilde\Phi}
\newcommand\epst{\tilde\epsilon}
\newcommand\ep{\epsilon}
\newcommand\bt{\tilde b}
\newcommand\Xt{\tilde X}
\newcommand\Mt{\tilde M}
\newcommand\At{\tilde A}
\newcommand\Et{\tilde E}
\newcommand\Ht{\tilde H}
\newcommand\at{\tilde a}
\newcommand\Ct{\tilde C}
\newcommand\pih{\hat\pi}
\newcommand\Rh{\hat R}
\newcommand\Ah{\hat A}
\newcommand\Bh{\hat B}
\newcommand\Ch{\hat C}
\newcommand\Gh{\hat G}
\newcommand\Hh{\hat H}
\newcommand\Qh{\hat Q}
\newcommand\Ph{\hat P}
\newcommand\Nh{\hat N}
\newcommand\Sh{\hat S}
\newcommand\Gcal{{\mathcal G}}
\newcommand\GcalC{{\mathcal G}_C}
\newcommand\Jcal{{\mathcal J}}
\newcommand\JcalC{{\mathcal J}_C}
\setcounter{secnumdepth}{3}
\newtheorem{lemma}{Lemma}[section]
\newtheorem{prop}[lemma]{Proposition}
\newtheorem{thm}[lemma]{Theorem}
\newtheorem{cor}[lemma]{Corollary}
\newtheorem{result}[lemma]{Result}
\newtheorem*{thm*}{Theorem}
\newtheorem*{prop*}{Proposition}
\newtheorem*{cor*}{Corollary}
\newtheorem*{conj*}{Conjecture}
\numberwithin{equation}{section}
\theoremstyle{remark}
\newtheorem{rem}[lemma]{Remark}
\newtheorem*{rem*}{Remark}
\theoremstyle{definition}
\newtheorem{Def}[lemma]{Definition}
\newtheorem*{Def*}{Definition}
\def\signature#1#2{\par\noindent#1\dotfill\null\\*
{\raggedleft #2\par}}

\renewcommand{\theenumi}{\roman{enumi}}
\renewcommand{\labelenumi}{(\theenumi)}

\title[The wave equation on asymptotically
de Sitter-like spaces]{The wave equation on asymptotically\\
de Sitter-like spaces}
\author[Andras Vasy]{Andr\'as Vasy}
\date{June 25, 2007.}
\thanks{The author gratefully acknowledges financial support for this
project from the National Science Foundation under
grants DMS-0201092 and DMS-0733485, from a Clay Research Fellowship and a
Sloan Fellowship.}
\address{Department of Mathematics, Stanford University, Stanford, CA
94305-2125, U.S.A.}
\email{andras@math.stanford.edu}

\begin{abstract}
In this paper we obtain the asymptotic behavior of solutions of the
Klein-Gordon equation on Lorentzian manifolds $(X^\circ,g)$
which are de Sitter-like
at infinity. Such manifolds are Lorentzian analogues of the so-called
Riemannian conformally compact (or asymptotically hyperbolic) spaces. 
Under global assumptions on the (null)bicharacteristic flow,
namely that the boundary of the compactification $X$
is a union of two disjoint manifolds, $Y_\pm$,
and each
bicharacteristic converges to one of these two manifolds as the parameter
along the bicharacteristic goes to $+\infty$, and to the other manifold
as the parameter goes to $-\infty$,
we also define the scattering operator, and show that it is a Fourier
integral operator associated to the bicharacteristic flow from $Y_+$ to
$Y_-$.
\end{abstract}

\maketitle

\section{Introduction}
Consider a de Sitter-like pseudo-Riemannian metric $g$ of signature
$(1,n-1)$ on an $n$-dimensional
($n\geq 2$) manifold
with boundary $X$, with boundary $Y$, which near $Y$ is of the form
\begin{equation*}
g=\frac{dx^2-h}{x^2},
\end{equation*}
$h$ a smooth symmetric 2-cotensor on $X$ such that with respect to some
product decomposition of $X$ near $Y$, $X=Y\times[0,\ep)_x$,
$h|_Y$ is a section of $T^*Y\otimes T^*Y$ (rather than merely
$T^*_Y X\otimes T^*_Y X$) and is a
Riemannian metric on $Y$.
Let the wave operator $\Box$ be the Laplace-Beltrami operator
associated to this metric, and let $P=P(\lambda)
=\Box-\lambda$ be the Klein-Gordon
operator, $\lambda\in\RR$.

Below we consider solutions of $P u=0$. The {\em bicharacteristics} of $P$
over $X^\circ$ are the integral curves of the Hamilton vector field
of the principal symbol $\sigma_2(P)$ (given by the dual metric function)
{\em inside the characteristic set} of
$P$. As $g$ is conformal to $dx^2-h$,
bicharacteristics of $P$ are reparameterizations of bicharacteristics
of $dx^2-h$ (near $Y$, that is). Since $g$ is complete, this means that
the bicharacteristics $\gamma$ of $P$ have limits
$\lim_{t\to\pm\infty}\gamma(t)$ in $S^*_YX$, provided that they approach $Y$.
While many of the results below are local in character, it is simpler to
state a global result, for which we need to assume that
\begin{itemize}
\item[(A1)]
$Y=Y_+\cup Y_-$ with $Y_+$ and $Y_-$ a union of connected components of $Y$
\item[(A2)]
each bicharacteristic $\gamma$ of $P$ converges to $Y_+$ as
$t\to+\infty$ and to $Y_-$ as $t\to-\infty$, or vice versa
\end{itemize}
Due to
the conformality,
the characteristic set $\Sigma(P)$ of $P$ can be identified with
a smooth submanifold of $S^*X$, transversal to $\pa X$, so
$S^*_YX\cap\Sigma(P)$ can be identified with two copies $S^*_\pm Y$
of $S^*Y$, one
for each sign of the dual variable of $x$. Under our assumptions we
thus have a classical scattering map $\cS_{\cl}:S^*_+Y_+\to S^*_-Y_-$.

It is well-known, cf.\ \cite{Geroch:Domain}, that (A1) and (A2) imply the
existence of a global compactified `time' function $T$,
with $T\in\CI(X)$,
$T|_{Y_\pm}=\pm 1$, and the pullback of $T$ to $S^*X$ having positive/negative
derivative along the Hamilton vector field inside the characteristic
set $\Sigma(p)$ depending on whether the corresponding bicharacterstics
tend to $Y_+$ or $Y_-$. Notice that $1-x$ resp.\ $x-1$ has the desired properties near $Y_+$ resp.\ $Y_-$,  so the point is that a function like these
can be extended to all of $X$. Moreover, such a function gives a fibration
$T:X\to[-1,1]$, hence $X$ is in fact diffeomorphic to $[-1,1]\times S$ for
a compact manifold $S$. In particular, $Y_+$ and $Y_-$ are both diffeomorphic
to $S$. Denote the level set $T=t_0$ by $S_{t_0}$.
With any choice of such a function $T$, a constant
$t_0\in(-1,1)$, and a vector field $V$ transversal to $S_{t_0}$ (e.g.\ take
the vector field corresponding to $dT$ under the metric identification of
$TX^\circ$ and $T^*X^\circ$),
$P$ is strictly hyperbolic, and
the Cauchy problem $Pu=0$ in $X^\circ$, $u|_{S_{t_0}}=\psi_0$,
$Vu|_{S_{t_0}}=\psi_1$, $\psi_0,\psi_1\in\Cinf(S_{t_0})$ is well posed.

\begin{thm}(See Theorem~\ref{thm:Cauchy-exist}.)
Let  $s_\pm(\lambda)=\frac{n-1}{2}\pm\sqrt{\frac{(n-1)^2}{4}-\lambda}$.
Assuming (A1) and (A2),
the solution $u$ of the Cauchy problem
has the form
\begin{equation}\label{eq:asymp-exp}
u=x^{s_+(\lambda)}v_++x^{s_-(\lambda)}v_-,\ v_\pm\in\CI(X),
\end{equation}
if $s_+(\lambda)-s_-(\lambda)=2\sqrt{\frac{(n-1)^2}{4}-\lambda}$
is not an integer.
If $s_+(\lambda)-s_-(\lambda)$
is an integer, the same conclusion holds if we replace
$v_-\in\CI(X)$ by $v_-=\CI(X)
+x^{s_+(\lambda)-s_-(\lambda)}\log x\,\CI(X)$.
\end{thm}

Conversely, the asymptotic behavior of $v_\pm$ either at $Y_+$ or at $Y_-$
can be prescribed arbitrarily, see Theorem~\ref{thm:smooth-solns}. Thus,
assuming A1 and A2, if $s_+(\lambda)-s_-(\lambda)$ is not an integer,
we show that given $g_\pm\in\CI(Y_+)$
there exists a unique $u\in\CI(X^\circ)$ such that $Pu=0$
and which is of the form \eqref{eq:asymp-exp}
and such
that
\begin{equation}\label{eq:v+-v--spec}
v_+|_{Y_+}=g_+,\ v_-|_{Y_+}=g_-.
\end{equation}
If $s_+(\lambda)-s_-(\lambda)$
is a non-zero integer, the same conclusion holds if we replace
$v_-\in\CI(X)$ by $v_-=\sum_{j=0}^{s_+(\lambda)-s_-(\lambda)-1}a_jx^j
+x^{s_+(\lambda)-s_-(\lambda)}\log x\,\CI(X)$, $a_j\in\CI(Y)$,
see Theorem~\ref{thm:smooth-solns}. For $\lambda=\frac{(n-1)^2}{4}$,
a similar results holds, with
\begin{equation}\label{eq:asymp-exp-threshold}
u=x^{(n-1)/2}v_++x^{(n-1)/2}\log x\,v_-,\ v_\pm\in\CI(X),\ v_\pm|_{Y_+}=g_\pm.
\end{equation}

That is, for all $\lambda\in\RR$, there
is a unique solution of $Pu=0$ with two pieces of `Cauchy data'
specified at $Y_+$. Note the contrast with the elliptic asymptotically
hyperbolic problem (conformally compact Riemannian metrics): there
one specifies one of the two pieces of the Cauchy data, but over all
of $Y$ (not only at $Y_+$), see \cite{Mazzeo-Melrose:Meromorphic}.
The quantum scattering map
is the map:
\begin{equation*}
\cS:\CI(Y_+)\oplus\CI(Y_+)\to\CI(Y_-)\oplus\CI(Y_-),
\ \cS(g_+,g_-)=(v_+|_{Y_-},v_-|_{Y_-}).
\end{equation*}
Of course, the labelling of $Y_+$ and $Y_-$ can be reversed, so $\cS$ is
invertible. In fact, it is useful to renormalize $\cS=\cS(\lambda)$ somewhat
so that the two pieces of Cauchy data at infinity carry the same
`weight'.
Let $\Delta_h'$ denote the operator which is
$\Delta_h$ on the orthocomplement of the nullspace of $\Delta_h$
and is the identity on the nullspace, so $\Delta_h'$ is positive
and invertible.
Then the renormalization is
\begin{equation*}\begin{split}
&\tilde \cS(\lambda)\\
&=
((\Delta'_h)^{-s_+(\lambda)/2+n/4}\oplus(\Delta'_h)^{-s_-(\lambda)+n/4})
\cS(\lambda)((\Delta'_h)^{s_+(\lambda)/2-n/4}
\oplus(\Delta'_h)^{s_-(\lambda)/2-n/4});
\end{split}\end{equation*}
this is analogous to using $A\psi_0$ in place of $\psi_0$ for the
finite time Cauchy data, where $A\in\Psi^1(S_{t_0})$ elliptic, invertible.
We show that:

\begin{thm} (See Theorem~\ref{thm:S-FIO}.)
Suppose that $s_+(\lambda)-s_-(\lambda)$ is not an integer,
i.e.\ $\lambda\neq\frac{(n-1)^2-m^2}{4}$, $m\in\Nat$.
$\tilde\cS=\tilde\cS(\lambda)$
is an invertible elliptic $0$th order
Fourier integral operator with canonical relation given by
$\cS_{\cl}$, and $\cS$ is a Fourier integral operator.
\end{thm}

\begin{rem}
The somewhat strange powers in the normalization correspond to making
the map from Cauchy data at infinity to Cauchy data at time
$t_0\in(-1,1)$ a FIO of order $0$; see Proposition~\ref{prop:S+ep-FIO}.
\end{rem}

Note that the canonical relation is independent of $\lambda$. While
our parametrix construction for $\cS(\lambda)$ does not work apparently
if  $s_+(\lambda)-s_-(\lambda)$ is an integer due to the possible
non-solvability of a model problem with the prescribed ansatz,
it is expected that with
more detailed analysis (changing the ansatz slightly to allow logarithmic
terms in $x$)
one can prove the theorem in this case as well.
Moreover, we actually construct a parametrix for the solution operator
$(g_+,g_-)\mapsto u$, and even if $s_+(\lambda)-s_-(\lambda)$ is an integer,
the part of the operator corresponding to $g_+$ (i.e.\ with $g_-=0$)
can be constructed as a Fourier integral operator.

In addition, if
$g$ is {\em even}, i.e.\ there is a boundary defining function $x$ such that
only even powers of $x$ appear in the Taylor series of $g$ at $\pa X$ expressed
in geodesic normal coordinates,
see \cite{Guillarmou:Meromorphic} for the Riemannian case,
then the $\log x$ terms in $v_-$ disappear and our parametrix construction
for $\cS(\lambda)$ goes through provided that $s_+(\lambda)-s_-(\lambda)$
is odd. In particular, this covers the actual d'Alembertian ($\lambda=0$)
if $n$ is even.

For the Cauchy problem, we similarly have:

\begin{thm}\label{thm:Cauchy-FIO}
For $t_0\in(-1,1)$ and for
all $(\psi_0,\psi_1)\in\CI(S_{t_0})^2$, let $u\in\CI(X^\circ)$
denote the unique solution
of the Cauchy problem
$Pu=0$ in $X^\circ$,
$u|_{S_{t_0}}=\psi_0$, $V u|_{S_{t_0}}=\psi_1$.
This solution $u$ has asymptotic
expansion as in \eqref{eq:asymp-exp}. If
$\lambda\neq\frac{(n-1)^2-m^2}{4}$, $m\in\Nat$, the operators
\begin{equation*}
(\psi_0,\psi_1)\mapsto (v_+|_{Y_+},v_-|_{Y_+})\Mand
(\psi_0,\psi_1)\mapsto (v_+|_{Y_-},v_-|_{Y_-})
\end{equation*}
are both Fourier
integral operators associated to the bicharacteristic flow.
\end{thm}

To justify our terminology of asymptotically de Sitter spaces,
we recall that de Sitter space is given
by the hyperboloid $z_1^2+\ldots+z_n^2=z_{n+1}^2+1$ in $\Real^{n+1}$
equipped with the pull-back of the Lorentzian metric
$dz_{n+1}^2-dz_1^2-\ldots-dz_n^2$. Introducing polar coordinates $(r,\theta)$
in the first $n$ variables and writing $t=z_{n+1}$, the hyperboloid
can be identified with $\Real_t\times\sphere^{n-1}_\theta$ with the
Lorentzian metric
\begin{equation*}
\frac{dt^2}{t^2+1}-(t^2+1)\,d\theta^2,
\end{equation*}
with $d\theta^2$
being the standard Riemannian metric on the sphere.
For $t>1$, say, we let $x=t^{-1}$, and note that the metric becomes
$\frac{(1+x^2)^{-1}\,dx^2-(1+x^2)\,d\theta^2}{x^2}$, which is of
the required form. An analogous formula holds for $t<-1$, so compactifying
the real line as an interval $[-1,1]_s$ (with $s=1-x$
for $x<\frac{1}{2}$, say), we see that de Sitter space indeed fits into
our framework. (Thus, one can take $T=s$ for the global compactified
time function.) We also note that another, perhaps more familiar, form
of the metric can be obtained by
letting $t=\sinh\rho$; the metric becomes $d\rho^2-\cosh^2\rho\,d\theta^2$.
(One can take e.g.\ $T=\tanh\rho$ here.)

\begin{figure}[ht]
\begin{center}
\mbox{\epsfig{file=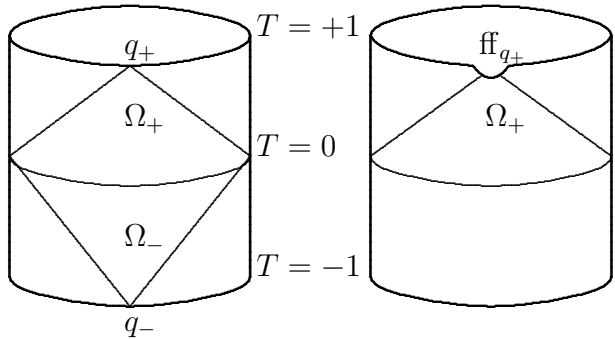}}
\end{center}
\caption{On the left, the compactification of de Sitter space with the
backward light cone from $q_+$ and forward light cone from $q_-$ are
shown. $\Omega_+$, resp.\ $\Omega_-$, denotes the intersection of these
light cones with $T>0$, resp.\ $T<0$.
On the right, the blow up of de Sitter space at $q_+$ is shown. The
interior of the light cone inside the front face $\ff_{q_+}$ can
be identified with the spatial part of the static model of de Sitter
space.}
\label{fig:deSitter-space}
\end{figure}

We also use this occasion to explain the connection with the static model
of de Sitter space. This corresponds to singling out a point on
$\sphere^{n-1}_{\theta}$,
e.g.\ $q_0=(1,0,\ldots,0)\in\sphere^{n-1}\subset\RR^n$. The static model
of de Sitter space then is the intersection of the backward lightcone
from $q_0$ considered as a point $q_+$ on $Y_+$ (so $T(q_+)=1$) and the forward
light cone from $q_0$ considered as a point $q_-$ on $Y_-$
(so $T(q_-)=-1$). These happen to intersect the equator $T=0$ (here $t=0$)
in the same set, and altogether form a `diamond', see
Figure~\ref{fig:deSitter-space}. Explicitly this
region is given by $z_2^2+\ldots+z_n^2\leq 1$ inside the
hyperboloid. The standard static
coordinates $(\tau,r,\omega)$ on the `diamond' are given by
\begin{equation*}\begin{split}
r&=\sqrt{z_2^2+\ldots+z_n^2}=\sqrt{1+z_{n+1}^2-z_1^2},\\
\sinh\tau&=\frac{z_{n+1}}{\sqrt{z_1^2-z_{n+1}^2}},\\
\omega&=r^{-1}(z_2,\ldots,z_n)\in\sphere^{n-2}.
\end{split}\end{equation*}
In these coordinates the
metric becomes $(1-r^2)\,d\tau^2-(1-r^2)^{-1}dr^2-r^2\,d\omega^2$. Note that
the singularity at $r=1$ is completely artificial (is due to the
coordinates), the metric is incomplete, but is conformal to a complete
Lorentzian metric, of product type, with $\Box$ also of product type.
While one {\em can} analyze the solutions of the wave equations on de Sitter
space at points inside the `diamond' by considering the diamond only
(in view of the finite propagation speed for the wave equation), the
resulting picture does include rather artificial limitations. For instance,
the asymptotics at the sides of the diamond are automatically smooth
in de Sitter space (as we have a standard wave equation there), which
is not obvious if one's world consists of the diamond, and the local
static asymptotics, corresponding to the tip of the diamond at $Y_+$,
describes only a small part of the asymptotics of solutions of the Cauchy
problem on de Sitter space. However, the `spatial' part of the static
operator (or modifications of it) do show up in our analysis as models
for the Poisson operator $(g_+,g_-)\mapsto u$; the proper place
for its existence is on the interior of the light cone in
the blow up of the distinguished point $q_+$
in de Sitter space.

It should be pointed out that the de Sitter-Schwarzschild metric in fact
has many similar features, and the analogous result is the subject
of an ongoing project with Ant\^onio S\'a Barreto and Richard Melrose.
Weaker results on the asymptotics in that case
are contained in the part of works of Dafermos and
Rodnianski concerned with the underlying linear problem
\cite{Dafermos-Rodnianski:Price} (they study a non-linear problem),
and local energy decay was studied by Bony and H\"afner
\cite{Bony-Haefner:Decay}, in part based on the stationary resonance
analysis of S\'a Barreto and Zworski \cite{Sa-Barreto-Zworski:Distribution}.

We also note that on de Sitter space itself, one can solve the
wave equation explicitly, see \cite{Polarski:Hawking}, but
even the `smooth asymptotics' result,
Theorem~\ref{thm:Cauchy-exist}, is not apparent from such a solution.

There are two rather different techniques used to prove the results
here. The `rough' results yielding the existence of the asymptotics,
Theorems~\ref{thm:smooth-solns} and \ref{thm:Cauchy-exist},
are proved using positive commutator estimates, which roughly speaking
describe the microlocal (i.e.\ phase space) propagation of $L^2$ (or
Sobolev) mass (`energy'). Such methods are very
robust, but (unless they are used in a more sophisticated form as
in \cite{Hassell-Melrose-Vasy:Scattering})
give less precise results. The Fourier integral operator
results are proved by a parametrix construction which is significantly
more delicate (taking up two-fifth of this paper),
but is very instructive. It is at this stage that
the static de Sitter model shows up on the front face of
$[X\times Y_+;\diag_{Y_+}]$; see $P_\sigma$ in Section~\ref{sec:scattering}.
One should think of this as analogous to the way the hyperbolic
Laplacian shows up as a model on the front face of the 0-double space for
conformally compact Riemannian manifolds,
see \cite{Mazzeo-Melrose:Meromorphic}.

The plan of the paper is the following. In Section~\ref{sec:0-geom} we
adopt a 0-microlocal point of view, and analyze propagation of
singularities in the 0-cotangent bundle introduced by Mazzeo and
Melrose \cite{Mazzeo-Melrose:Meromorphic} two decades ago. The proof
uses positive commutator estimates, and
is quite similar to propagation of singularities for manifolds with
boundary equipped with a so-called (incomplete) edge metric,
which includes e.g.\ manifolds with conic points -- see
\cite{Melrose-Wunsch:Propagation} and
\cite{Melrose-Vasy-Wunsch:Propagation} and references therein.
In the following sections we analyze local solvability near the boundary
as well as conormal regularity of the solutions there. {\em We
emphasize that the results of Sections~\ref{sec:0-geom}-\ref{sec:conormal}
do not need the global assumptions (A1)-(A2).}
In Section~\ref{sec:global} we prove a unique continuation theorem
at $\pa X$ (i.e.\ at `infinity') by a Carleman-type estimate, and use
it to prove that the asymptotic behavior of the solutions in fact
determines the solutions, i.e.\ we can talk about a `Cauchy problem at
infinity', hence also about the scattering map. In the final section
we construct a parametrix for the scattering map, and use it to show
that it is indeed a Fourier integral operator.

I am very grateful for Rafe Mazzeo, Richard Melrose, Ant\^onio
S\'a Barreto and Maciej Zworski
for numerous fruitful discussions. In particular, I thank Richard
Melrose for pointing out that the assumptions (A1) and (A2) imply the existence
of a global time foliation, while relating the analysis here to the
static de Sitter model arose from discussions with Maciej Zworski.

\section{0-geometry and propagation of 0-singularities}\label{sec:0-geom}

For the purposes of analysis, we need a good understanding
of bicharacteristic geometry.
Thus, note
that $P\in\Diff^2_0(X)$, in the zero-calculus of Mazzeo and Melrose
\cite{Mazzeo-Melrose:Meromorphic}. Let $\zT^* X$ denote the zero-cotangent
bundle of $X$. Its elements are covectors of the form $\xi\,\frac{dx}{x}
+\eta\,\frac{dy}{x}$. Then the principal symbol
$p=\sigma(P)$ is a homogeneous degree $2$
polynomial on $\zT^*X$; explicitly at $Y$,
$p|_Y=\xi^2-H|_Y$, $H|_Y$ the metric function
corresponding to $h$, and $p$ itself is the metric function of the dual
pseudo-Riemannian metric $g$. We refer to \cite{Mazzeo-Melrose:Meromorphic,
Sa-Barreto-Zworski:Distribution} for nice descriptions of the basic
setup, and \cite{Melrose-Wunsch:Propagation,Melrose-Vasy-Wunsch:Propagation}
for analysis of a hyperbolic problem in the related edge setting.

If $a$ is a homogeneous function on $\zT^* X\setminus o$, then there
is a (homogeneous)
Hamilton vector field $H_a$ associated to it on $T^*X^\circ\setminus o$. A
change of coordinates calculation shows that in the
0-canonical coordinates given above
\begin{equation*}
H_a=(\pa_\xi a)
(\xi\pa_\xi+\eta\pa_\eta+x\pa_x)+x(\pa_\eta a\pa_y-\pa_y a\pa_\eta)
-((x\pa_x+\xi\pa_\xi+\eta\pa_\eta)a)\pa_\xi,
\end{equation*}
so $H_a$ in fact extends to a $\CI$ vector field on $\zT^* X\setminus o$
which is tangent to $\zT^*_{\pa X}X$.
At $x=0$ this gives $H_a=(\pa_\xi a) R-(Ra)\pa_\xi$, where $R$ is the radial
vector field $\xi\pa_\xi+\eta\pa_\eta$ on $\zT^*X$. Since $a$ is homogeneous
of degree, say, $k$, $Ra=ka$, and $H_a=(\pa_\xi)a R-ka\pa_\xi$, so
on the characteristic set $\Sigma(a)=a^{-1}(\{0\})$ of $a$, at $x=0$,
$H_a$ is radial.
It is thus rather convenient to consider the cosphere bundle $\zS^*X$
which is the boundary at fiber infinity of the fiber radial compactification
$\zcT^*X$ of $\zT^*X$.

As we work with $p$, so that near $Y$, $\xi\neq 0$ on the characteristic set,
we use projective coordinates $\etah=\eta/|\xi|$, $\rho=|\xi|^{-1}$
valid near $\Sigma(p)$.
Then
\begin{equation*}\begin{split}
(\sign\xi)\rho^{-1}H_a=&-\left((\rho\pa_\rho+\etah\pa_{\etah})a\right)
(-\rho\pa_\rho+x\pa_x)
+x(\pa_{\etah}a\pa_y-\pa_ya\pa_{\etah})\\
&\qquad\qquad+\left((x\pa_x-\rho\pa_\rho)a\right)
(\rho\pa_\rho+\etah\pa_{\etah}).
\end{split}\end{equation*}
Thus, for $a\in\rho^{-k}\CI(\zcT^*X)$, $W_a=\rho^{k-1}H_a$ is a smooth
vector field on $\zcT^*X$, whose restriction to $\zS^*_Y X$ is
$(\sign\xi)ka\etah\pa_{\etah}$, i.e.\ it vanishes at $a=0$.
Thus, if $da$ is not
conormal to $\zS^*_YX$ in $\zS^*X$, so $\Sigma(a)$ is
transversal to $\zS^*_YX$, then $W_a$ is a smooth vector field on $\Sigma(a)$
that vanishes at $x=0$, and hence is of the form $W_a=xW'_a$, $W'_a\in
\Vf(\Sigma(a))$.

Applying this with $a=p=\rho^{-2}\CI(\zcT^*X)$ yields that inside $\Sigma(p)$,
$W_p=xW'_p$,
\begin{equation*}
W'_p|_{x=0}=(\sign\xi)(2\pa_x+H_h),
\end{equation*}
$H_h$ the Hamilton vector field
of $h$ (evaluated at $(y,\etah)$). In particular, $W'_p$ is transversal
to $Y$. Also, $W_p$ is complete, and $\gamma$ is an integral curve of
$W_p$, then a reparameterized version of $\gamma$ is an integral
curve of $W'_p$, hence $\lim_{t\to -(\sign\xi)\infty}\gamma(t)$ exists in
$\zS^*_YX\cap\Sigma(p)$.
(Note
that the map switching the sign of covectors preserves even functions,
such as $p$, while transforms $H_p$ to $-H_p$.) Conversely,
for any $q\in \zS^*_YX\cap\Sigma(p)$
there is a unique (up to translation of the
parameterization) integral curve of $W_p$
with limit $q$ as $t\to-(\sign\xi)\infty$, namely this is just a
reparameterization of the unique integral curve of $W'_p$ through $q$.
Note also that $\zS^*_YX\cap\Sigma(p)$ can be identified with two
copies of $S^*Y$, one
for each sign of $\xi$; we write these as $S^*_+Y$ and $S^*_-Y$.

Suppose now
that $Y=Y_+\cup Y_-$, where $Y_\pm$ are unions of connected
components of $Y$, and this decomposition satisfies that all bicharacteristics
$t\mapsto \gamma(t)$ of $P$ satisfy $\lim_{t\to +\infty}\gamma(t)\in S^*Y_+$,
$\lim_{t\to -\infty}\gamma(t)\in S^*Y_-$, or vice versa, i.e.\ that (A1) and
(A2) of the introduction hold.
For $q\in S^*_+Y_+$ there is a unique
bicharacteristic with $\lim_{t\to-\infty}\gamma(t)=q$. By
(A1) and (A2), $\lim_{t\to+\infty}\gamma(t)=q'\in S^*Y_-$
exists; as we saw above, it necessarily lies in $S^*_-Y_-$. The classical
scattering map is the map $\cS_{\cl}:S^*Y_+\to S^* Y_-$ with
$\cS_{\cl}:q\mapsto q'$.
Fixing a homogeneous degree $1$ function on $T^*Y\setminus o$,
we can extend these to
maps $T^*Y_+\setminus o\to T^*Y_-\setminus o$ -- we can use $h^{1/2}$,
for instance. The induced relation on $(T^*Y_+\setminus o) \times
(T^*Y_-\setminus o)$ is Lagrangian with respect to the twisted symplectic
form (i.e.\ with a negative sign on one of the factors).

As follows easily from the results of \cite{Geroch:Domain},
(A1) and (A2) imply the
existence of a global compactified `time' function $T$,
with $T\in\CI(X)$,
$T|_{Y_\pm}=\pm 1$, and the pullback $\pi^*T$
of $T$ to $S^*X$ having positive/negative
derivative along the Hamilton vector field inside the characteristic
set $\Sigma(p)$ depending on whether the corresponding bicharacterstics
tend to $Y_+$ or $Y_-$. Notice that $1-x$ resp.\ $x-1$ has the desired properties near $Y_+$ resp.\ $Y_-$, so the point is the interior of $X$
these can be extended to all of $X$.

With any choice of such a function $T$, a constant
$t_0\in(-1,1)$, and a vector field $V$ transversal to $S_{t_0}$
$P$ is strictly hyperbolic, and
the Cauchy problem $Pu=0$ in $X^\circ$, $u|_{S_{t_0}}=\psi_0$,
$Vu|_{S_{t_0}}=\psi_1$, $\psi_0,\psi_1\in\Cinf(S_{t_0})$ is well posed.

Our first result is that of $0$-regularity of solutions of $Pu=0$ with a weight
given by a space $u$ a priori lies in. There is a dichotomy
between solutions depending on the a priori regularity relative to
this weighted space. If the a priori regularity is low, we only
obtain regularity up to a limit implied by the weight, but we do so
without having to assume any interior regularity for $u$. If the
a priori regularity is high, then we obtain additional regularity up to the
limit corresponding to the smoothness of $u$ in $X^\circ$.

\begin{prop}\label{prop:edge-reg}
Suppose that $q\in Y$, and suppose that $u$ is in
$H_0^{r_0,s_0}(X)$ in a neighborhood of $q$ and $Pu=0$. Then:
\begin{enumerate}
\item
If $r_0<s_0+1/2$ then $u$ is in $H_0^{r,s_0}(X)$ near $q$ for all
$r<s_0+1/2$.
\item
If $r_0>s_0+1/2$ and $r>r_0$, $\alpha\in \zS^*_q X\cap\Sigma(p)$
then $\alpha\nin\zWF^{r,s_0}(u)$
provided that the bicharacteristic $\gamma$ approaching $\alpha$
is disjoint from $\WF^r(u)\subset S^*X^\circ$. The same conclusion
holds if $r_0\leq s_0+1/2$, but $\alpha\nin\zWF^{r_1,s_0}(u)$
for some $r_1>s_0+1/2$.
\item
In particular, if $r_0>s_0+1/2$ and $r>r_0$, then
$u$ is in $H_0^{r,s_0}(X)$ near $q$ provided that all bicharacteristics
approaching $\zS^*_q X$ are disjoint from $\WF^r(u)\subset S^*X^\circ$.
\end{enumerate}
\end{prop}

\begin{proof}
This proof is very similar to the proofs of propagation of `edge regularity'
for the wave equation with {\em incomplete} metrics in
\cite{Melrose-Wunsch:Propagation} and
\cite{Melrose-Vasy-Wunsch:Propagation}, so we shall be brief.

While $\rho H_p$ restricts to a smooth vector field
on $\Sigma(p)$ with vanishing restriction
at $Y$, if we evaluate $\rho H_p$ as a section of the b-tangent bundle
of $\zcT^*X$ at $\Sigma(p)\cap \zS^*_YX$, the result is more interesting:
$\rho H_p=2(-\rho\pa_\rho+x\pa_x)$ in this sense. Correspondingly,
for $A\in\zPs^{m,l}(X)$, the symbol of $i[P,A]\in\zPs^{m+1,l}(X)$
is $H_p a=2(m+l)\rho^{-1}a$,
$a=\sigma(A)$, at $\Sigma(p)\cap \zS^*_YX$. Thus, much as
\cite{Melrose-Wunsch:Propagation} and
\cite{Melrose-Vasy-Wunsch:Propagation},
one can show propagation of zero-regularity
into the boundary for $m+l\neq 0$. Unlike in the setting of
\cite{Melrose-Wunsch:Propagation}, the characteristic set of $P$ only
intersects the boundary $Y$ in radial points, i.e.\ there is no propagation
inside $Y$, which explains why there is no requirement for $m+l$ having
a particular sign (as long as it is non-zero), although the results
are different depending on the sign: (i) has no wave front set assumptions
on $u$. This corresponds to the presence of a
cutoff $\chi$, identically $1$ near $Y$, such that $\pa_x\chi\leq 0$,
the sign of the commutator with $\chi$ agrees with the sign arising
from the weights if $m+l<0$. Moreover, one can microlocalize in $S^*_Y X$
by pulling back functions from $S^*_Y X\cap\Sigma(p)$ using the
flow of $W'_p$, extending them to a neighborhood of the characteristic
set in an arbitrary smooth fashion.

Thus, let $\psi_0\in\CI(S^*_Y X\cap\Sigma(p))$, and
for any integral curve $\tilde \gamma$ of $W'_p$ with
$\tilde\gamma(0)\in S^*_Y\cap\Sigma(p))$,
we let $\psi(\tilde\gamma(t))=\psi_0(\gamma(0))$. Note that this
defines a $\CI$ function on $\Sigma(p)$ near $Y$, for
the map $\Phi: S^*_YX\cap\Sigma(p)\times[0,\ep)\to\Sigma(p)$
given by $\Phi(q,t)=\exp(tW'_p)q$ is a local diffeomorphism near $t=0$.
As $\Sigma(p)$ is a $\CI$ submanifold of $S^* X$, we can extend
$\psi$ to a $\CI$ function on $S^*X$, still denoted by $\psi$, hence
further to an element of $\CI(\zcT^*X)$, at least near $Y$. Now let
$\chi\in\CI_c([0,\ep))$ be such that $\chi'=-\chi_0^2$, $\chi\equiv 1$ near
$0$, $\chi\geq 0$, $\chi^{1/2}$ is $\CI$, and let
$a=\rho^{-m}x^l \chi(x) \psi$, and note that $W_p \chi(x)=b^2 x\chi'(x)$ with
$b>0$ near $Y$. As $W'_p \psi$ vanishes at $p=0$,
we deduce that
\begin{equation*}
H_p a=2(m+l)\rho^{-m-1} x^l \chi(x) \psi+2\rho^{-m-1}x^{l+1}b^2\chi'(x)
+p \rho^{-m+1} x^l e+\rho^{-m} x^l f,
\end{equation*}
with $b,e,f\in\CI(\zcT^*X)$.

Now the standard positive commutator argument finishes the proof of the
proposition, see e.g.\ \cite{Melrose-Wunsch:Propagation}.
For the reader's convenience, we sketch the argument, skipping the
(necessary but straightforward) regularization part of the argument.
Thus, $\sigma_{-m-1}(i[P,A])=H_p a$ shows that
\begin{equation*}\begin{split}
&i[P,A]=2(m+l)\Lambda\tilde A^*\tilde A\Lambda-B^*B+PE+F,\\
&\tilde A\in\zPs^{m/2,l/2}(X),\ \sigma(\tilde A)=\sigma(A)^{1/2},\\
&B\in\zPs^{(m+1)/2,l/2}(X),\ \WF'(B)\subset \supp\chi_0\cap\supp\psi,
\ \sigma(B)=b\chi_0(2\rho^{-m-1}x^{l+1})^{1/2}\\
&E\in\zPs^{m-1,l}(X),\ F\in\zPs^{m,l}(X),
\end{split}\end{equation*}
$\Lambda\in\zPs^{1/2,0}(X)$
elliptic formally self-adjoint with positive principal symbol,
$\rho^{-1/2}$. Proceeding as in
\cite{Melrose-Wunsch:Propagation} shows that for $u$ with $Pu=0$,
\begin{equation*}
|\pm \|\tilde A \Lambda u\|_{H_0^{0,0}(X)}^2-\|Bu\|_{H_0^{0,0}(X)}^2|\leq
C\|u\|_{H_0^{m/2,-l/2}(X)}^2,
\end{equation*}
provided that the right hand side is finite,
with the $-$ sign applying if $m+l<0$, and the $+$ sign applying if $m+l>0$.
In the first case, the second term on the left hand side can simply
be dropped, so we do not need to make any assumptions on the $H^{(m+1)/2}$
norm of $u$,
while in the second case we need to assume that $\WF^{(m+1)/2}(u)$
is disjoint from $\supp\chi_0$, in order to conclude that
$\|\tilde A\Lambda u\|_{H_0^{0,0}(X)}$ is finite,
i.e.\ $\zWF^{(m+1)/2,-l/2}(u)$ is disjoint from the elliptic set of $A$,
i.e.\ from the interior of $\supp\psi$ near $x=0$.

The standard iteration argument now proves the proposition.
\end{proof}

The approximation process prevents us from crossing the line $r=s_0+1/2$, which
is why we cannot directly
obtain information about $u$ in $H_0^{r,s_0}(X)$ with
$r>s_0+1/2$ unless we know $u$ is in $H_0^{r_0,s_0}(X)$ for $r_0>s_0+1/2$.
However, if $u\in H^{r_0,s_0}(X)$ with $r_0=s_0+1/2-\ep/2$,
so $r_0<s_0+1/2$, then
$u\in H^{r_0,s_0-\ep}(X)$, and $r_0>(s_0-\ep)+1/2$ now. We thus deduce:

\begin{cor}\label{cor:edge-reg}
Suppose that $q\in Y$, and suppose that $u$ is in
$H_0^{r_0,s_0}(X)$ in a neighborhood of $q$ and $Pu=0$.
If $r>r_0$ and $s<s_0$, $\alpha\in \zS^*_q X\cap\Sigma(p)$
then $\alpha\nin\zWF^{r,s}(u)$
provided that the bicharacteristic $\gamma$ approaching $\alpha$
is disjoint from $\WF^r(u)\subset S^*X^\circ$.

In particular, $u$ is in $H_0^{r,s}(X)$ near
$q$ provided that all bicharacteristics
approaching $\zS^*_q X$ are disjoint from $\WF^r(u)\subset S^*X^\circ$.
\end{cor}

\begin{rem}
Thus, we gain {\em full} 0-regularity for solutions if we are willing
to give up some (arbitrarily little) decay. Note that (ii) of
the Proposition states that one can take $s=s_0$ if $r_0>s_0+1/2$,
so the present corollary is only interesting if $r_0\leq s_0+1/2$.
\end{rem}

\begin{proof}
Let $s<s_0$ be given, and let $\ep=s_0-s>0$.
As remarked, we may assume $r_0\leq s_0+1/2$, and if needed, we can
decrease $r_0$ so that $r_0<s_0+1/2$.
By (i) of Proposition~\ref{prop:edge-reg}, $\alpha\nin\zWF^{r,s_0}(u)$ for
all $r<s_0+1/2$. Then $\alpha\nin\zWF^{s_0+1/2-\ep/2,s_0}(u)$,
and hence $\alpha\nin\zWF^{s_0+1/2-\ep/2,s_0-\ep}(u)$.
By (ii) of Proposition~\ref{prop:edge-reg}, $\alpha\nin\zWF^{r,s_0-\ep}(u)
=\zWF^{r,s}(u)$
for all $r$, proving the corollary.
\end{proof}

\section{Local solvability near $\pa X$}\label{sec:local-solvability}
In this section we show the solvability of $Pu=0$ near $\pa X$
in suitable senses, $P=\Box-\lambda$. This
relies on a positive commutator estimate with {\em compact} error term,
so we need to control the normal operator of our commutator in the
0-calculus. Recall from \cite{Mazzeo-Melrose:Meromorphic} that
the normal operator map on $\Diff^k_0(X)$ (or $\zPs^k(X)$)
captures $Q\in\Diff^k_0(X)$ modulo $x\Diff^k_0(X)$, as opposed
to the principal symbol map, which captures it modulo $\Diff^{k-1}_0(X)$.
The compactness referred to above then is that of the inclusion map
for the associated Sobolev spaces, $H_0^{r,s}(X)$ to $H_0^{r',s'}(X)$,
with $r>r'$, $s>s'$; note that compactness requires improvements in
both the regularity and decay orders, hence control of both the
principal symbols (described in the previous section) and normal operators.

We thus start by calculating the normal operator of $P$, as well as that
of its commutator with another operator $A$.
Thus, we calculate the the commutator modulo terms with an additional
order of vanishing. As $P\in \Diff^2_0(X)$, and our commutant
will be an operator $A_r\in x^{r-1}\Diff^1_0(X)$, $[P,A_r]
\in x^{r-1}\Diff^2_0(X)$, so we need to compute $[P,A_r]$ modulo
$x^r\Diff^2_0(X)$. This is computation is thus unaffected if $P$ is
changed by addition of
a term in $x\Diff^2_0(X)$, or $A_r$ is changed by
a term in $x^r\Diff^1_0(X)$. This means that effectively we may
assume that $X$ has a product decomposition near $Y$
and $h$ is actually a Riemannian metric on $Y$.
The wave operator is the Laplace-Beltrami operator associated to this metric:
\begin{equation*}
\Box=(xD_x)^2+i(n-1)(xD_x)-x^2\Delta_Y=(xD_x)^*(xD_x)-x^2\Delta_Y,
\end{equation*}
with the adjoint taken with respect to the pseudo-Riemannian density
$x^{-n}\,|dx\,dy|$.

We remark here that the actual normal operator in the 0-calculus
(which results from restricting the Schwartz kernels to the 0-front face)
is even simpler than this model, for it localizes in $Y$. Thus, one
could simply compute with the Euclidean Laplacian in $Y$, but as this
has absolutely no impact on our considerations, we use our more
global model.

We let
$A_r=x^r D_x+i\frac{n-r}{2}\,x^{r-1}$, which is symmetric, and compute
\begin{equation*}\begin{split}
[P,A]&=[(xD_x)^2+i(n-1)(xD_x),x^r D_x+i\frac{n-r}{2}\,x^{r-1}]\\
&\qquad\qquad-[x^2,x^r D_x+i\frac{n-r}{2}\,x^{r-1}]\Delta_Y\\
&=-2i\left\{(r-1)(xD_x+i\frac{n-r}{2})^*x^{r-1}(xD_x+i\frac{n-r}{2})
+x^{r+1}\Delta_Y\right\}.
\end{split}\end{equation*}
Thus, up to the factor $-2i$, this is clearly a positive operator for
$r\geq 1$. We would like to improve this statement, and in particular
show that this is greater than $Cx^{r-1}$ for suitable $C$, at least
in a range of $r$, and at least modulo terms of the form $P B+B^*P$.

The flexibility we have here in arranging this positivity
is the choice of the coefficient $B$ of $P$. Thus, we convert
part of the tangential Laplacian term, $x^{r+1}\Delta_Y$ into $P$
by writing $x^{r+1}\Delta_Y=\gamma x^{r+1}\Delta_Y+(1-\gamma) x^{r+1}\Delta_Y$,
with $\gamma$ to be determined, and
writing
\begin{equation*}
x^{r+1}\Delta_Y=\frac{1}{2}\{x^{r-1}((xD_x)^*(xD_x)-\lambda-P)+
((xD_x)^*(xD_x)-\lambda-P)x^{r-1}\}
\end{equation*}
in the first term. We deduce with $B=-\frac{\gamma}{2}\,x^{r-1}$,
\begin{equation*}\begin{split}
\frac{i}{2}[P,A]
=&(r-1)(xD_x+i\frac{n-r}{2})^*x^{r-1}(xD_x+i\frac{n-r}{2})
+(1-\gamma)x^{r+1}\Delta_Y\\
&+\frac{\gamma}{2}x^{r-1}(xD_x)^*(xD_x)+\frac{\gamma}{2}(xD_x)^*(xD_x)x^{r-1}
-\gamma\lambda x^{r-1}
+PB+B^*P.
\end{split}\end{equation*}
Now, the form of the first term is quite convenient to us in
view of the factor $x^{r-1}$, corresponding to a weighted estimate
on $x^{-(r-1)/2}L^2$ relative to $x^{-n}\,dx$, since its null-space
consists of $x^{(n-r)/2}$, which just misses being in $x^{-(r-1)/2}L^2$
(i.e.\ is in $x^{-(r-1)/2-\delta}L^2$ for all $\delta>0$), so it will
give us optimal zeroth order terms below, and saves us having to use
that for all $s$,
\begin{equation}\label{eq:optimal-xD_x-est}
\frac{(2s-n-1)^2}{4}\|x^{s-1}u\|^2\leq \|x^s D_x u\|^2.
\end{equation}
Note, however, that the first term
can easily be written in a simpler looking form,
\begin{equation*}
(xD_x+i\frac{n-r}{2})^*x^{r-1}(xD_x+i\frac{n-r}{2})=(xD_x)^*x^{r-1}(xD_x)
-\frac{(n-r)^2}{4}x^{r-1}.
\end{equation*}
This can be checked easily as the two sides have the same principal
symbol, so their difference is first order, moreover both sides are real
and self-adjoint, hence actually zeroth order, i.e.\ multiplication by
a smooth function. Their equality can be checked by evaluating
them on $1$.
Moreover, a similar calculation yields
\begin{equation*}\begin{split}
&\frac{1}{2}(x^{r-1}(xD_x)^*(xD_x)+(xD_x)^*(xD_x)x^{r-1})\\
&\qquad=
(xD_x+i\frac{n-r}{2})^*x^{r-1}(xD_x+i\frac{n-r}{2})+\frac{(n+r-2)(n-r)}{4}x^{r-1}.
\end{split}\end{equation*}

Thus,
\begin{equation}\begin{split}\label{eq:comm-form}
\frac{i}{2}\,[P,A]=&(r-1+\gamma)(xD_x+i\frac{n-r}{2})^*x^{r-1}(xD_x+i\frac{n-r}{2}))
+(1-\gamma)x^{r+1}\Delta_Y\\
&\qquad+\gamma \left(\frac{(n-r)(n+r-2)}{4}-\lambda\right)x^{r-1}
+PB+B^*P.
\end{split}\end{equation}
In order to obtain a `positive commutator', modulo the terms involving $P$,
we thus need that
\begin{equation}\label{eq:sign-list}
r-1+\gamma,\ 1-\gamma\ \Mand\ \gamma \left(\frac{(n-r)(n+r-2)}{4}-\lambda\right)
\end{equation}
As $\frac{(n-r)(n+r-2)}{4}-\lambda=0$ gives
\begin{equation*}
\frac{r-1}{2}=\pm \sqrt{\left(\frac{n-1}{2}\right)^2-\lambda},
\end{equation*}
we introduce
\begin{equation}\label{eq:l(lambda)-def}
l(\lambda)=\re\sqrt{\left(\frac{n-1}{2}\right)^2-\lambda},
\end{equation}
so $l(\lambda)=0$ for $\lambda\geq\frac{(n-1)^2}{4}$, $l(\lambda)>0$
for $\lambda<\frac{(n-1)^2}{4}$.

\begin{lemma}\label{lemma:same-signs}
The quantities listed in \eqref{eq:sign-list} have the same (non-zero)
sign if:
\begin{itemize}
\item
if $r>\max(0,1-2l(\lambda))$, $r\neq 1+2l(\lambda)$, in which
case they are all positive, or
\item
if $r<\min(0,1-2l(\lambda))$, in which case they are all negative.
\end{itemize}
\end{lemma}

\begin{proof}
First, note that for $\frac{r-1}{2}\in(-l(\lambda),l(\lambda))$,
i.e.\ $r\in(1-2l(\lambda),1+2l(\lambda))$,
$\frac{(n-r)(n+r-2)}{4}-\lambda>0$, while for
$\frac{r-1}{2}\nin[-l(\lambda),l(\lambda)]$,
$\frac{(n-r)(n+r-2)}{4}-\lambda<0$.

For $r>1$, $r\neq 1+2l(\lambda)$ it is easy to arrange
that all three quantities in \eqref{eq:sign-list} have the same sign
since the first two terms are positive if $|\gamma|$ is sufficiently
small, so choosing the sign of $\gamma$ correctly, the last term can
also be made positive as long as $r\neq 1+2l(\lambda)$
($r>1$ rules out $r=1-2l(\lambda)$).

In general, the first two terms have the same sign if $\gamma\in(1,1-r)$,
resp.\ $\gamma\in(1-r,1)$, depending on whether $r<0$, resp.\ $r>0$, and
this sign is negative, resp.\ positive in the two cases.

Suppose first that $\lambda\leq\frac{(n-1)^2}{4}$.

If $r<0$, we have $\gamma>1$ by the previous remark,
so we need $(n+r-2)(n-r)-\lambda<0$,
i.e.\ $r\nin[1-2l(\lambda),1+2l(\lambda)]$, which in view of $r<0$ amounts
to $r<1-2l(\lambda)$ (and $r<0$).
In the latter case, if $r\in (0,1]$, $\gamma>0$
still, but now we need $(n+r-2)(n-r)-\lambda>0$,
i.e.\ $r\in(1-2l(\lambda),1+2l(\lambda))$. As $r\in(0,1]$, this means
$r\in(\max(0,1-2l(\lambda)),1]$.
On the other hand,
if $r>1$, we have already seen that $\gamma \left(\frac{(n-r)(n+r-2)}{4}
-\lambda\right)$ can be made positive as well as long as $r\neq 1+2l(\lambda)$.
This completes the proof of the lemma if $\lambda\leq\frac{(n-1)^2}{4}$.

For $\lambda>\left(\frac{n-1}{2}\right)^2$,
$\frac{(n-r)(n+r-2)}{4}-\lambda<0$ for all values of $r$. The `positive'
commutator criterion thus becomes that $r-1+\gamma$, $1-\gamma$ and
$-\gamma$ must have the same sign. The first two give $\gamma\in(1,1-r)$,
resp.\ $\gamma\in(1-r,1)$ depending on $r<0$ or $r>0$, as beforehand, while
the last two give $\gamma\nin[0,1]$. As $(1,1-r)$ or $(1-r,1)$ intersects
the complement of $[0,1]$ in a non-empty set if $r<0$ or $r>1$,
we get exactly the range stated in the lemma, taking into account that
$\max(0,1-2l(\lambda))=1$, $\min(0,1-2l(\lambda))=0$.
\end{proof}

If the conditions of Lemma~\ref{lemma:same-signs} are satisfied,
the right hand side of \eqref{eq:comm-form}, applied to $v$
supported near $Y$, is,
modulo the terms involving $P$, bounded below a positive multiple
(if all quantities in \eqref{eq:sign-list} are positive),
resp.\ bounded above by a negative multiple
(if all quantities in \eqref{eq:sign-list} are negative),
of the squared $x^l H_0^1$
norm of $v$, $l=-\frac{r-1}{2}$. We thus have:

\begin{lemma}
Suppose
\begin{equation}\label{eq:exp-range-lambda}
l\in(-\infty,\min(\frac{1}{2},l(\lambda)),\ l\neq -l(\lambda)
\ \Mor\ l\in(\max(\frac{1}{2},l(\lambda)),+\infty).
\end{equation}
Then there exists $C>0$ and $\delta>0$ such that
\begin{equation}\label{eq:dual-estimate}
\|x^{-l}v\|_{H^1_0}\leq C\|x^{-l}Pv\|_{L^2}.
\end{equation}
for all $v\in \dCI(X)$ with $\supp v\subset\{x<\delta\}$.
\end{lemma}

\begin{rem}
Note that (near $x=0$)
$x^s\in x^l L^2$ if $l<s-(n-1)/2$, so (neglecting the $\frac{1}{2}$
above) the two critical values
$l=-l(\lambda)$ and $l=l(\lambda)$
arise from the monomials $x^{-l(\lambda)+\frac{n-1}{2}}$,
resp.\ $x^{l(\lambda)+\frac{n-1}{2}}$, which are exactly the
monomial solutions of $Pv=0$.
\end{rem}

\begin{proof}
Note that \eqref{eq:exp-range-lambda} holds if and only if one
of the conditions in Lemma~\ref{lemma:same-signs} holds with
$l=-\frac{r-1}{2}$.

First, suppose that $v\in\dCI(X)$ supported in $x<\delta$ and $g$
is an exact warped product Lorentzian metric for $x<2\delta$.
Then
\begin{equation*}\begin{split}
\langle\frac{i}{2} Av,Pv\rangle&
-\langle\frac{i}{2} Pv,Av\rangle=\langle \frac{i}{2}[P,A]v,v\rangle\\
&=(r-1+\gamma)\|x^{\frac{r-1}{2}}(xD_x+i\frac{n-r}{2})v\|^2
+(1-\gamma)\|x^{\frac{r+1}{2}}d_Y v\|^2\\
&\qquad+\gamma\left(\frac{(n-r)(n+r-2)}{4}-\lambda\right)
\|x^\frac{r-1}{2}v\|^2+\langle Bv,Pv\rangle+\langle Pv,Bv\rangle,
\end{split}\end{equation*}
so as the three squares on the right hand side have coefficients
with the same sign,
\begin{equation*}\begin{split}
\|x^{-l}v\|^2_{H^1_0}&\leq C\|x^{-l}Pv\|_{L^2}
(\|x^l Av\|_{L^2}+\|x^l Bv\|_{L^2})\\
&\leq
C\ep^{-1}\|x^{-l}Pv\|_{L^2}^2+C\ep(\|x^l Av\|_{L^2}^2+\|x^l Bv\|_{L^2}^2).
\end{split}\end{equation*}
As $\|x^l Av\|_{L^2}^2+\|x^l Bv\|_{L^2}^2\leq C'\|x^{-l}v\|_{H^1_0}^2$,
for $B=-\frac{\gamma}{2}\,x^{-2l}$, $A=x^{-2l}(xD_x+i\frac{n-r}{2})$, for
$\ep>0$ small we deduce that (with a new $C>0$)
\begin{equation*}
\|x^{-l}v\|_{H^1_0}\leq C\|x^{-l}Pv\|_{L^2}.
\end{equation*}
This proves the lemma for warped product $g$ (with $\delta>0$ arbitrary,
as long as on $x<2\delta$ the metric is warped product).

If we do not consider an exact warped
product metric near $Y$, then $P=P_0+P_1$,
$P_0=\Box_0$ is the wave operator for the warped product metric and
$P_1\in x\Diff^2_0(X)$. Moreover, making $A$ self-adjoint with respect to
the new metric, $A=A_0+A_1$, $A_1\in x^r\Diff^1_0(X)$. Thus,
\begin{equation*}
[P,A]=[P_0,A_0]+R',\ R'\in x^r\Diff^2_0(X).
\end{equation*}
Taking into account that $l=-\frac{r-1}{2}$, for functions $v$ supported in
$x<\delta$ this gives
\begin{equation*}
|\langle v,R'v\rangle|\leq C\delta\|x^{-l}v\|^2_{H^1_0}
\end{equation*}
with $C$ depending on $R'$ only (i.e.\ independent of $\delta\in (0,1]$),
so for sufficiently small $\delta>0$, \eqref{eq:dual-estimate} still
holds.
\end{proof}

The estimate \eqref{eq:dual-estimate} gives, by duality, an existence
result. As the argument is local
near each connected component of $Y$, we have:

\begin{prop}\label{prop:solvable}
Suppose $g$ is asymptotically de Sitter like, $P=\Box-\lambda$,
$l(\lambda)$ is given by \eqref{eq:l(lambda)-def},
and
\begin{equation}\label{eq:dual-exp-range-lambda}
l\in(-\infty,-\max(\frac{1}{2},l(\lambda))),
\ \Mor\ l\in(-\min(\frac{1}{2},l(\lambda)),+\infty),
\ l\neq l(\lambda).
\end{equation}
For every $f\in x^{l}L^2(X)$ there exists $u\in x^{l}H^1_0(X)$ such that
$Pu=f$ near $Y$. Moreover, if $Y_j$ is a connected component of $Y$,
and $\supp f$ is disjoint from other components of $Y$, then $\supp u$
may be taken disjoint from other components of $Y$.
\end{prop}

\begin{proof}
Note that $P=P^*$ (formal adjoint). The result is standard
then, see \cite[Proof of Theorem~26.1.7]{Hor}.
Indeed, \eqref{eq:dual-estimate} shows that for $f\in x^{-l}H^1_0$,
$v\in\dCI(X)$ supported in $x<\delta$,
\begin{equation*}
|\langle f,v\rangle|\leq C\|x^{-l}Pv\|_{L^2}.
\end{equation*}
Thus, $Pv\mapsto \langle f,v\rangle$ is an anti-linear functional
on elements of $\dCI(X)$ supported in $x<\delta$, continuous with respect
to the $x^l L^2$-norm. By the Hahn-Banach theorem
it can be extended to a continuous conjugate-linear functional
on $x^l L^2$, so there exists $u\in x^{-l} L^2$ such that
$\langle f,v\rangle=\langle u,Pv\rangle$, and $u$ is now the desired solution
for $l$ as above.
\end{proof}

In order to use the positive commutator
argument with $v$ not supported near $Y$,
we need a cutoff $\chi$, so instead of $A=A_r$, we would really
use $A=\chi(x)^2 A_r+A_r \chi(x)^2$, $\chi\equiv 1$ near $0$,
$\chi\in\Cinf_c(\RR)$. We can also localize at any given connected
component of $Y$; as this can be done by a locally constant function
on $\supp\chi$, we do not indicate this in the notation as it
leaves the commutator unchanged. Then
\begin{equation}\begin{split}\label{eq:comm-chi}
\frac{i}{2}\,[P,A]=&(r-1+\gamma)(xD_x+i\frac{n-r}{2})^*
x^{r-1}\chi^2(xD_x+i\frac{n-r}{2})\\
&\qquad+(1-\gamma)x^{r+1}\chi^2\Delta_Y
+\gamma \left(\frac{(n-r)(n+r-2)}{4}-\lambda\right)x^{r-1}\chi^2\\
&\qquad+(xD_x)^*(\chi^2)'(xD_x)+R
+PB+B^*P,
\end{split}\end{equation}
where $R=R(x)$, $R\in\Cinf_c(\Real)$, supported away from $0$. (Again, this
comes from a principal symbol computation, which has to be carried out
away from $\pa X$, and reality plus self-adjointness shows that $R$ is
$0$th order.) Thus, modulo the 0th order term supported in the interior
and terms involving $P$ we have a {\em global} `positive commutator'
estimate (all
terms have the same sign) if $r<\min(0,1-2l(\lambda))$; if
$r>\max(0,1-2l(\lambda))$ but $r\neq 1+2l(\lambda)$, the
commutator terms with $\chi^2$ has opposite sign compared to the `main' terms.

One can also add a regularizing factor, $\left(\frac{x}{x+\ep}\right)^s
=(1+\ep x^{-1})^{-s}$ with $s>0$ small. For $\ep>0$, this is a symbol
of order $-s$ (i.e.\ decaying as $x\to 0$), and is uniformly bounded
as a symbol of order $0$. Moreover,
\begin{equation*}
(x\pa_x)^k(1+\ep x^{-1})^{-s}
=s(1+\ep x^{-1})^{-s} f_{k,\ep,s},
\end{equation*}
where $f_{k,\ep,s}$ is a symbol
of order $0$, and is uniformly bounded as such a symbol. Consequently, as
long as one has a positive normal operator for the commutator
of $P$ with some operator $A$, one
will also have a positive normal operator for the commutator
of $P$ with $(1+\ep x^{-1})^{-s}A
(1+\ep x^{-1})^{-s}$ if $s$ is small. It is actually even easier
to simply apply our previous estimate, \eqref{eq:dual-estimate},
to a regularized
version $v_\ep=(1+\ep x^{-1})^{-s}v$ of $v$, for
$Pv_\ep=(1+\ep x^{-1})^{-s}Pv+[P,(1+\ep x^{-1})^{-s}]v$, noting
that $(1+\ep x^{-1})^{s}
[P,(1+\ep x^{-1})^{-s}]$ is bounded by $C's$ in $\Diff^1_{0,c}(X)$
($c$ denotes conormal coefficients, but should be changed),
so the $L^2$ norm of $[P,(1+\ep x^{-1})^{-s}]v$ can be absorbed
into the left-hand side of \eqref{eq:dual-estimate} for $s>0$ small.
Applying this iteratively, we deduce the following:

\begin{prop}\label{prop:decay-gain}
Suppose $g$ is asymptotically de Sitter like, $P=\Box-\lambda$,
$\lambda\in\Real$. Suppose that $u\in x^{l_0} H^1_0(X)$
$Pu\in x^l L^2(X)$, $l>l_0$. Suppose also that one of the following
conditions holds:

\begin{enumerate}
\item
$l<-l(\lambda)$,
\item
$l_0>\max(\frac{1}{2},l(\lambda))$,
\item
$l_0>-l(\lambda)$, $l<\min(\frac{1}{2},l(\lambda))$.
\end{enumerate}

Then $u\in x^l H^1_0(X)$.

Moreover, the result is local near each connected component of $Y$.
\end{prop}

This immediately gives that if a solution of $Pu=0$ decays faster than
a borderline rate, given by $x^{l(\lambda)}L^2$,
then it is Schwartz.
In fact, later in Proposition~\ref{prop:unique},
we show that such $u$ is necessarily identically $0$.

\begin{cor}\label{cor:decay-Schwartz}
Suppose that $u\in x^l H^k_0(X)$, $k\in\Real$,
$\lambda\in\Real$, $l>\max(\frac{1}{2},l(\lambda))$,
$Pu\in\dCI(X)$. Then $u\in\dCI(X)$.

If the assumptions hold near a connected component of $Y$ only, so
does the conclusion.

\end{cor}

\begin{rem}\label{rem:decay-Schwartz}
The assumption $l> \max(\frac{1}{2},l(\lambda))$
is probably not optimal if $l(\lambda)<\frac{1}{2}$,
cf.\ Remark~\ref{rem:solvable-up-to-compact}; one expects
$l>l(\lambda)$ simply.
However, this makes no difference in the present paper. Moreover,
for $\Box$ itself this
is not a restriction as $n\geq 2$ so $l(\lambda)\geq\frac{1}{2}$.

This corollary also states in particular that for $f\in\dCI(X)$ the solution
$u\in x^l H^1_0(X)$ of $Pu=f$ near $Y$, whose existence is
guaranteed by Proposition~\ref{prop:solvable}, is in fact in $\dCI(X)$.
\end{rem}

\begin{proof}
First, we may assume $k=1$. Indeed, if $k<1$, then $l>1/2$ gives
$k<1<l+1/2$, so (i) of Proposition~\ref{prop:edge-reg} applies
and gives $u\in H^{1,l}_0(X)$.

By Proposition~\ref{prop:decay-gain}, $u\in x^l H^1_0(X)$ for all $l$.
Thus, by Proposition~\ref{prop:edge-reg}, part (i), $u\in H^{r,s}_0(X)$ for
all $r$ and $s$ with $r<s+1/2$, hence for all $(r,s)$.
(Given $(r,s)$, consider $(r,s')$ with $s'>\max(s,r-1/2)$ to
see that $u\in H^{r,s'}_0(X)$ hence $u\in H^{r,s}_0(X)$.)
In particular, $x^m Qu\in L^2(X)$ for all $m$ and
all $Q\in\Diff(X)$, proving the corollary.
\end{proof}

\section{Conormal regularity}\label{sec:conormal}
While Proposition~\ref{prop:solvable} gives the correct
critical rates of growth or decay for solutions of $Pu=0$,
and Corollary~\ref{cor:edge-reg} gives their optimal smoothness in the
0-sense, this is not optimal: solutions of $Pu=0$ which are
$\CI$ in $X^\circ$ are conormal to the boundary, i.e.\ stable
(in terms of weighted $L^2$-spaces) under the application of b-differential
operators. In fact, as usual, cf.\ \cite{Vasy:Propagation-Wave}
and \cite{Melrose-Vasy-Wunsch:Propagation}, it is convenient to work
relative to 0-Sobolev spaces,
i.e.\ to work with $\Diff^k_0\Psib^m(X)$. However, rather than using
positive commutator estimates as in these papers, we rely on
an `exact' commutator argument (exact at the level of normal operators),
much like in \cite[Section~12]{RBMSpec}.
Although it was not discussed explicitly in \cite{RBMSpec} for reasons
of brevity, the analogous space of operators in that setting would
be $\Diffsc^k\Psi_c(X)$, with $\Psi_c(X)$ standing for {\em cusp}
pseudodifferential operators. (Instead, in \cite{RBMSpec} `tangential
elliptic regularity' was used.)

\begin{Def}
Elements of $\Diff^k_0\Psib^m(X)$ are finite sums of terms $QA$,
$Q\in\Diff^k_0(X)$, $A\in\Psib^m(X)$. We also let $x^r\Diff^k_0\Psib^m(X)$
be the space of operators of the form $x^rB$, $B\in\Diff^k_0\Psib^m(X)$.
\end{Def}

\begin{rem}
Directly from the definition, $\Diff^k_0\Psib^m(X)$ is
a $\CI(X)$-bimodule (under left and right multiplication), so
in particular $x^r\Diff^k_0\Psib^m(X)$ is well-defined independent
of the choice of a boundary defining function $x$.
\end{rem}

The key lemma is:

\begin{lemma}\label{lemma:Diff_0-Psib}
For $Q\in\Diff^k_0(X)$, $A\in\Psib^m(X)$, there exist $Q_j\in\Diff^k_0(X)$,
$A_j\in\Psib^m(X)$, $j=1,\ldots,l$, such that $QA=\sum A_j Q_j$.
(With a similar conclusion holding, with different $A_j$, $Q_j$,
for $AQ$.)
\end{lemma}

\begin{proof}
It suffices to prove the statement for $Q\in\Vf_0(X)$; the general
case then follows by an inductive argument. As
$\Vf_0(X)\subset\Vb(X)$, $[Q,A]\in\Psib^m(X)$, so
$QA=AQ+[Q,A]$ gives the desired result.
\end{proof}

\begin{cor}
$\Diff_0\Psib(X)$ is closed under composition: if $A\in\Diff^k_0\Psib^m(X)$
and $B\in\Diff^{k'}_0\Psib^{m'}(X)$ then
$AB\in\Diff^{k+k'}_0\Psib^{m+m'}(X)$.
\end{cor}

We also need the corresponding result about commutators.

\begin{lemma}\label{lemma:b-0-commutator}
Moreover, if $A\in x^r\Psib^m(X)$, $Q\in\Diff^k_0(X)$ then
\begin{equation*}
[Q,A]\in x^r\Diff^{k-1}_0\Psib^m(X).
\end{equation*}

If in addition
$\sigma_{b,m}(A)|_{\Tb^*\pa X}=0$
then $[Q,A]\in x^r\Diff^k_0\Psib^{m-1}(X)$.
\end{lemma}

\begin{rem}
$\Tb^*\pa X$ is a well-defined subbundle of $\Tb^*_{\pa X}X$.
If we write b-covectors as $\sigma\,\frac{dx}{x}+\eta\cdot dy$,
then $\Tb^*\pa X$ is given by $x=0$, $\sigma=0$ in $\Tb^*X$.
\end{rem}

\begin{proof}
Again, it suffices to prove the first statement for $Q\in\Vf_0(X)$. As
$\Vf_0(X)\subset\Vb(X)$, $[Q,A]\in\Psib^m(X)$, giving the result
for such $Q$. Iterating this also proves
that for $Q\in\Diff^k_0(X)$, $[Q,A]\in\Diff^{k-1}_0\Psib^m(X)$.

To have the better conclusion, it again suffices to consider
$Q\in\Vf_0(X)$. As above, $[Q,A]\in\Psib^m(X)$. But,
with $a=\sigma_{b,m}(A)$, $q=\sigma_{b,1}(Q)$,
\begin{equation*}\begin{split}
i\sigma_{b,m}([A,Q])&
=H_a q\\
&=(\pa_\sigma a)(x\pa_x q)-(x\pa_x a)(\pa_\sigma q)
+\sum \left((\pa_{\eta_j}a)(\pa_{y_j}q)-(\pa_{y_j}a)(\pa_{\eta_j}q)\right).
\end{split}\end{equation*}
This vanishes at $\Tb^*\pa X$ for $a$ vanishes there, hence so do all
terms but the first one, and the first one vanishes as $x\pa_x q$ vanishes
at $x=0$. Thus, $\sigma_{b,m}([A,Q])=\sigma b+xe$ for some $b\in
S^{m-1}_{\hom}(\Tb^*X\setminus o)$, $e\in S^m_{\hom}(\Tb^*X\setminus o)$.
We deduce that there exists $B\in\Psib^{m-1}(X)$, $E\in\Psib^m(X)$,
$R\in\Psib^{m-1}(X)$ such
that $[Q,A]=B(xD_x)+Ex+R$. As one can write $E=E_0(xD_x)+\sum E_j D_{y_j}
+R'$ with $E_j,R'\in\Psib^{m-1}(X)$, and as $x(xD_x),xD_{y_j}\in
\Vf_0(X)$, the second claim is proved.
\end{proof}

\begin{lemma}\label{lemma:Psib-bded}
Suppose $m\geq 0$ is an integer.
Any $A\in\Psib^0(X)$ defines a continuous linear map on $H^{m,l}_0(X)$
by extension from $\dCI(X)$.
\end{lemma}

\begin{proof}
We can use any collection $B^{(i)}\in\Diff_0^m(X)$, $i=1,\ldots,N$,
such that at each point of $\zS^*X$ at least one of the $B^{(i)}$ is
elliptic, to put a norm
on $H^{m,l}_0(X)$:
\begin{equation*}
\|u\|^2_{H^{m,l}_0(X)}=\sum_i\|x^{-l}B^{(i)}u\|^2_{L^2(X)}
+\|x^{-l}u\|^2_{L^2(X)}.
\end{equation*}
We need to show then that for $A$ as above,
$\|Au\|_{H^{m,l}_0(X)}\leq C\|u\|_{H^{m,l}_0(X)}$. Since $A$ is bounded
on $x^{-l}L^2(X)$, we only need to prove that for each $i$,
$\|x^{-l}B^{(i)}Au\|\leq C'\|u\|_{H^{m,l}_0(X)}$.
But $x^{-l}B^{(i)}A=\sum A_j x^{-l}B_j$
with $A_j\in\Psib^0(X)$ and $B_j\in\Diff_0^m(X)$
by Lemma~\ref{lemma:Diff_0-Psib}, so
$\|x^{-l}B^{(i)}A u\|\leq \sum C_j\|x^{-l}B_j u\|$
as $A_j$ are bounded on $L^2(X)$.
This proves the corollary.
\end{proof}

As we work relative to $x^l H^r_0(X)=H^{r,l}_0(X)$,
for $k\geq 0$ we use the Sobolev spaces
\begin{equation*}
x^l H^{k,r}_{b,0}(X)=\{u\in x^l H^r_0(X):\ \forall A\in\Psib^k(X),
\ Au\in x^l H^1_0(X)\}.
\end{equation*}
These can be normed by taking any elliptic $A\in\Psib^k(X)$ and
letting
\begin{equation*}
\|u\|_{x^l H^{k,r}_{b,0}(X)}^2=\|u\|_{x^l H^r_0(X)}^2
+\|Au\|^2_{x^l H^r_0(X)}.
\end{equation*}
Although the norm depends on the choice of $A$, different choices
give equivalent norms. Indeed, if $\tilde A\in\Psib^k(X)$, then
let $G\in\Psib^{-k}(X)$ be a parametrix for $A$, so
$GA=\Id+E$, $AG=\Id+F$, $E,F\in\Psib^{-\infty}(X)$, and note
that
\begin{equation}\begin{split}\label{eq:Hb0-well-def}
\|\tilde Au\|_{x^l H^r_0(X)}&\leq \|\tilde A GAu\|_{x^l H^r_0(X)}
+\|\tilde AE u\|_{x^l H^r_0(X)}\\
&\leq C(\|Au\|_{x^l H^r_0(X)}
+\|u\|_{x^l H^r_0(X)}),
\end{split}\end{equation}
where we used that $\tilde AG\in\Psib^0(X)$ and $AE\in\Psib^{-\infty}(X)
\subset\Psib^0(X)$ are bounded on $x^l H^r_0(X)$ by
Lemma~\ref{lemma:Psib-bded}. If $\tilde A$ is elliptic, there
is a similar estimate with the role of $A$ and $\tilde A$ interchanged,
which shows the claimed equivalence.

\begin{lemma}\label{lemma:Psib-b0-bded}
If $Q\in\Psib^0(X)$, then $Q$ is bounded on $x^l H^{k,r}_{b,0}(X)$.
\end{lemma}

\begin{proof}
As $Q$ is bounded on $x^lH^r_0(X)$, we only need to prove that
for $A\in\Psib^k(X)$, $\|AQu\|_{x^l H^r_0(X)}\leq
C(\|u\|_{x^l H^r_0(X)}
+\|Au\|_{x^l H^r_0(X)})$. But $\tilde A=AQ\in\Psib^k(X)$, though
not necessarily elliptic, so by \eqref{eq:Hb0-well-def}, this estimate
holds.
\end{proof}

\begin{lemma}\label{lemma:elliptic}
If $L\in\Diffb^k(X)$ is elliptic, $u\in x^l H_{b,0}^{s,\infty}(X)$,
$Lu\in x^l H_{b,0}^{s,\infty}(X)$, then $u\in x^l H_{b,0}^{s+k,\infty}(X)$.
\end{lemma}

\begin{proof}
Let $G\in\Psib^{-k}(X)$ be a parametrix for $L$ so that
$GL=\Id+R$, $R\in\Psib^{-\infty}(X)$. Then $u=G(Lu)-Ru$. Now,
if $A\in\Psib^k(X)$ then $Au=(AG)(Lu)-(AR)u\in x^l H_{b,0}^{s,\infty}(X)$
by Lemma~\ref{lemma:Psib-b0-bded}
since $AG,AR\in\Psib^0(X)$. This proves the lemma.
\end{proof}

The conormal regularity
theorem is {\em global} in each connected component of $Y$.
It uses the following lemma, which shows that the boundary Laplacian
commutes with $P$ one order better (in terms of decay)
than a priori expected:

\begin{lemma}
Let $\tilde\Delta_Y\in \Diffb^2(X)$ have normal operator given by
$\Delta_Y$. Then $[P,\tilde\Delta_Y]\in x\Diff^1_0\Diff^2_b(X)$.
\end{lemma}

\begin{proof}
Changing $\tilde\Delta_Y$ by $Q\in x\Diffb^2(X)$ changes the
commutator by an element of $x\Diff^1_0\Diff^2_b(X)$ due to
Lemma~\ref{lemma:b-0-commutator}, so the statement only
depends on the normal operator of $\tilde\Delta_Y$. Similarly,
it only depends on the normal operator of $P$. Thus, we may work
on the model space $[0,\ep)_x\times Y$, replace $P$ by
$(xD_x)^2+i(n-1)(xD_x)-x^2\Delta_Y$, $\tilde \Delta_Y$ by $\Delta_Y$,
and then the result is immediate.
\end{proof}

\begin{prop}\label{prop:conormal-reg}
Suppose $l\in\Real$, $u\in x^l H^{-\infty}_0(X)$,
$Pu\in\dCI(X)$ and $u\in \CI(X^\circ)$.
Then for all $\ep>0$, $u\in x^{l-\ep} H^{\infty,0}_{b,0}(X)=
x^{l-\ep} H^{\infty,\infty}_{b,0}(X)$.
\end{prop}

\begin{rem}
The proposition states
that once one knows that $u$ is smooth in $X^\circ$ and is
in some weighted $L^2$-space, one gets
b-regularity relative to that space.

Also, the proposition can be restated in terms of the standard b-spaces:
$u\in x^{l+\frac{n-1}{2}-\ep}H^{\infty}_b(X)$.
The shift $\frac{n-1}{2}$ in the exponent is simply due to
$H^s_b(X)$ being defined relative to $L^2_b(X)$, the $L^2$-space
relative to a non-vanishing b-measure.
\end{rem}

\begin{proof}
Assume first that $l<-l(\lambda)$. We prove that
$u\in x^{l-\ep} H^{\infty,\infty}_{b,0}(X)$. We first note that
by Corollary~\ref{cor:edge-reg}, $u\in H^{\infty,l-\ep}_0(X)$ for all
$\ep>0$, i.e.\ we have full 0-regularity. Let $\tilde\Delta_Y$ be
as above.

As $u\in H^{\infty,l-\ep}_0(X)$, $\tilde\Delta_Y\in x^{-2}\Diff^2_0(X)$, we
see that
$\tilde\Delta_Y u\in H^{\infty,l-2-\ep}_0(X)$. Then
\begin{equation}\label{eq:P-Delta_Y}
P\tilde\Delta_Y u=\tilde\Delta_Y Pu+[P,\tilde\Delta_Y]u\in H^{\infty,l-1-\ep}_0(X)
\end{equation}
since $[P,\tilde\Delta_Y]\in x\Diff^1_0\Diffb^2(X)\subset x^{-1}\Diff^3_0(X)$.
(In fact, this can be phrased by saying that $N(\tilde\Delta_Y)$ and $N(P)$
commute.) Thus, by Proposition~\ref{prop:decay-gain},
$\tilde\Delta_Y u\in H^{\infty,l-1-\ep}_0(X)$. As
$(xD_x)^2u\in H^{\infty,l-\ep}_0(X)$, $((xD_x)^2+\tilde\Delta_Y)u\in
H^{\infty,l-1-\ep}_0(X)$. Since $(xD_x)^2+\tilde\Delta_Y$ is elliptic
in $\Diffb^2(X)$, Lemma~\ref{lemma:elliptic} shows
that $u\in x^{l-1-\ep}H_{b,0}^{2,\infty}(X)$.

Thus, \eqref{eq:P-Delta_Y} and $[P,\tilde\Delta_Y]\in x\Diff^1_0\Diffb^2(X)$ gives
$P\tilde\Delta_Y u\in x^{l-\ep}H_0^\infty(X)$, so
by Proposition~\ref{prop:decay-gain},
$\tilde\Delta_Y u\in H^{\infty,l-\ep}_0(X)$. Proceeding as above, we
deduce that $u\in x^{l-\ep}H_{b,0}^{2,\infty}(X)$.

We now iterate this argument for $\tilde\Delta_Y^k u$ in place of $\tilde\Delta_Y u$.
So suppose we already know that $u\in x^{l-\ep}H_{b,0}^{2(k-1),\infty}(X)$
for all $\ep>0$. Then $[P,\tilde\Delta_Y^k]\in x\Diff^1_0\Diffb^{2k}
\subset x^{-1}\Diff^3_0\Diffb^{2(k-1)}(X)$, so
\begin{equation*}
P\tilde\Delta_Y^k u=\tilde\Delta_Y^k Pu+[P,\tilde\Delta_Y^k]u\in H^{\infty,l-1-\ep}_0(X)
\end{equation*}
Again, by Proposition~\ref{prop:decay-gain},
$\tilde\Delta_Y^k u\in H^{\infty,l-1-\ep}_0(X)$. As
$(xD_x)^{2k}u\in H^{\infty,l-\ep}_0(X)$, $((xD_x)^{2k}+\tilde\Delta_Y^k)u\in
H^{\infty,l-1-\ep}_0(X)$. Using Lemma~\ref{lemma:elliptic}, we
conclude that  $u\in x^{l-1-\ep}H_{b,0}^{2k,\infty}(X)$.

Equipped with this additional knowledge, we deduce that
$[P,\tilde\Delta_Y^k]u\in H^{\infty,l-\ep}_0(X)$, hence $P\Delta^k_Y u$
is in the same space. Applying Proposition~\ref{prop:decay-gain},
we see that
$\tilde\Delta_Y u\in H^{\infty,l-\ep}_0(X)$. Proceeding as above, we
deduce that $u\in x^{l-\ep}H_{b,0}^{2k,\infty}(X)$.
This proves the proposition if $l<-l(\lambda)$.

In general, if $l\geq -l(\lambda)$,
we may apply the previous
argument with $l$ replaced by any $l'<-l(\lambda)$
to conclude that
$u\in x^{l'}H_b^\infty(X)$ for all $l'<-l(\lambda)$.
Since $u\in x^l
L^2(X)$, interpolation gives $u\in x^{l-\ep}H_b^\infty(X)$ as stated.
\end{proof}

We now consider $P=\Box-\lambda$ acting on polyhomogeneous functions, or
more generally symbols. Recall that $u\in\bcon^k(X)$ means that
$Lu\in x^k L^2_b(X)$ for all $L\in\Diffb(X)$, so in particular
$u\in x^k L^2_b(X)$.

We remark that if $s_+,s_-\in\Cx$ with $s_+-s_-\notin\intgr$, and a function
$u$ has the form $x^{s_+}v_++x^{s_-}v_-$, $v_\pm\in\CI(X)$, then the
leading terms $v_\pm|Y$ (in fact, the full Taylor series of $v_\pm$) is
well-defined. However, if $s_+-s_-$ is an integer, this is no longer true,
which explains some of the complications we face in stating the converse
direction of the following lemma.

\begin{lemma}\label{lemma:expansion}
Suppose $\lambda\in\Real$, $\lambda\neq\frac{(n-1)^2}{4}$. Let
\begin{equation*}
s=s_\pm(\lambda)=
\frac{n-1}{2}\pm\sqrt{\left(\frac{n-1}{2}\right)^2-\lambda},
\end{equation*}
be the (not necessarily real) indicial roots of $(xD_x+i(n-1))(xD_x)-\lambda$.
If $u\in\bcon^k(X)$ for some $k$ and $Pu\in\dCI(X)$
and $s_+(\lambda)-s_-(\lambda)$ is not an integer then
there exists $v_\pm\in\CI(X)$,
such that
\begin{equation*}
u=x^{s_+(\lambda)}v_++x^{s_-(\lambda)}v_-.
\end{equation*}
If $s_+(\lambda)-s_-(\lambda)$ is an integer (in which case both
$s_\pm(\lambda)$ are real) then the analogous statement
holds with $v_-\in\CI(X)$ replaced by
\begin{equation*}
v_-\in\CI(X)+x^{s_+(\lambda)-s_-(\lambda)}\log x\,\CI(X).
\end{equation*}
In either case, if $v_\pm|_{Y}$ vanish, then $u\in\dCI(X)$.

Conversely,
given $g_+,g_-\in\CI(Y)$, there exist
\begin{enumerate}
\item
$v_\pm\in\CI(X)$ if $s_+(\lambda)-s_-(\lambda)$ is not an integer,
\item
\begin{equation*}
v_+\in\CI(X),\ v_--\sum_{j=0}^{s_+(\lambda)-s_-(\lambda)-1}
a_j x^j\in x^{s_+(\lambda)-s_-(\lambda)}\log x\,\CI(X),\ a_j\in\CI(Y),
\end{equation*}
if $s_+(\lambda)-s_-(\lambda)$ is an integer,
\end{enumerate}
such that
\begin{equation*}
u=x^{s_+(\lambda)}v_++x^{s_-(\lambda)}v_-,\ v_\pm|_Y=g_\pm,
\end{equation*}
satisfies $Pu\in\dCI(X)$.
\end{lemma}

\begin{proof}
We start with the converse direction.
As $P=(xD_x+i(n-1))(xD_x)-\lambda+Q$, $Q\in x\Diffb^2(X)$, for
$v\in\CI(X)$,
\begin{equation}\label{eq:P-xs}
P(x^s v)=(s(n-1-s)-\lambda)x^sv+w,\ w\in x^{s+1}\CI(X).
\end{equation}
Thus, when $s$ is an indicial root,
$P(x^s v)\in x^{s+1}\CI(X)$ automatically, and otherwise given
$f\in x^s\CI(X)$,
$P(x^s v)=f$ can be solved {\em uniquely},
modulo $x^{s+1}\CI(X)$, with $v\in\CI(X)$.
Iterating this argument, and using Borel summation, we deduce that
unless the two indicial roots differ by an integer,
given $g_+,g_-\in\CI(Y)$, there exists $v_+,v_-\in\CI(X)$ such that
\begin{equation*}
u=x^{s_+(\lambda)}v_++x^{s_-(\lambda)}v_-,\ v_\pm|_Y=g_\pm,
\end{equation*}
satisfies $Pu\in\dCI(X)$.

If the two indicial roots differ by an integer
(but are distinct, i.e.\ not equal to $\frac{n-1}{2}$),
only a minor modification is needed in that we need to allow
logarithmic factors. Thus, for $v\in\CI(X)$,
\begin{equation}\begin{split}\label{eq:P-xs-log}
P(x^s\log x v)=&(s(n-1-s)-\lambda)\log x\, x^s v+(n-1-2s)x^s v+w,\\
&\qquad\qquad w\in x^{s+1}\log x\,\CI(X)+x^{s+1}\CI(X),
\end{split}\end{equation}
so if $s=s_\pm(\lambda)$, $Pu=f$, $f\in x^s \CI(X)$, has a
solution modulo $x^{s+1}\log x\,\CI(X)+x^{s+1}\CI(X)$, of the
form $u\in x^s\log x\,\CI(X)$, so applying this with $s=s_+(\lambda)$, the
error term arising from $s_-(\lambda)$ of the form $x^s$ times
a smooth function, can be solved away to leading order. Moreover,
for $s\neq s_\pm (\lambda)$, $Pu=f$, $f\in x^s\log x\,\CI(X)$ has a solution,
modulo $x^{s+1}\log x\,\CI(X)+x^{s+1}\CI(X)$, of the
form $u\in
x^s\log x\,\CI(X)$, so again iteration gives infinite order solvability,
in this case of the form:
given $g_+,g_-\in\CI(Y)$, there exists $v_+\in\CI(X)$,
$v_-\in\CI(X)+x^{s_+(\lambda)-s_-(\lambda)}\log x\,\CI(X)$ such that
\begin{equation*}
u=x^{s_+(\lambda)}v_++x^{s_-(\lambda)}v_-,\ v_\pm|_Y=g_\pm,
\end{equation*}
satisfies $Pu\in\dCI(X)$.

On the other hand, suppose that $u\in\bcon^k(X)$ and $Pu\in\dCI(X)$.
As $Qu\in\bcon^{k+1}(X)$,
we have $((xD_x+i(n-1))(xD_x)-\lambda)u\in\bcon^{k+1}$.
Since near $Y$, using an product decomposition of a neighborhood
of $Y$, $\bcon^r(X)$ can be identified with $\CI(Y;\bcon^r([0,\ep)))$,
we can treat $Y$ as a parameter and solve this ODE. If there
is no indicial root in $(k,k+1]$, one deduces that
$u\in\bcon^{k+1}(X)$; otherwise $u=\sum_j x^{s_j} g_j+u'$ where the $s_j$ are
the indicial roots in the interval, $g_j$ are smooth and $u'\in\bcon^{k+1}$.
By the first part of the proof one can choose $v_j$ as in the statement of
the lemma (denoted by $v_\pm$ there) to get $u_j=x^{s_j}v_j\in\bcon^k$
with $Pu_j\in\dCI(X)$ and $u_j-x^{s_j}g\in\bcon^{k+1}$. Thus,
$u-\sum u_j\in\bcon^{k+1}$ with $P(u-\sum u_j)\in\dCI(X)$, so one
can proceed iteratively to finish the existence argument. Note that
if $g_j|Y$ vanish, one concludes $u\in\bcon^{k+1}$, which by iteration
gives the uniqueness.
\end{proof}

In fact, the same argument also deals with the case $\lambda=(n-1)^2/4$, but
as the result is of a slightly different form, we state it separately:

\begin{lemma}\label{lemma:expansion-threshold}
Suppose $\lambda=\frac{(n-1)^2}4$, so $s_\pm(\lambda)=\frac{n-1}{2}$.
If $u\in\bcon^k(X)$ for some $k$ and $Pu\in\dCI(X)$ then
there exists $v_\pm\in\CI(X)$,
such that
\begin{equation*}
u=x^{s_+(\lambda)}v_++x^{s_-(\lambda)}\log x \,v_-.
\end{equation*}
Conversely,
given $g_+,g_-\in\CI(Y)$, there exists $v_\pm\in\CI(X)$,
such that
\begin{equation*}
u=x^{s_+(\lambda)}v_++x^{s_-(\lambda)}\log x\,v_-,\ v_\pm|_Y=g_\pm,
\end{equation*}
satisfies $Pu\in\dCI(X)$.
\end{lemma}

\begin{proof}
$s=s_\pm(\lambda)=(n-1)/2$ now satisfies $s(n-1-s)-\lambda=0$
as $n-1-2s=0$, so
\eqref{eq:P-xs} and \eqref{eq:P-xs-log} imply that
$P(x^sv_1+x^s\log x\,v_2)\in x^{s+1}\CI(X)+x^{s+1}\log x\CI(X)$.
The argument of the previous lemma then shows the second claim.

For the first claim, we need to observe that if
$u\in\bcon^k(X)$ and $Pu\in\dCI(X)$ then
$Qu\in\bcon^{k+1}(X)$, so
$((xD_x+i(n-1))(xD_x)-\lambda)u\in\bcon^{k+1}$, i.e.\ $(xD_x+i(n-1)/2)^2u
\in\bcon^{k+1}$. Proceeding as above, the only difference is that
if $s=\frac{n-1}{2}\in(k,k+1]$, one deduces that
$u=x^s g_1+x^s\log x\,g_2+u'$, $g_j$ smooth, $u'\in\bcon^{k+1}$. One finishes
the proof exactly as above.
\end{proof}

Since we already know
(by virtue of Proposition~\ref{prop:solvable}
and Remark~\ref{rem:decay-Schwartz})
that we can solve $Pu'=f$, $f\in\dCI(X)$, with $u'\in
\dCI(X)$, modulo $\CI_c(X^\circ)$, we deduce that these $u$ can be further
extended to be exact solutions near $\pa X$.

\section{Global solvability}\label{sec:global}
For global solvability, i.e.\ solvability on all of $X$ rather than
just near $\pa X$, of $Pu=0$ we need the additional assumptions (A1)-(A2).
We thus assume that $Y=Y_+\cup Y_-$, where $Y_\pm$ are unions of connected
components of $Y$, and this decomposition satisfies that all bicharacteristics
$t\mapsto \gamma(t)$ of $P$ (i.e.\ those of $\Box$, independent
of $\lambda$) satisfy $\lim_{t\to +\infty}\gamma(t)\in Y_+$,
$\lim_{t\to -\infty}\gamma(t)\in Y_-$, or vice versa.
In this case,
noting that the
sign of the $\chi'$ term agrees with the others if $r<\min(0,1-2l(\lambda))$
(for they are
all negative; recall $l(\lambda)=\frac{n-1}{2}$ for the wave operator
itself), one can easily `cut and paste' the estimates with
\begin{itemize}
\item
near $Y_+$, $r=r_+>1+2l(\lambda)$ (or just $r=r_+>\max(0,1-2l(\lambda))$,
$r_+\neq 1+2l(\lambda)$),
\item
near $Y_-$, $r=r_-<\min(0,1-2l(\lambda))$, and
\item
standard
microlocal propagation estimates in the interior of $X$
\end{itemize}
to
deduce that for a partition of unity $\chi_++\chi_-+\chi_0=1$ with $\chi_+$
supported near $Y_+$, identically $1$ in a smaller neighborhood of $Y_+$,
analogously with $\chi_-$, $\chi_0\in\CI_c(X^\circ)$, there exists
$\tilde\chi_0\in\CI_c(X^\circ)$ such that
\begin{equation}\label{eq:exist-est}
\|x^{(r_+-1)/2}\chi_+ v\|^2_{H^1_0}
+\|x^{(r_--1)/2}\chi_- v\|^2_{H^1_0}+\|\chi_0 v\|^2_{H^1_0}
\leq C(\|\tilde\chi_0 v\|^2_{H^{1/2}_0}+\|Pv\|^2).
\end{equation}
Let $H_0^{m,q_+,q_-}(X)$ be the space $x_+^{q_+}x_-^{q_-}H_0^m(X)$,
where $x_\pm$ are defining functions of $Y_\pm$, we can put the
norm
\begin{equation*}
\|v\|_{H_0^{m,q_+,q_-}(X)}^2=\|x^{-q_+}\chi_+ v\|^2_{H^m_0}
+\|x^{-q_-}\chi_- v_-\|^2_{H^m_0}+\|\chi_0 v\|^2_{H^m_0};
\end{equation*}
on it. (Note that it is the completion of $\dCI(X)$ with respect to this norm.)
This is just $x^{q_\pm} H^m_0(X)$ near $Y_\pm$, $H^m(X^\circ)$
in the interior.
Let $l_\pm=(r_\pm-1)/2$.
The argument of \cite[Proof of Theorem~26.1.7]{Hor} shows the following:

\begin{prop}\label{prop:solvable-up-to-compact}
Suppose that $\lambda\in\Real$,
$l_+>\max(\frac{1}{2},l(\lambda))$, $l_-<-\max(\frac{1}{2},l(\lambda))$.
Then
\begin{equation*}
N_{l_+,l_-}=\{v\in H_0^{1,-l_+,-l_-}(X):\ Pv=0\}
\end{equation*}
is finite dimensional, and for $f\in H_0^{0,l_+,l_-}(X)$, $f$ orthogonal
to $N_{l_+,l_-}$, $Pu=f$ has a solution $u\in H_0^{1,l_+,l_-}(X)$.

Moreover, elements of $N_{l_+,l_-}$ are in $H_0^{\infty,l,-l_-}(X)$ for all
$l<-l_+$,
are Schwartz at $Y_-$, and have an expansion as in Lemma~\ref{lemma:expansion}
at $Y_+$.
\end{prop}

\begin{rem}\label{rem:solvable-up-to-compact}
Note that the expansion of Lemma~\ref{lemma:expansion} implies
that $N_{l_+,l_-}$ are in $H_0^{\infty,l,\infty}(X)$ for all
$l<-l(\lambda)$, not merely $l<-l_+$.
\end{rem}

\begin{proof}
We first prove the last statement. For $v\in N_{l_+,l_-}$,
by Corollary~\ref{cor:decay-Schwartz}, $v$ is Schwartz at $Y_-$.
In particular, $v$ is $\CI$ near $Y_-$, so by the standard propagation
of singularities for $P$, $v\in\CI(X^\circ)$. Then,
by Corollary~\ref{cor:edge-reg}, $v\in H_0^{\infty,l,-l_-}$ for all
$l<-l_+$. By Proposition~\ref{prop:conormal-reg} and the remark
following it, $u\in x^{l+\frac{n-1}{2}}
H^\infty_b(X)=\bcon^{l+\frac{n-1}{2}}(X)$ for all $l<-l_+$.
Thus, by Lemma~\ref{lemma:expansion}, it has an expansion at $Y_+$
of the form given by Lemma~\ref{lemma:expansion}.

This in particular implies that the commutator calculations giving rise to
\eqref{eq:exist-est} can be applied directly
(without mollification) to all $v\in N_{l_+,l_-}$
The proof of the first part is finished as in \cite{Hor}, and the second
part can then be proved exactly as in \cite{Hor}.
\end{proof}

Note that the role of $Y_\pm$ is reversible, so the estimates, hence
the proposition, also hold with
$l_\pm$ interchanged. Correspondingly, we deduce that the solution $u$ of
$Pu=f$ above is unique modulo the finite dimensional space $N_{-l_+,-l_-}$.

One can also get uniqueness, namely that

\begin{prop}\label{prop:unique}
Suppose $u\in\dCI(X)$ and $Pu=0$. Then $u=0$.

In fact, it suffices to assume that $u$ is Schwartz at $Y_+$.

If we merely assume that $u$ is Schwartz
at a connected component $Y_j$ of $Y$, and $Pu=0$ near $Y_j$, then
we can still conclude that $u=0$ near $Y_j$.
\end{prop}

\begin{proof}
The proof is very similar to \cite[Section~4]{Vasy:Exponential} and
to \cite{Vasy-Wunsch:Absence}.
Consider $P_h=x^{-1/h} h^2 P x^{1/h}$. The basic claim is that the
semiclassical symbols of $\re P_h\in\Diff^2_{0,h}(X)$ and $\im P_h
\in \Diff^1_{h,0}(X)$ never vanish at the same place at $Y$.
In fact, as $P$ is formally self-adjoint, one has
\begin{equation*}\begin{split}
&P_h=h^2P+x^{-1/h}[h^2 P,x^{1/h}],\\
&\re P_h=h^2P+\frac{1}{2}[x^{-1/h},[h^2P,x^{1/h}]],\\
&\im P_h=\frac{1}{2i}(x^{-1/h}[h^2P,x^{1/h}]+[h^2P,x^{1/h}]x^{-1/h}).
\end{split}\end{equation*}
Now, for $Q\in x^l\Diff^k_{0,h}(X)$,
$x^{-1/h}[Q,x^{1/h}]\in x^l\Diff^{k-1}_{0,h}(X)$, so if we only want
to compute the commutators modulo higher order terms in $x$, we can
work with the normal operator of $P$ instead of $P$. Also, modulo
higher order terms in $h$, only the principal symbol of $P$ matters
in the calculations, as we are considering $h^2P$, and changing
$P$ by a first order term changes $h^2P$ by an element of
$h\Diff^1_{0,h}(X)$. Thus, a straightforward computation gives
\begin{equation*}\begin{split}
&\re P_h
=(hxD_x)^2-x^2\Delta_Y+\frac{1}{2}[x^{-1/h},[h^2(xD_x)^2,x^{1/h}]]+R_1\\
&\qquad\qquad=
(hxD_x)^2-h^2x^2\Delta_Y-1+R_1,\\
&\im P_h=\frac{1}{2i}(x^{-1/h}[(hxD_x)^2,x^{1/h}]+[(hxD_x)^2,x^{1/h}]x^{-1/h})
+R_2=-2hxD_x+R_2,
\end{split}\end{equation*}
with $R_1\in h\Diff^2_{0,h}(X)+x\Diff^2_h(X)$, $R_2\in h\Diff^1_{0,h}(X)+
x\Diff^1_{0,h}(X)$. Moreover,
\begin{equation*}
i[\re P_h,\im P_h]=i[-h^2x^2\Delta_Y,-2hxD_x]+R_3=-4h^3x^2\Delta_Y+hR_3,
\end{equation*}
$R_3\in h\Diff^2_{0,h}(X)+x\Diff^2_{0,h}(X)$.
Thus,
\begin{equation*}
i[\re P_h,\im P_h]=h+4h\re P_h-h(\im P_h)^2+hR_4,
\end{equation*}
with $R_4$ having the same properties as $R_3$.

Now let $u_h=x^{-1/h}u\in\dCI(X)$, so $P_h u_h=0$ and
\begin{equation*}\begin{split}
0&=\|P_h u_h\|^2=\|\re P_h u_h\|^2+\|\im P_h u_h\|^2
+\langle i[\re P_h,\im P_h]u_h,u_h\rangle\\
&=\|\re P_h u_h\|^2+(1-h)\|\im P_h u_h\|^2+h\|u_h\|^2
+4h\langle \re P_h u_h,u_h\rangle+h\langle R_4u_h,u_h\rangle.
\end{split}\end{equation*}
This is the analogue of Equations (4.2) and (4.3) of \cite{Vasy:Exponential},
except that here terms arising from
the commutator $i[\re P_h,\im P_h]$ do not have an
additional factor of $x$ compared to the first two squares on the
right hand side. The proof can be finished exactly as in
\cite{Vasy:Exponential}, writing $R_4=hR_5+x^{1/2}R_6x^{1/2}$,
$R_5,R_6\in \Diff^2_{0,h}(X)$, and noting that $-\re P_h+(\im P_h)^2$
is elliptic second order, so
\begin{equation*}\begin{split}
&|\langle hR_5u_h,u_h\rangle|\leq
Ch(\|\re P_h u_h\|\,\|u_h\|+\|\im P_h u_h\|^2+\|u_h\|^2),\\
&|\langle x^{1/2}R_6 x^{1/2}u_h,u_h\rangle|\\
&\qquad\qquad\leq
Ch(\|\re P_h x^{1/2}u_h\|\,\|x^{1/2}u_h\|+\|\im P_h x^{1/2}u_h\|^2
+\|x^{1/2}u_h\|^2)\\
&\qquad\qquad\leq
C'h(\|\re P_h u_h\|\,\|x^{1/2}u_h\|+\|\im P_h u_h\|^2
+\|x^{1/2}u_h\|^2).
\end{split}\end{equation*}

Indeed, for $\delta>0$ one writes
\begin{equation*}
\|x^{1/2}u_h\|^2
=\|x^{1/2}u_h\|^2_{x\leq\delta}+\|x^{1/2}u_h\|^2_{x\geq\delta}
\leq\delta\|u_h\|^2+\delta^{1-2/h}\|u\|^2,
\end{equation*}
so
\begin{equation*}\begin{split}
0\geq (1-C_1h)\|\re P_h u\|^2+(1-C_2 h)\|\im P_h u\|^2&+h(1-C_3h
-C_4\delta)\|u_h\|^2\\
&-C_5\delta^{1-2/h}\|u\|^2.
\end{split}\end{equation*}
Thus, there exists $h_0>0$ such that for $h\in(0,h_0)$,
\begin{equation*}
hC_5\delta^{1-2/h}\|u\|^2\geq h(\frac{1}{2}-C_4\delta)\|u_h\|^2.
\end{equation*}

Suppose $\delta\in(0,\min(\frac{1}{4C_4},\frac{1}{h_0}))$ and
$\supp u\cap\{x\leq\frac{\delta}{4}\}$ is non-empty. Then
$\|u_h\|^2\geq C_6 (\delta/4)^{-2/h}$ with $C_6>0$. Thus,
\begin{equation*}
C_5\delta\|u\|^2\geq \frac{C_6}{4}4^{2/h}.
\end{equation*}
As the right hand side goes to $+\infty$ as $h\to 0$, this provides
a contradiction.

Thus, $u$ vanishes for $x\leq\delta/4$, and then the usual hyperbolic
uniqueness (well-posedness of the non-characteristic Cauchy problem)
gives that it vanishes on $X$.
\end{proof}

Combined with Proposition~\ref{prop:solvable-up-to-compact} this gives:

\begin{thm}\label{thm:global-solvability}
Suppose that $\lambda\in\Real$,
$l_+>\max(\frac{1}{2},l(\lambda))$, $l_-<-\max(\frac{1}{2},l(\lambda))$.
Then
for $f\in H_0^{0,l_+,l_-}(X)$,
$Pu=f$ has a unique solution $u\in H_0^{1,l_+,l_-}(X)$.
\end{thm}

\begin{proof}
With the notation of Proposition~\ref{prop:solvable-up-to-compact},
we want to prove $N_{l+,l_-}=\{0\}$. But for $v\in N_{l_+,l_-}$,
by Corollary~\ref{cor:decay-Schwartz}, $v$ is Schwartz at $Y_-$.
Thus, by Proposition~\ref{prop:unique}, $v=0$. Thus,
by Proposition~\ref{prop:solvable-up-to-compact},
the required $u$ exists.

Conversely, if $u\in H_0^{1,l_+,l_-}(X)$ and $Pu=0$ then
by Corollary~\ref{cor:decay-Schwartz}, $u$ is Schwartz at $Y_+$,
so by Proposition~\ref{prop:unique}, $u=0$.
\end{proof}

We also deduce:

\begin{thm}\label{thm:smooth-solns}
Suppose $\lambda\neq\frac{(n-1)^2}{4}$.
Given $g_\pm\in\CI(Y_+)$
there exists a unique $u\in\CI(X^\circ)$ such that $Pu=0$
and which is of the form
\begin{equation*}\begin{split}
&u=x^{s_+(\lambda)}v_++x^{s_-(\lambda)}v_-,\ v_\pm|_{Y_+}=g_\pm,
\ v_+\in\CI(X),\\
&\qquad v_--\sum_{j=0}^{s_+(\lambda)-s_-(\lambda)-1}
a_j x^j\in x^{s_+(\lambda)-s_-(\lambda)}\log x\,\CI(X),\ a_j\in\CI(Y_\pm).
\end{split}\end{equation*}
If $s_+(\lambda)-s_-(\lambda)$ is not an integer, then
$v_-\in\CI(X)$.

On the other hand, if $\lambda=\frac{(n-1)^2}{4}$, then
given $g_\pm\in\CI(Y_+)$
there exists a unique $u\in\CI(X^\circ)$ such that $Pu=0$
and which is of the form
\begin{equation*}
u=x^{(n-1)/2}v_++x^{(n-1)/2}\log x\,v_-,\ v_\pm|_{Y_+}=g_\pm,\ v_\pm\in\CI(X).
\end{equation*}
\end{thm}

\begin{proof}
Suppose $\lambda\neq\frac{(n-1)^2}{4}$.
As shown in Lemma~\ref{lemma:expansion}, there exists $u_0$ supported
near $Y_+$ and of the desired
form there, such that $Pu_0\in\dCI(X)$. By
Theorem~\ref{thm:global-solvability}, for any
$l_+>\max(\frac{1}{2},l(\lambda))$ and $l_-<-\max(\frac{1}{2},l(\lambda))$
there exists a unique
$u_1\in H_0^{1,l_+,l_-}(X)$ such that $Pu_1=-Pu_0\in\dCI(X)$. As $l_\pm$
are arbitrary subject to the constraints, and $u_1$ is unique,
$u_1\in H_0^{1,l_+,l_-}(X)$ {\em for all} $l_+>\max(\frac{1}{2},l(\lambda))$
$l_-<-\max(\frac{1}{2},l(\lambda))$.
By Corollary~\ref{cor:decay-Schwartz}, $u_1$
is Schwartz at $Y_+$. Thus, $u=u_0+u_1$ satisfies $Pu=0$, and is
smooth near $Y_+$, so by the standard propagation of singularities $u\in
\CI(X^\circ)$. As $u\in H_0^{1,l_-}(X)$ for all
$l_-<-\max(\frac{1}{2},l(\lambda))$ near $Y_-$,
Corollary~\ref{cor:edge-reg} gives
$u\in H_0^{\infty,l_-}(X)$ for all such $l_-$.
By Proposition~\ref{prop:conormal-reg} and the remark
following it, $u\in x^{l_-+\frac{n-1}{2}}
H^\infty_b(X)=\bcon^{l_-+\frac{n-1}{2}}(X)$ for all such $l_-$.
Thus, by Lemma~\ref{lemma:expansion} it has the stated form near $Y_-$.

Conversely, if $u$ has the stated properties and $g_\pm=0$,
then $v_\pm$ are Schwartz at $Y_+$ by Lemma~\ref{lemma:expansion},
so $u$ is Schwartz at $Y_+$. Then $u=0$ by Proposition~\ref{prop:unique}.

If $\lambda=\frac{(n-1)^2}{4}$, the same argument, but using
Lemma~\ref{lemma:expansion-threshold} instead of Lemma~\ref{lemma:expansion},
completes the proof of the theorem.
\end{proof}

\section{The Cauchy problem}
We now consider global solutions for the Cauchy problem
posed near $Y_\pm$.

Let $T$ be a compactified time function, as in the introduction.
For any constant
$t_0\in(-1,1)$, and a vector field $V$ transversal to $S_{t_0}$,
$P$ is strictly hyperbolic, and
the Cauchy problem
\begin{equation}\begin{split}\label{eq:Cauchy-problem}
&Pu=0\ \text{in}\ X^\circ,\\
&u|_{S_{t_0}}=\psi_0,\\
&Vu|_{S_{t_0}}=\psi_1,
\end{split}\end{equation}
$\psi_0,\psi_1\in\Cinf(S_{t_0})$ is well posed.

\begin{thm}\label{thm:Cauchy-exist}
Let  $s_\pm(\lambda)=\frac{n-1}{2}\pm\sqrt{\frac{(n-1)^2}{4}-\lambda}$.
Assuming (A1) and (A2),
the solution $u$ of the Cauchy problem \eqref{eq:Cauchy-problem}
has the form
\begin{equation}\label{eq:asymp-exp-2}
u=x^{s_+(\lambda)}v_++x^{s_-(\lambda)}v_-,\ v_\pm\in\CI(X),
\end{equation}
if $s_+(\lambda)-s_-(\lambda)=2\sqrt{\frac{(n-1)^2}{4}-\lambda}$
is not an integer.
If $s_+(\lambda)-s_-(\lambda)$
is an integer, the same conclusion holds if we replace
$v_-\in\CI(X)$ by $v_-\in\CI(X)
+x^{s_+(\lambda)-s_-(\lambda)}\log x\,\CI(X)$.
\end{thm}

\begin{proof}
As $P$ is strictly hyperbolic with respect to $S_{t_0}$,
\cite[Theorem~23.2.4]{Hor} guarantees the existence of $u_0\in\CI_c(X^\circ)$
with $Pu_0=0$ in a neighborhood of $S_{t_0}$ and having the required
Cauchy data. We may choose
$t_1<t_0<t_2$ so
that $Pu_0=0$ for $T\in(t_1,t_2)$. Let $\chi_1,\chi_2\in\CI(X)$ be
such that $\chi_1\equiv 1$ in a neighborhood of $T\geq t_0$, $\chi_1$
is supported in $T>t_1$, while $\chi_2\equiv 1$ in a neighborhood
of $T\leq t_0$, supported in $T<t_2$. In particular, $\chi_1\chi_2$
is supported where $T\in(t_1,t_2)$, is identically $1$ near
$S_{t_0}$, and each $\chi_i$ is identically $1$ on the support of
the $d\chi_j$, $j\neq i$. Then $P(\chi_1\chi_2 u_0)=[P,\chi_1] u_0
+[P,\chi_2]u_0$. Denoting these two terms by $f_1$, resp.\ $f_2$,
we use Theorem~\ref{thm:global-solvability} to solve away $f_1$ towards
$Y_+$ and $f_2$ towards $Y_-$ so that the Cauchy data are unchanged.

First, by Theorem~\ref{thm:global-solvability}, with any $l>
\max(\frac{1}{2},l(\lambda))$, there exists $u_2\in H^{1,l,-l}_0(X)$
such that $Pu_2=f_2$. By Corollary~\ref{cor:decay-Schwartz},
$u_2$ is Schwartz at $Y_+$, and then by Proposition~\ref{prop:unique},
$u_2\equiv 0$ near $Y_+$. Hyperbolic propagation then shows that $\supp u_2
\subset\{T> t_0\}$ as $f_2$ is supported in this set, so $u_2\equiv 0$
near $S_{t_0}$. In addition, as in the argument of
Theorem~\ref{thm:smooth-solns} we deduce that $u_2\in\CI(X^\circ)$
has an expansion as in Theorem~\ref{thm:smooth-solns}.

Interchanging the weights at $Y_\pm$,
we can similarly show the existence of $u_1\in H^{1,-l,l}_0(X)$
such that $Pu_1=f_1$, $\supp u_1\subset\{T<t_0\}$, and $u_1$
having an expansion at $Y_+$. Thus,
$u=\chi_1\chi_2 u_0-u_1-u_2\in\CI(X^\circ)$ satisfies
$Pu=0$, $u|_{S_{t_0}}=\psi_0$, $Vu|_{S_{t_0}}=\psi_1$,
and $u$ has an asymptotic expansion as in Theorem~\ref{thm:smooth-solns},
proving the existence part.

Uniqueness follows easily, for if $u$ solves the Cauchy problem
with $\psi_0=0$, $\psi_1=0$, then $u=0$ near $S_{t_0}$, hence
vanishes globally.
\end{proof}

It is useful to relate the Cauchy data at different hypersurfaces to
each other, particularly for hypersurfaces near $Y_+$, resp,\ $Y_-$.
This is very easy using the standard FIO result. We renormalize
this operator in order to make all entries in the FIO matrix have the
same order. Namely, let $\Delta_{t_j}$ be the Laplacian of the restriction
of $g$ to $S_{t_j}$, $j=1,2$, so $\Delta_{t_j}\geq 0$ as $S_{t_j}$ is
space like. Let $\Delta_{t_j}'$ denote the operator which is
$\Delta_{t_j}$ on the orthocomplement of the nullspace of $\Delta_{t_j}$
and is the identity on the nullspace, so $\Delta_{t_j}'$ is positive
and invertible.

\begin{prop}\label{prop:FIO-int}(\cite{FIOII})
For any $t_1,t_2\in(-1,1)$,
the map $C_{t_1,t_2}$ sending Cauchy data of global smooth
solutions of $Pu=0$
at $S_{t_1}$ to Cauchy data at $S_{t_2}$:
\begin{equation*}
C_{t_1,t_2}:((\Delta'_{t_1})^{1/2}u|_{S_{t_1}},Vu|_{S_{t_1}})
\mapsto((\Delta'_{t_2})^{1/2}u|_{S_{t_2}},
V u|_{S_{t_2}})
\end{equation*}
is an invertible Fourier integral operator of order $0$
corresponding to the bicharacteristic flow.
\end{prop}

\section{The scattering operator}\label{sec:scattering}
In order to prove that the scattering operator is a Fourier integral operator,
we construct a parametrix as a conormal distribution on a resolution
of $X\times Y_+$ for the solution operator,
also called  the `Poisson operator',
$(g_+,g_-)\mapsto u$ with
notation as in
\eqref{eq:asymp-exp} and \eqref{eq:v+-v--spec}.

Near $Y_+$, this can be done by considering
$[X\times Y_+;\diag_{Y_+}]$. On this space the parametrix is a conormal
distribution near $Y_+$ associated to the `flowout' of points in $Y_+$.
That is, for $q'\in Y_+$, consider the bicharacteristics approaching
$\zS^*_{q'}X$. These form a Lagrangian submanifold of $T^*X^\circ$,
which near $Y_+$ has constant rank projection (since the rank at the front
face is maximal, namely $n-1$), and is thus the conormal
bundle of a submanifold $F_{q'}$ of $X$. These $F_{q'}$ depend smoothly
on $q'$ so that $F=\cup_{q'}F_{q'}\times\{q'\}$ is a smooth submanifold
of $X^\circ\times Y_+$, and indeed it extends to be smooth
to $[X\times Y_+;\diag_{Y_+}]$.

In order to orient ourselves, we first make some remarks regarding
distributions conormal to $F$. First, recall that if $M$ is a manifold
with corners of dimension $m$, and $Z$ is an interior p-submanifold,
$I^{p}(M,Z)$ is the space of distributions on $M$ conormal to $Z$,
see \cite{RBMDiff,RBMCalcCon}. Here we only need the case where $Z$ meets
all boundary faces transversally; in fact, in this case, $Z$ only meets
a (codimension one) boundary hypersurface. Thus,
in local coordinates $(x,y)$, $x=(x_1,\ldots,x_k)$, $y=(y_1,\ldots,y_{m-k})$
in which $M$ is locally given by $x_j\geq 0$ for all $j$, and $Z$
is given by $y_1=\ldots=y_N=0$, elements of
$I^p(M,Z)$ have the form 
\begin{equation*}
(2\pi)^{-(m+2N)/4}\int_{\RR^N} e^{iy'\cdot\xi}\,a(x,y,\xi)\,d\xi,
\end{equation*}
with $a\in S^{p+(m-2N)/4}(M;\RR^N)$, $y'=(y_1,\ldots,y_N)$. Note that
$x$ behaves as a parameter, i.e.\ the presence of boundaries does not cause
any complications, hence the standard treatment in the boundaryless case
\cite{FIO1,Hor} actually suffices. Note that if $A\in\Diff^r(M)$ and
$u\in I^p(M,Z)$ then $Au\in I^{p+m}(M,Z)$, and if $A$ is characteristic
in $Z$, i.e.\ its principal symbol vanishes on $N^*Z$, then
$Au\in I^{p+m-1}(M,Z)$, with $\sigma_{p+m-1}(Au)=H_a \sigma_m(u)+bu$,
where $b$ depends on $A$ only. This equation is an ODE along the
bicharacteristics of $A$, and is called a transport equation.

We also need to allow weights, i.e.\ consider the spaces
$x^s I^p(M,Z)$. $\Diff(M;Z)$ is not well-behaved on these spaces
(because of derivatives possibly falling on $x^s$) but
$\Diffb(M)$ is.

\begin{lemma}(see \cite{Hor}[Section~18.2] and \cite{RBMDiff})
Suppose that $A\in\Diffb^m(M)$. Then
\begin{equation}\label{eq:A-x^s}
A:x^sI^p(M,Z)\to x^s I^{p+m}(M,Z).
\end{equation}
If $A$ is characteristic on $Z$, then
\begin{equation}\label{eq:A-x^s-char}
A:x^sI^p(M,Z)\to x^s I^{p+m-1}(M,Z),
\end{equation}
and there is function $b$ depending on $A$ only such that
$\sigma_{p+m-1}(Au)=H_a \sigma_m(u)+bu$.
\end{lemma}

\begin{proof}
As $x^{-s}Ax^s\in\Diffb^m(M)\subset\Diff^m(M)$,
\eqref{eq:A-x^s} follows immediately from the remarks above.
Next, if $A$ is characteristic on $Z$, then so is $x^{-s}Ax^s$,
so the remarks above prove \eqref{eq:A-x^s-char}. As the principal
symbol of $x^{-s}Ax^s$ is the same as that of $A$,
$\sigma_{p+m-1}(Au)=H_a \sigma_m(u)+bu$ follows.
\end{proof}

In our case, $M=[X\times Y_+;\diag_{Y_+}]$, and $Z=F$.
The transport equation will allow us to solve away errors modulo smooth
terms in our
construction of the `Poisson operator', $(g_+,g_-)\mapsto u$.
However, we need to see
first what the `errors' are errors of, i.e.\ where the Schwartz kernel
of the Poisson operator comes from, which will also give a relationship
between the orders $s$ and $p$ above.

Even for arbitrary $Y$, 
the model on the front face is the same as when $Y$ is Euclidean
space with a translation-invariant metric. Let $y$ denote
local coordinates on $Y$, as well as their extension to $X$, so $(x,y)$
are local coordinates on $X$. On $X\times Y_+$ then we have local
coordinates $(x,y,y')$, where $y'$ is the pull-back of $y$ from the
second factor. (The pull-back of $y$ from the first factor, $X$, is
still denoted by $y$.)
Using projective coordinates
\begin{equation*}
X=x,\ Y=\frac{y-y'}{x},\ y',
\end{equation*}
the $P=\Box-\lambda$ becomes
\begin{equation*}
(XD_X-YD_Y+i(n-1))(XD_X-YD_Y)-\sum_{i,j}h_{ij}(y')D_{Y_i}D_{Y_j}-\lambda,
\end{equation*}
modulo $X\Diffb^2([X\times Y_+;\diag_{Y_+}])$.
To analyze this operator for fixed $y'$, we may arrange that $h_{ij}(y')
=\delta_{ij}$, so the operator becomes
\begin{equation}\label{eq:P-blown-up}
(XD_X-YD_Y+i(n-1))(XD_X-YD_Y)-\Delta_Y-\lambda.
\end{equation}
When acting on functions of the form $u=x^s v$, $v$ a function of $Y$,
$XD_X$ becomes a multiplication operator, and the operator we arrive at
after this substitution is a degenerate PDE with radial points over
$|Y|=1$, i.e.\ where $F$ hits the front face. This is indeed what enables
us to find solutions supported in $|Y|\leq 1$, with singularities
carried away by $F$.

While this form is helpful in seeing the big picture, we need to solve
this exactly at $X=0$ to leading order, for which it is useful to
view $\Box$ on the warped product model as the analytic continuation
of the Laplacian on hyperbolic space, which is arrived at by complex
rotation in $x$ (replacing $x$ by $ix$),
i.e.\ considering the Laplacian of $\frac{dx^2+h}{x^2}$.
Correspondingly, the explicit solutions we are
interested in are analytic continuations of the Eisenstein functions
(Poisson kernel) on hyperbolic space, i.e.\ they take the form
\begin{equation*}
X^s(|Y|^2-1\pm i0)^s,\ -s(n+s-1)=\lambda.
\end{equation*}
Note that these values of $s$ are different from the usual indicial roots;
these give
\begin{equation*}
s=\hat s_\pm(\lambda)
=-\frac{n-1}{2}\pm\sqrt{\left(\frac{n-1}{2}\right)^2-\lambda}
=s_\pm(\lambda)-(n-1).
\end{equation*}
We in fact have two interesting solutions corresponding
to branches of the analytic continuation. As we are interested in
solutions supported inside $|Y|\leq 1$, we take their difference,
\begin{equation*}
X^s[(|Y|^2-1+ i0)^s-(|Y|^2-1-i0)^s]=c_s X^s(|Y|^2-1)^s_-,
\end{equation*}
with $c_s=e^{i\pi s}-e^{-i\pi s}$ if $s$ is not a negative integer, and
\begin{equation*}
X^s[(|Y|^2-1+ i0)^s-(|Y|^2-1-i0)^s]=c_s X^s\delta_0^{(-s-1)}(|Y|^2-1),
\end{equation*}
with $c_s=\frac{2\pi i(-1)^{-s}}{(-s-1)!}$ if $s$ is a negative integer.
Here the notation is that if $f$ is a distribution on $\Real$ which
is conormal to the origin, then $f(|Y|^2-1)$ denotes $T^*f$, where
$T:\Real^{n-1}\to\Real$ is the map $T(Y)=|Y|^2-1$. The preimage of the
origin under $T$ is the unit sphere, and on the unit sphere the
differential of $T$ is surjective, so the pull-back of these conormal
distributions indeed makes sense.

If the boundary is actually Euclidean, then near $Y_+\times Y_+$ we
thus obtain an exact solution with singularities on $F$,
\begin{equation*}
E_{0,\pm}(x,y,y',\lambda)=C_s x^s(1-\frac{|y-y'|^2}{x^2})^s_+,
\end{equation*}
with $C_s$ to be determined and $s=\hat s_\pm(\lambda)$, if $s$ is not
a negative integer, and
\begin{equation*}
E_{0,\pm}(x,y,y',\lambda)=C_s x^s\delta_0^{(-s-1)}(1-\frac{|y-y'|^2}{x^2})
\end{equation*}
if $s$ is a negative integer.
Note that for each $\lambda$,
\begin{equation*}
E_{0,\pm}=E_{0,\pm}(\lambda)
\in x^s I^{m(s)}([X\times Y_+;\diag_{Y_+}],F),\ m(s)=-s-\frac{2n+1}{4},
\ s=\hat s_\pm(\lambda).
\end{equation*}

\begin{lemma}\label{lemma:Poisson-leading}
Suppose that $\hat s_\pm(\lambda)\nin -\frac{n-1}{2}-\Nat_+$.
Then there is a constant $C_s\neq 0$ such that for all $\phi\in\CI(Y_+)$
the operator $E_{0,\pm}(\lambda)$ with Schwartz kernel $E_\pm\,dh$:
\begin{equation*}
E_{0,\pm}\phi=\int E_{0,\pm}(x,y,y',\lambda)\phi(y')\,dh(y')
\end{equation*}
satisfies
\begin{equation*}
E_{0,\pm}\phi=x^{s(\lambda)}v,\ v\in\CI(X),\ v|_{Y_+}=\phi.
\end{equation*}
\end{lemma}

\begin{rem}
Note that for $\lambda>\frac{(n-1)^2}{4}-1$, the condition
$\hat s_\pm(\lambda)\nin -\frac{n-1}{2}-\Nat_+$ automatically holds.
For $\Box$ itself (i.e.\ $\lambda=0$) the condition holds if $n$ is even.
In addition, the condition always holds for {\em one} of the two
indicial roots, namely the larger one (i.e.\ the one with more decay/less
growth at $Y_+$).
\end{rem}

\begin{proof}
Suppose first that $\hat s_\pm(\lambda)$ is not a negative integer.

Changing variables in the integral we deduce that
for $\phi\in\Cinf_c(Y_+)$, and $s=\hat s_\pm(\lambda)$ still,
\begin{equation*}\begin{split}
&\int E_{0,\pm}(x,y,y',\lambda)\phi(y')\,dy'=x^{n-1+\hat s_\pm(\lambda)}
\int  (1-|Y|^2)^s_+\phi(y-xY)\,dY\\
&\qquad
=x^{s_\pm(\lambda)}v,\ v\in\CI(X),\ v(0,y)=C_s((1-|Y|^2)^s_+,1)\phi(y),
\end{split}\end{equation*}
where the second factor in the expression for $v(0,y)$ is the evaluation of the
distribution $(1-|Y|^2)^s_+$ on $1$, and where we used that $s_\pm(\lambda)
=\hat s_\pm(\lambda)+(n-1)$. We need to check for which values of $s$
does $C_s$ vanish, so we compute this pairing.

For $\re s>-1$, the distributional pairing
is an absolutely convergent integral, which in polar coordinates
becomes
\begin{equation*}
c_{n-2}
\int (1-\rho^2)^s \rho^{n-1}\,d\rho=\frac{c_{n-2}}{2} B(\frac{n-1}{2},s+1)
=\frac{c_{n-2}\Gamma(\frac{n-1}{2})\Gamma(s+1)}{2\Gamma(\frac{n-1}{2}+s+1)},
\end{equation*}
where $c_{n-2}$ is the volume of the $(n-2)$-sphere and $B$ is the
beta-function.
As both the distributional pairing and the $\Gamma$ function are
meromorphic in $s$ (indeed analytic away from $-\Nat$), we deduce that
\begin{equation*}
((1-|Y|^2)^s_+,1)
=\frac{c_{n-2}\Gamma(\frac{n-1}{2})\Gamma(s+1)}{2\Gamma(\frac{n-1}{2}+s+1)}
\end{equation*}
for all $s$ which are not negative integers. This vanishes only
if $s\in-\frac{n-1}{2}-\Nat_+$ and $n$ is even
(so $s$ is not a negative integer).

If $s=\hat s_\pm(\lambda)$ is a negative integer, say $s=-k$,
\begin{equation*}\begin{split}
&\int E_{0,\pm}(x,y,y',\lambda)\phi(y')\,dy'=x^{n-1+\hat s_\pm(\lambda)}
\int  (1-|Y|^2)^s_+\phi(y-xY)\,dY\\
&\qquad
=x^{s_\pm(\lambda)}v,\ v\in\CI(X),\ v(0,y)=C_s(\delta_0^{(-s-1)}(1-|Y|^2),1)
\phi(y).
\end{split}\end{equation*}
The distributional pairing now becomes
\begin{equation*}
\frac{c_{n-2}}{2}(\delta_0^{(k-1)}(z),(1-z)^{(n-3)/2})
=\frac{c_{n-2}}{2}\frac{d^{k-1}}{(dz)^{k-1}}(1-z)^{(n-3)/2}|_{z=0}.
\end{equation*}
If $n$ is even,
all derivatives of $(1-z)^{(n-3)/2}$ at $z=0$ are non-zero, while
if $n$ is odd, the derivatives of order $<\frac{n-1}{2}$ are non-zero,
so this pairing vanishes only if $s=-k\in-\frac{n-1}{2}-\Nat_+$.

Combining these two cases, $s=\hat s_\pm(\lambda)\nin-\frac{n-1}{2}-\Nat_+$
implies that the respective distributional pairings are non-zero. Letting
$C_s$ to be their reciprocal yield $E_{0,\pm}(\lambda)$ satisfying the
lemma.
\end{proof}

If the metric is not exact warped product, then $E_{0,\pm}$ will play the role
of the model at the front face of $[X\times Y_+;\diag_{Y_+}]$,
which then will need to be `extended' into the interior. First, let
$\cY:Y_+\times Y_+\to\RR^{n-1}$ be local coordinates on the first
factor of $Y_+$ centered at the diagonal so that at the diagonal,
the metric $h$ lifted from the first factor is the standard Euclidean
metric $d\cY^2$. That is, informally, $\cY=\cY(y')$
is a family of local coordinates
on $Y_+$, parameterized by $y'\in Y_+$, so that for fixed $y'$, $\cY(y')$
gives local coordinates centered at $y'$ in which $h$ is $d\cY^2$ at the
center, $\cY(y')=0$. Thus, with the notation considered above in the
Euclidean setting, we can take
$\cY=y-y'$. Let $Y=\frac{\cY(y')}{x}$, so $(x,Y,y')$ form
a local coordinate system in a neighborhood of the interior
of the front face of $[X\times Y_+;\diag_{Y_+}]$.

As $F$ is a $\CI$ codimension $1$
submanifold of $[X\times Y_+;\diag_{Y_+}]$ transversal to the
front face, intersecting it in the sphere $|Y|=1$,
there exists a $\CI$ function
$\rho$ on $[X\times Y_+;\diag_{Y_+}]$ such that $\rho$ defines $F$
(i.e.\ $\rho$ vanishes exactly on $F$, and $d\rho$ does not vanish there),
and $\rho|_{\ff}=1-|Y|^2$. We let $r\geq 0$ be defined by
$r=(1-\rho)^{1/2}$, so $r=|Y|$ at $\ff$, and for convenience
we often write (slightly imprecisely) $(1-r^2)^s_+$, etc.,
for $\rho^s_+$. Our model is then
\begin{equation*}
E_{0,\pm}(x,y,y',\lambda)=C_s x^s(1-r^2)^s_+=C_s x^s \rho^s_+,
\end{equation*}
if $s=\hat s_\pm(\lambda)$ is not
a negative integer, and
\begin{equation*}
E_{0,\pm}(x,y,y',\lambda)=C_s x^s\delta_0^{(-s-1)}(1-r^2)
=C_s x^s\delta_0^{(-s-1)}(\rho)
\end{equation*}
if $s$ is a negative integer, with $C_s$ as in
Lemma~\ref{lemma:Poisson-leading}.

Then we want to find
\begin{equation*}
E_\pm\in x^s I^{m(s)}([X\times Y_+;\diag_{Y_+}],F),\ m(s)=-s-\frac{2n+1}{4},
\ s=\hat s_\pm(\lambda),
\end{equation*}
with $PE_\pm=0$,
$E_\pm-E_{0,\pm}\in x^{s+1}I^{m(s)}([X\times Y_+;\diag_{Y_+}],F)$, and $E_\pm$
vanishing to infinite order off the front face.
The equation $PE_\pm\in\dCI(X\times Y_+)$
becomes a degenerate
transport equation at the level of principal symbols and can
be solved to leading order.
In fact, in order to
simplify the transport equation, which is an equation for the principal
symbol of $E_\pm$, given by an ODE along the Lagrangian, $N^*F$, it is
convenient to notice that we want
\begin{equation*}\begin{split}
E_\pm=ax^s(1-r^2)^s_++E'_\pm,&\ a\in\CI([X\times Y_+;\diag_{Y_+}]),\\
&\ E'_\pm
\in x^s I^{m(s)-1+\ep}([X\times Y_+;\diag_{Y_+}],F),
\end{split}\end{equation*}
$\ep\in(0,1)$ arbitrarily small, so the principal symbol of $E_\pm$ can
be identified with $a|_F$, and the transport equation is an ODE for $a|_F$.
Namely,
\begin{equation*}
PE_\pm=(Qa)x^s(1-r^2)^{s-1}_++\tilde E_\pm,\ \tilde E_\pm
\in x^s I^{m(s)+\ep}([X\times Y_+;\diag_{Y_+}],F),
\end{equation*}
where $Q$ is a first order differential operator of the form $Q=xV+b$,
$V$ a vector field tangent to $F$ transversal to $\pa F$ -- $xV(q)$ is
a non-vanishing multiple of the
push-forward of the Hamilton vector field $H_p$ evaluated at the
one-dimensional space $N^*_q F_{q'}\setminus 0$. (This vector field
is homogeneous, so the choice of
$\alpha\in N^*_q F_{q'}$ only changes the push forward by a non-vanishing
factor.)

Solving the transport equation
and iterating the construction gives a new $E_\pm\in
x^s I^{m(s)}([X\times Y_+;\diag_{Y_+}],F)$ vanishing to infinite order off
the front face with
$PE_\pm\in x^{s+1}\CI([X\times Y_+;\diag_{Y_+}])$; we show this
in Proposition~\ref{prop:transport} below. In fact, we can do better:
we can ensure that near $Y_+$ (where this makes sense)
$E_\pm$ is supported {\em in the interior of the light cone};
this is important as we show momentarily.

In order to remove the leading term at the front face
(i.e.\ to improve the error, $PE_\pm$,
to $x^{s+2}\CI([X\times Y_+;\diag_{Y_+}])$, which can then
be further iterated away), we need to study
$P$ acting on functions of the form $x^\sigma v$,
$v\in\CI([X\times Y_+;\diag_{Y_+}])$, modulo
$x^{\sigma+1}\CI([X\times Y_+;\diag_{Y_+}])$. This only
uses the model at $\ff$.
But \eqref{eq:P-blown-up} gives
\begin{equation*}
x^{-\sigma}Px^{\sigma}v=P_\sigma v,
\ P_\sigma=(YD_Y-i(n-1-\sigma))(YD_Y+i\sigma)
-\Delta_Y-\lambda,
\end{equation*}
with $P_\sigma$ on operator on Euclidean space identified with the fiber
of the front face over $y'$. This is of course a differential
operator with smooth coefficients, but it is not elliptic. To see its
precise behavior, it is convenient to introduce polar coordinates $(r,\omega)$
in $Y$. (This agrees with our preceeding definition of $r$ at the front
face.) In such coordinates,
\begin{equation*}
P_\sigma=(rD_r-i(n-1-\sigma))(rD_r+i\sigma)-D_r^2+i\,\frac{n-2}{r}\,D_r
-\frac{1}{r^2}\,\Delta_\omega-\lambda,
\end{equation*}
with $\Delta_\omega$ the positive Laplacian on the standard $(n-2)$-sphere.
The principal symbol of $P_\sigma$ is $(r^2-1)|\xi|^2-r^{-2}|\eta|^2_{\omega}$,
with $(\xi,\eta)$ denoting the dual variables of $(r,\omega)$.
Thus, $P_\sigma$ is elliptic for $r<1$, i.e.\ {\em inside the light cone}.
A straightforward calculation
shows that $P_\sigma$ is microhyperbolic for $r>1$; it has some radial points
at $r=1$. There are two slightly different (but related) aspects of
$P_\sigma$ to address: the solvability of the transport equations, i.e.\ the
removability of singularities at $r=1$, and the solvability of smooth
terms.

We start with the transport equations. It is convenient to consider the
conjugate $(1-r^2)^{-s}P_\sigma(1-r^2)^s$, more precisely, in view of
the singularity of the conjugating factor,
$(1-r^2\pm i0)^{-s}P_\sigma(1-r^2\pm i0)^s$, considered on all of the
front face, i.e.\ as an operator from $\CI(\ff)$ to $\dist(\ff)$.
The following lemma is the result of a straightforward
calculation when replacing $\pm i0$ by $\pm i\ep$, and the lemma then follows
by taking the limit.

\begin{lemma}\label{lemma:P_sigma-conj}
For all $s\in\Real$, $P_\sigma$ satisfies
\begin{equation*}\begin{split}
(1-r^2\pm i0)^{-s}&P_\sigma(1-r^2\pm i0)^s\\
&=
4s(s-\sigma)(1-r^2\pm i0)^{-1}+(P_\sigma-4s(r\pa_r+s-\sigma+\frac{n-1}{2}))\\
&=
4s(s-\sigma)(1-r^2\pm i0)^{-1}+P_{\sigma-2s}
\end{split}\end{equation*}
as operators from $\CI(\ff)$ to $\dist(\ff)$.
\end{lemma}

We in fact always need logarithmic terms to solve away singularities
because there are automatic integer coincidences between the powers of $x$
we need in the Taylor series, i.e.\ $\sigma$, and the orders of the
singularities along $F$, i.e.\ $s$.

\begin{lemma}\label{lemma:P_sigma-conj-log}
For all $s\in\Real$, $k\in\Nat$, $P_\sigma$ satisfies
\begin{equation}\begin{split}\label{eq:P_sigma-log}
P_\sigma&(1-r^2\pm i0)^s\log(1-r^2\pm i0)^k\\
=&
4k(2s-\sigma)(1-r^2\pm i0)^{s-1}\log(1-r^2\pm i0)^{k-1}\\
&+4s(s-\sigma)(1-r^2\pm i0)^{s-1}\log(1-r^2\pm i0)^k\\
&+(1-r^2\pm i0)^s\log(1-r^2\pm i0)^k P_{\sigma-2s}\\
&+\sum_{j=0}^{k-2}(1-r^2\pm i0)^{s-1}\log(1-r^2\pm i0)^j Q_j\\
&+(1-r^2\pm i0)^s\log(1-r^2\pm i0)^{k-1} Q_{k-1},
\end{split}\end{equation}
as operators from $\CI(\ff)$ to $\dist(\ff)$, where the $Q_j$, $j=0,\ldots,
k-1$ are first order differential operators with smooth coefficients
on $\ff$ (depending smoothly on $s,\sigma,k$).
\end{lemma}

\begin{rem}
The principal utility of allowing logarithmic singularities arises
if $s=\sigma$, in which case the second term on the right hand side
is missing, hence the first term can be used to remove error terms
with a lower power of logarithm (that could not be removed without
logarithms, i.e.\ by the preceeding lemma).
\end{rem}

\begin{proof}
The case $k=0$ follows from the preceeding lemma. We then proceed by induction.
If $k\geq 1$, and the result has been proved for $k$ replaced by $k-1$,
then for $a\in\CI(\ff)$,
\begin{equation*}
P_\sigma (1-r^2\pm i0)^s\log(1-r^2\pm i0)^k a
=
\frac{d}{ds}P_\sigma (1-r^2\pm i0)^s\log(1-r^2\pm i0)^{k-1} a
\end{equation*}
shows that we simply need to differentiate \eqref{eq:P_sigma-log} (with
$k-1$ in place of $k$) with respect to $s$. The only terms giving
rise to additional factors of logarithms are the ones in which
$(1-r^2\pm i0)^{s-1}$ or $(1-r^2\pm i0)^{s}$ is differentiated. As
we are applying the result with $k$ replaced by $k-1$, the last
two (residual) terms of \eqref{eq:P_sigma-log} (for $k-1$) give rise
to residual terms (for $k$). Also, the only term that is not negligible
even though it has a power of logarithm less than $k$ is the first one,
with a factor of $(1-r^2\pm i0)^{s-1}$. Thus, the first three terms
of \eqref{eq:P_sigma-log} (for $k-1$) will contribute to the last two
residual terms (for $k$) except when
$(1-r^2\pm i0)^{s-1}$ or $(1-r^2\pm i0)^{s}$ is differentiated, or
when the coefficient of the second term is differentiated. The latter
gives $4(2s-\sigma)(1-r^2\pm i0)^{s-1}\log(1-r^2\pm i0)^{k-1}$, so
altogether we have $4(2s-\sigma)+4(2s-\sigma)(k-1)=4k(2s-\sigma)$ of
$(1-r^2\pm i0)^{s-1}\log(1-r^2\pm i0)^{k-1}$, giving the desired result.
\end{proof}

\begin{cor}
If $s\neq -1$ and $s+1\neq\sigma$ then
for $b$ smooth function on $\sphere^{n-2}$, there exists a unique
smooth function $q$ on $\sphere^{n-2}$ such that
\begin{equation*}
P(x^\sigma (1-r^2)_+^{s+1} q)=x^\sigma (1-r^2)_+^s (b+(1-r^2)e),
\end{equation*}
holds near $\sphere^{n-2}$, with $e$ smooth near $\sphere^{n-2}$.

More generally, under the same assumptions on $s$,
for $b$ smooth function on $\sphere^{n-2}$, there exists a unique
smooth function $q$ on $\sphere^{n-2}$ such that
\begin{equation*}\begin{split}
&P(x^\sigma (1-r^2)_+^{s+1} \log(1-r^2)_+^k q)\\
&=x^\sigma (1-r^2)_+^s
(\log(1-r^2)_+^k b+\sum_{j=0}^{k-1}\log(1-r^2)_+^j e_j
+ (1-r^2)\log(1-r^2)_+^k e),
\end{split}\end{equation*}
holds near $\sphere^{n-2}$, with $e,e_j$ smooth near $\sphere^{n-2}$,
$j=0,1,\ldots,k-1$.

If $s=-1$ or $s+1=\sigma$, but $\sigma\neq 0$, then
for $b$ smooth function on $\sphere^{n-2}$, there exists a unique
smooth function $q$ on $\sphere^{n-2}$ such that
\begin{equation*}\begin{split}
&P(x^\sigma (1-r^2)_+^{s+1} \log(1-r^2)_+^{k+1} q)\\
&=x^\sigma (1-r^2)_+^s
(\log(1-r^2)_+^k b+\sum_{j=0}^{k}\log(1-r^2)_+^j e_j
+ (1-r^2)\log(1-r^2)_+^{k+1} e),
\end{split}\end{equation*}
holds near $\sphere^{n-2}$, with $e,e_j$ smooth near $\sphere^{n-2}$,
$j=0,1,\ldots,k$.
\end{cor}

\begin{proof}
In the first case,
let $q=4s^{-1}(s-\sigma)^{-1}b$, and apply Lemma~\ref{lemma:P_sigma-conj},
expressing $(1-r^2)_+^{s+1}$ as a difference of $(1-r^2\pm i0)^{s+1}$.
Uniqueness is clear.

In the second case, proceed the same way, applying
Lemma~\ref{lemma:P_sigma-conj-log}.
\end{proof}

\begin{prop}\label{prop:transport}
The transport equations can be solved near the front face, i.e.\ there
exist
\begin{equation*}
E_\pm\in x^s I^{m(s)}([X\times Y_+;\diag_{Y_+}],F),\ m(s)=-s-\frac{2n+1}{4},
\ s=\hat s_\pm(\lambda),
\end{equation*}
with
\begin{equation}\label{eq:E-E_0-ext}
E_\pm-E_{0,\pm}\in x^{s+1}I^{m(s)}([X\times Y_+;\diag_{Y_+}],F),
\end{equation}
$PE_\pm\in x^{s+1}\CI([X\times Y_+;\diag_{Y_+}])$,
and $E_\pm$
vanishing to infinite order off the front face.
\end{prop}

\begin{proof}
First, with any $E_\pm\in x^s I^{m(s)}([X\times Y_+;\diag_{Y_+}],F)$
extending $E_{0,\pm}$ in the sense of
\eqref{eq:E-E_0-ext}, having an expansion in terms
of $(1-r^2)^\beta_+$, so $E=x^{\hat s_\pm(\lambda)}(1-r^2)^{\hat s_\pm(\lambda)}_+a$,
$a$ smooth, $a|_{x=0}=1$,
one has
\begin{equation*}
PE_\pm=x^{\hat s_\pm(\lambda)+1}(1-r^2)^{\hat s_\pm(\lambda)-1}_+b,
\end{equation*}
$b$ smooth. By the corollary (if $\hat s_\pm(\lambda)\neq 0$), one can find
$E_{1,\pm}=x^{\hat s_\pm(\lambda)+1}(1-r^2)^{\hat s_\pm(\lambda)}_+b$ such that
\begin{equation*}
PE_{1,\pm}-PE_\pm\in x^{\hat s_\pm(\lambda)+1}(1-r^2)^{\hat s_\pm(\lambda)}_+\CI
+x^{\hat s_\pm(\lambda)+2}(1-r^2)^{\hat s_\pm(\lambda)-1}_+\CI,
\end{equation*}
so replacing $E_\pm$ by $E_\pm-E_{1,\pm}$,
one has an extension of $E_{0,\pm}$ of the same
form as the original $E_\pm$, but with
\begin{equation*}
PE_\pm\in x^{\hat s_\pm(\lambda)+1}(1-r^2)^{\hat s_\pm(\lambda)}_+\CI
+x^{\hat s_\pm(\lambda)+2}(1-r^2)^{\hat s_\pm(\lambda)-1}_+\CI.
\end{equation*}
Leaving the first term unchanged, one iterates the second
term away, using $E_{j,\pm}=x^{\hat s_\pm(\lambda)+j}(1-r^2)^{\hat s_\pm(\lambda)}_+b$
to remove errors in $x^{\hat s_\pm(\lambda)+j}(1-r^2)^{\hat s_\pm(\lambda)-1}_+\CI$,
with the result that the new $E_\pm$ satisfies
\begin{equation*}
PE_\pm\in x^{\hat s_\pm(\lambda)+1}(1-r^2)^{\hat s_\pm(\lambda)}_+\CI
+x^{\hat s_\pm(\lambda)+j+1}(1-r^2)^{\hat s_\pm(\lambda)-1}_+\CI.
\end{equation*}
Note that there is no obstacle for this procedure as long as $\hat s_\pm(\lambda)
\neq 0$. By an asymptotic summation argument one gets an $E$ with
\begin{equation*}
PE_\pm\in x^{\hat s_\pm(\lambda)+1}(1-r^2)^{\hat s_\pm(\lambda)}_+\CI
+(1-r^2)^{\hat s_\pm(\lambda)-1}_+\dCI.
\end{equation*}
For the last term the singular transport equations are now easily solvable,
so one obtains near the front face
\begin{equation*}
PE_\pm\in x^{\hat s_\pm(\lambda)+1}(1-r^2)^{\hat s_\pm(\lambda)}_+\CI.
\end{equation*}

Now using the corollary, we can find $E_1=x^{\hat s_\pm(\lambda)+1}
(1-r^2)_+^{\hat s_\pm(\lambda)+1}\log(1-r^2)_+$ such that
\begin{equation*}\begin{split}
PE_{1,\pm}-PE_\pm\in &x^{\hat s_\pm(\lambda)+1}(1-r^2)^{\hat s_\pm(\lambda)+1}_+\CI\\
&+x^{\hat s_\pm(\lambda)+1}(1-r^2)^{\hat s_\pm(\lambda)+1}_+\log(1-r^2)_+\CI\\
&+x^{\hat s_\pm(\lambda)+2}(1-r^2)^{\hat s_\pm(\lambda)}_+\CI.
\end{split}\end{equation*}
Replacing $E_\pm$ by $E_\pm-E_{1,\pm}$,
leaving the first two terms unchanged, we can iterate away the last term
exactly as above to obtain
\begin{equation*}
PE_\pm\in x^{\hat s_\pm(\lambda)+1}(1-r^2)^{\hat s_\pm(\lambda)+1}_+\CI
+x^{\hat s_\pm(\lambda)+1}(1-r^2)^{\hat s_\pm(\lambda)+1}_+\log(1-r^2)_+\CI.
\end{equation*}

Repeating this argument proves this proposition. Note that we obtain
arbitrarily large powers of logarithms, but these correspond to
increasingly less singular terms in terms of the power $s$ in $(1-r^2)^s_+$.
\end{proof}

As we would like our operator $E_\pm$ to be localized in the
interior of light cone (for hyperbolic propagation would spread singularities
outside otherwise and $E_\pm$ could not satisfy ), it is convenient to consider $P_\sigma$ as an operator
on tempered distributions in
\begin{equation*}
\BB^{n-1}_{1/2}=\{Y:\ |Y|\leq 1\},
\end{equation*}
here
equipped with the smooth structure arising from adjoining $\sqrt{1-|Y|^2}$
to the smooth structure induced from the front face (this is
what the subscript $1/2$ denotes). Let $\nu=(1-r^2)^{1/2}$ be a defining
function for $\pa\BB^{n-1}_{1/2}$. If only even powers of $\nu$ occur
as coefficients of products of $\nu D_\nu$ and $\nu V$, $V$ a vector
field on $\pa\BB^{n-1}_{1/2}$ extended to a neighborhood using the polar
coordinate decompositions, then one calls the corresponding differential
operator even, see \cite{Guillarmou:Meromorphic}. Note that the subspace of
even elements of $\CI(\BB^{n-1}_{1/2})$ is exactly $\CI(\BB^{n-1})$. Then:

\begin{lemma}
$P_\sigma\in\nu^{-2}\Diff^2_0(\BB^{n-1}_{1/2})$ is elliptic and even.

For
$\sigma$ real with
$\lambda+\sigma^2-\sigma(n-1)\geq 0$, $-P_\sigma$ is positive
with repect to the
$L^2(\BB^{n-1}_{1/2},(1-\nu^2)^{(n-3)/2}\nu^{1+2\sigma}\,d\nu\,d\omega)$
inner product on
\begin{equation*}
\nu H^1_0(\BB^{n-1}_{1/2},(1-\nu^2)^{(n-3)/2}\nu^{1+2\sigma}\,d\nu\,d\omega),
\end{equation*}
with $d\omega$ denoting the standard measure on the unit sphere.
\end{lemma}

\begin{proof}
As $P_\sigma$ is a differential operator with smooth coefficients on all
of $\ff$, elliptic for $r<1$, we only need to analyze its behavior near
$r=1$. For this purpose it is convenient to use the boundary
defining function $\nu$ on $\BB^{n-1}_{1/2}$.
A straightforward calculation using
$(1-r^2)^{1/2}D_r=-(1-\nu^2)^{1/2}D_\nu$ gives that in fact
\begin{equation*}\begin{split}
-P_\sigma
&=
(D_\nu+i(2\sigma-1)\nu^{-1}+i(n-3)\nu(1-\nu^2)^{-1})(1-\nu^2)D_\nu\\
&\qquad\qquad\qquad+\frac{1}{1-\nu^2}\,\Delta_\omega
+\lambda+\sigma^2-\sigma(n-1)\\
&=
\nu^{-1}\left((\nu D_\nu+i(2\sigma-1)+\frac{i(n-3)\nu^2}{1-\nu^2})(1-\nu^2)
(\nu D_\nu-i)\right.\\
&\qquad\qquad
\left.+\frac{\nu^2}{1-\nu^2}\,\Delta_\omega
+(\lambda+\sigma^2-\sigma(n-1))\nu^2 \right)\nu^{-1}
\end{split}
\end{equation*}
from which the first claim follows immediately.
For the second claim we merely need to notice that the formal adjoint of
$D_\nu\nu=\nu D_\nu-i$ with respect to $f\nu^{-1}\,d\nu\,d\omega$
$f=(1-\nu^2)^{(n-3)/2}\nu^{2+2\sigma}$, is
$f^{-1}(\nu D_\nu-i)f=\nu D_\nu+i(2\sigma-1)+i(n-3)\nu^2(1-\nu^2)^{-1}$,
so
\begin{equation*}
\langle u,-P_\sigma u\rangle=\|(1-\nu^2)^{1/2}
D_\nu u\|^2+\|(1-\nu^2)^{-1/2}d_\omega u\|^2
+(\lambda+\sigma^2-\sigma(n-1))\|u\|^2.
\end{equation*}
\end{proof}

In fact, it is also convenient to identify the interior of $\BB^{n-1}_{1/2}$
with the Poincar\'e ball model of hyperbolic $(n-1)$-space $\HH^{n-1}$
using polar coordinates around the origin, letting $\cosh\rho=\nu^{-1}$,
$\rho$ is the distance from the origin. The Laplacian on $\HH^{n-1}$ in
these coordinates is
\begin{equation*}
\Delta_{\HH^{n-1}}=D_{\rho}^2-i(n-2)\coth\rho\,D_{\rho}+(\sinh\rho)^{-2}
\Delta_{\omega}.
\end{equation*}

\begin{lemma}\label{lemma:P_sigma-hyp}
Let $s$ be such that $2s=\sigma-\frac{n}{2}$. Then
\begin{equation*}\begin{split}
&(1-r^2)^{-s}P_\sigma(1-r^2)^s
=\nu^{\frac{n}{2}-\sigma}P_\sigma\nu^{\sigma-\frac{n}{2}}\\
&=-\nu^{-1}\left(\Delta_{\HH^{n-1}}+\sigma^2-\left(\frac{n-2}{2}\right)^2
+\nu^2\left(\lambda-\frac{n(n-2)}{4}\right)\right)\nu^{-1}\\
&=-\cosh\rho\left(\Delta_{\HH^{n-1}}+\sigma^2-\left(\frac{n-2}{2}\right)^2
+(\cosh\rho)^{-2}\left(\lambda-\frac{n(n-2)}{4}\right)\right)\cosh\rho.
\end{split}
\end{equation*}
\end{lemma}

Thus, this conjugate of $P_\sigma$ is essentially a compact perturbation
of the hyperbolic Laplacian, shifted by the eigenparameter
$(n-2)^2/4-\sigma^2$. Note that the spectrum of $\Delta_{\HH^{n-1}}$
on $L^2(\HH^{n-1})$ is $[(n-2)^2/4,\infty)$. In fact, we have the following
result of Mazzeo and Melrose \cite{Mazzeo-Melrose:Meromorphic}:

\begin{lemma}
The operator
\begin{equation*}
L_\sigma=\Delta_{\HH^{n-1}}+\sigma^2-\left(\frac{n-2}{2}\right)^2
+\nu^2\left(\lambda-\frac{n(n-2)}{4}\right)
\end{equation*}
is invertible on
\begin{equation*}
L^2(\HH^{n-1},\mu_{\HH^{n-1}})=L^2(\BB^{n-1}_{1/2},(1-\nu^2)^{(n-3)/2}
\nu^{1-n}\,d\nu)=L^2(\BB^{n-1}_{1/2},(\sinh\rho)^{n-1}\,d\rho)
\end{equation*}
for $\sigma^2\nin\RR$, and it is Fredholm in $\sigma^2$
for $\sigma^2\in\Cx\setminus[0,\infty)$.

For $\sigma>0$,
any element of the $L^2$-nullspace of $L_\sigma$
lies in $\nu^{(n-2)/2+\sigma}
\Cinf(\BB^{n-1}_{1/2})$.

The inverse $L_\sigma^{-1}$ is meromorphic for
$\sigma^2\in\Cx\setminus[0,\infty)$ with finite rank residues, maps
$\nu^k H^m_0(\BB^{n-1}_{1/2},\mu_{\HH^{n-1}})\to
\nu^k H^{m+2}_0(\BB^{n-1}_{1/2},\mu_{\HH^{n-1}})$ continuously, provided that
$|k|<|\re\sigma|$. For $k>|\re\sigma|$, it maps
\begin{equation*}
\nu^k H^m_0(\BB^{n-1}_{1/2},\mu_{\HH^{n-1}})\to
\nu^k H^{m+2}_0(\BB^{n-1}_{1/2},\mu_{\HH^{n-1}})
+\nu^{(n-2)/2+\sigma}\Cinf(\BB^{n-1}).
\end{equation*}
\end{lemma}

In fact, $L_\sigma^{-1}$, defined at first in $\re\sigma>0$, extends
meromorphically to all of $\Cx$ (i.e.\ the Riemann surface of
$\sigma^2$), as shown in \cite{Mazzeo-Melrose:Meromorphic} with improvements
in \cite{Guillarmou:Meromorphic}:

\begin{lemma}
The operator $L_\sigma^{-1}$ defined at first for $\re\sigma>0$
as the inverse of $L_\sigma$, extends to a meromorphic family
of operators
\begin{equation*}
R_0(\sigma):\nu^k H^m_0(\BB^{n-1}_{1/2},\mu_{\HH^{n-1}})\to
\nu^k H^{m+2}_0(\BB^{n-1}_{1/2},\mu_{\HH^{n-1}})
+\nu^{(n-2)/2+\sigma}\Cinf(\BB^{n-1}),
\end{equation*}
$k>|\re\sigma|$ with no poles for $\sigma\neq 0$
pure imaginary, which satisfies $L_\sigma R_0(\sigma)=\Id$ on
$\nu^k H^m_0(\BB^{n-1}_{1/2},\mu_{\HH^{n-1}})$.

Moreover, $\sigma$ is a pole of $R_0$, then $L_\sigma u=0$ has
a non-zero solution
\begin{equation*}
u\in\nu^{(n-2)/2+\sigma}\Cinf(\BB^{n-1}).
\end{equation*}
\end{lemma}

\begin{cor}\label{cor:P_sigma-Fredholm}
For $\sigma^2\in\Cx\setminus[0,\infty)$, $\re\sigma>0$,
the operator $P_\sigma$ is
Fredholm, of index $0$, as a map
\begin{equation*}
P_\sigma:\nu^k  H^m_0(\BB^{n-1}_{1/2},\mu_{\HH^{n-1}})\to
\nu^{k-2} H^{m+2}_0(\BB^{n-1}_{1/2},\mu_{\HH^{n-1}})
\end{equation*}
for $-\frac{n-2}{2}<k<2\re\sigma-\frac{n-2}{2}$, $P_{\sigma}^{-1}$ is
meromorphic, with finite rank poles, and all poles satisfy $\sigma^2\in\RR$.

Moreover, for $\sigma>0$, elements of the nullspace of $P_\sigma$
on $\nu^k  H^m_0(\BB^{n-1}_{1/2},\mu_{\HH^{n-1}})$, $k$ as above,
lie in $\nu^{2\sigma}\Cinf(\BB^{n-1}_{1/2})$.

In addition, for $k>2\re\sigma-\frac{n-2}{2}$, whenever $P_\sigma$ is
invertible on $L^2$,
\begin{equation*}
P_\sigma^{-1}:\nu^k H^m_0(\BB^{n-1}_{1/2},\mu_{\HH^{n-1}})\to
\nu^{k-2} H^{m+2}_0(\BB^{n-1}_{1/2},\mu_{\HH^{n-1}})
+\nu^{2\sigma}\Cinf(\BB^{n-1}).
\end{equation*}

Finally, $R(\sigma)=P_\sigma^{-1}$, $\re\sigma>0$, extends to a meromorphic
family
\begin{equation*}
R(\sigma):\nu^k H^m_0(\BB^{n-1}_{1/2},\mu_{\HH^{n-1}})\to
\nu^{k-2} H^{m+2}_0(\BB^{n-1}_{1/2},\mu_{\HH^{n-1}})
+\nu^{2\sigma}\Cinf(\BB^{n-1}),
\end{equation*}
$k>2|\re\sigma|-\frac{n-2}{2}$, with no poles for $\sigma\neq 0$
pure imaginary, and $P_\sigma R(\sigma)=\Id$ on
$\nu^k H^m_0(\BB^{n-1}_{1/2},\mu_{\HH^{n-1}})$, $k$ as above.
If $\sigma$ is a pole of $R_0$, then $P_\sigma u=0$ has
a non-zero solution
\begin{equation*}
u\in\nu^{2\sigma}\Cinf(\BB^{n-1}).
\end{equation*}
\end{cor}

\begin{proof}
$P_\sigma=-\nu^{\sigma-\frac{n}{2}-1}L_\sigma\nu^{-\sigma+\frac{n}{2}-1}$,
so
\begin{equation*}
P_\sigma^{-1}=-\nu^{\sigma-\frac{n}{2}+1}L_\sigma^{-1}
\nu^{-\sigma+\frac{n}{2}+1}.
\end{equation*}
\end{proof}

Note that $1$ just barely fails to be in
$\nu^{-(n-2)/2}L^2(\BB^{n-1}_{1/2},\mu_{\HH^{n-1}})$, while
\begin{equation*}
\nu^{2\sigma}\Cinf(\BB^{n-1})\subset\nu^{k}L^2(\BB^{n-1}_{1/2}
,\mu_{\HH^{n-1}})
\end{equation*}
for $k<2\re\sigma-\frac{n-2}{2}$.

If $\lambda<0$, $\hat s_+(\lambda)>0$,
and $P_{\hat s_+(\lambda)}$ fails to be invertible on the spaces listed above
as $P_{\hat s_+(\lambda)} \nu^{2\hat s_+(\lambda)}=0$, and
$\nu^{2\hat s_+(\lambda)}$ lies in these
spaces. However, we claim that $P_\sigma$ is invertible
for $\sigma>\hat s_+(\lambda)$.
In fact,
\begin{equation*}\begin{split}
-P_\sigma &=-\nu^{2\sigma}(\nu^{-2\sigma}
P_\sigma\nu^{2\sigma})\nu^{-2\sigma}=-\nu^{2\sigma}P_{-\sigma}\nu^{-2\sigma}\\
&=\nu^{-1}(\nu D_\nu-i)^*(1-\nu^2)(\nu D_\nu-i)\nu^{-1}
+\frac{1}{1-\nu^2}\Delta_\omega+\lambda+\sigma^2+\sigma(n-1),
\end{split}\end{equation*}
with adjoint taken relative to $(1-\nu^2)^{(n-3)/2}\nu^{1-2\sigma}\,d\nu
\,d\omega$. The first two terms are positive with respect to
the corresponding $L^2$ space, while the roots of
$\lambda+\sigma^2+\sigma(n-1)$ are exactly $\hat s_\pm(\lambda)$,
so $\lambda+\sigma^2+\sigma(n-1)>0$ for $\sigma>\hat s_+(\lambda)$.
As $\nu^{2\sigma}\CI(\BB^{n-1})
\subset H^1_0(\BB^{n-1}_{1/2},\nu^{1-2\sigma}\,d\nu\,d\omega)$, it
follows from Corollary~\ref{cor:P_sigma-Fredholm} that
$P_\sigma$ has no nullspace in the listed spaces, so it is invertible.
(A different way of arguing would have been to note that
$\nu P_{\hat s_+(\lambda)}\nu$
has a positive eigenfunction, $\nu^{-1+2\hat s_+(\lambda)}$, which thus
must correspond to the bottom of the spectrum.)

That for $\lambda\geq 0$ the poles do not occur follows from
the following lemma as
\begin{equation*}
\nu L^2(\BB^{n-1}_{1/2},\nu^{n-2\re\sigma}
\mu_{\HH^{n-1}})=\nu^{1-\frac{n}{2}+\re\sigma}
L^2(\BB^{n-1}_{1/2},\mu_{\HH^{n-1}}),
\end{equation*}
and $1-\frac{n}{2}+\re\sigma<2\re\sigma-\frac{n-2}{2}$ if $\re\sigma>0$.

\begin{lemma}\label{lemma:P_sigma-cont-spec-form}
$P_\sigma$ satisfies
\begin{equation}\begin{split}\label{eq:P-sigma-form}
P_\sigma
&=-(D_r-\frac{i\sigma r}{1-r^2})^*(1-r^2)(D_r-\frac{i\sigma r}{1-r^2})
-r^{-2}\Delta_\omega-\frac{\sigma^2}{1-r^2}-\lambda\\
&=-(D_\nu+i\sigma \nu^{-1})^*(1-\nu^2)(D_\nu+i\sigma \nu^{-1})
-\frac{1}{1-\nu^2}\,\Delta_\omega
-\sigma^2\nu^{-2}-\lambda
\end{split}\end{equation}
with the (formal) adjoint taken with respect to the measure
\begin{equation*}
\mu=
r^{n-2}(1-r^2)^{-\re\sigma}\,dr\,d\omega=(1-\nu^2)^{\frac{n-3}{2}}\nu^{1-2\re\sigma}\,d\nu\,d\omega=\nu^{n-2\re\sigma}\mu_{\HH^{n-1}}.
\end{equation*}
\end{lemma}

\begin{cor}\label{cor:P_sigma-inv}
Suppose that $\lambda<(n-1)^2/4$. Then $P_\sigma$ is invertible for
$\sigma>\max(0,\hat s_+(\lambda))$.
\end{cor}

In fact, we can analyze the poles of the analytic continuation $R(\sigma)$
rather accurately using special algebraic properties of $P_\sigma$.
Unlike the preceeding considerations, which were rather general, i.e.\ hold
for operators of the same form, the following relies on the precise
form of $P_\sigma$.

\begin{lemma}
The following identities hold:
\begin{equation*}
P_{\sigma-2}\Delta_Y=\Delta_YP_\sigma,
\ P_{\sigma+2}\nu^{2\sigma+4}\Delta_Y\nu^{-2\sigma}=\nu^{2\sigma+4}\Delta_Y
\nu^{-2\sigma}P_\sigma.
\end{equation*}
\end{lemma}

\begin{proof}
First, as $\Delta_Y$ is homogeneous of degree $-2$ with respect to
dilations on $Y$, $[y\pa_y,\Delta_Y]=-2\Delta_Y$, so $[YD_Y,\Delta_Y]
=2i\Delta_Y$.
As
\begin{equation*}
P_\sigma=(YD_Y-i(n-1-\sigma))(YD_Y+i\sigma)-\Delta_Y-\lambda,
\end{equation*}
we deduce that
\begin{equation*}\begin{split}
\Delta_YP_\sigma
&=P_\sigma\Delta_Y+[\Delta_Y,(YD_Y)^2+i(2\sigma-(n-1))YD_Y]\\
&=P_\sigma\Delta_Y-2i\Delta_Y(YD_Y)-2i(YD_Y)\Delta_Y+2(2\sigma-(n-1))\Delta_Y\\
&=P_\sigma\Delta_Y-4\Delta_Y-4i(YD_Y)\Delta_Y+2(2\sigma-(n-1))\Delta_Y\\
&=\left((YD_Y)^2+i(2\sigma-(n-1))YD_Y+\sigma(n-1-\sigma)-\Delta_Y-\lambda\right.\\
&\qquad\qquad\qquad\left.-4-4i(YD_Y)+4\sigma-2(n-1)\right)\Delta_Y\\
&=\left((YD_Y)^2+i(2(\sigma-2)-(n-1))YD_Y+(\sigma-2)(n-1-\sigma+2)\right.\\
&\qquad\qquad\qquad\left.-\Delta_Y-\lambda\right)\Delta_Y\\
&=P_{\sigma-2}\Delta_Y.
\end{split}\end{equation*}
Thus, using $P_{-\sigma}=\nu^{-2\sigma}P_\sigma\nu^{2\sigma}$ with
$\sigma$ replaced by $\sigma+2$ first, then with $\sigma$ replaced by
$-\sigma$,
\begin{equation*}\begin{split}
P_{\sigma+2}\nu^{2\sigma+4}\Delta_Y\nu^{-2\sigma}
&=\nu^{2\sigma+4}P_{-\sigma-2}\Delta_Y\nu^{-2\sigma}\\
&=\nu^{2\sigma+4}\Delta_YP_{-\sigma}\nu^{-2\sigma}
=\nu^{2\sigma+4}\Delta_Y\nu^{-2\sigma}P_{\sigma}
\end{split}\end{equation*}
as claimed.
\end{proof}

\begin{lemma}
Suppose that $\sigma$ is such that
$P_{\sigma+2}w=0$, $w\in\nu^{2(\sigma+2)}\CI(\BB^{n-1})$
implies $w=0$. If $P_\sigma u=0$ for some $u\in\nu^{2\sigma}
\CI(\BB^{n-1})$ then either $\sigma\in\hat s_{\pm}(\lambda)
-\Nat$ or $u=0$.
\end{lemma}

\begin{proof}
If $P_\sigma u=0$, then by the previous lemma,
$P_{\sigma+2}\nu^{2\sigma+4}\Delta_Y\nu^{-2\sigma} u=0$.
Moreover, $\nu^{-2\sigma}u\in\CI(\BB^{n-1})$, so 
$w=\nu^{2\sigma+4}\Delta_Y\nu^{-2\sigma} u\in \nu^{2(\sigma+2)}\CI(\BB^{n-1})$,
hence $w=0$. Thus, $v=\nu^{-2\sigma}u\in\CI(\BB^{n-1})$ satisfies
$\Delta_Y v=0$ and $P_{-\sigma}v=\nu^{-2\sigma}P_{-\sigma}\nu^{2\sigma}v
=\nu^{-2\sigma}P_{\sigma}u=0$. Thus, $(P_{-\sigma}+\Delta_Y)v=0$,
so
\begin{equation*}
\left((YD_Y)^2-i(n-1+2\sigma)YD_Y-(\lambda+\sigma(n-1+\sigma))\right)v=0.
\end{equation*}
Factoring the operator as $(YD_Y+i\alpha_+)(YD_Y+i\alpha_-)$ with
\begin{equation*}
\alpha_\pm=-\frac{n-1}{2}-\sigma\pm\sqrt{\left(\frac{n-1}{2}\right)^2-\lambda}
=\hat s_\pm(\lambda)-\sigma,
\end{equation*}
we deduce that $v$ satisfies either $(YD_Y+i\alpha_+)v=0$ or
$(YD_Y+i\alpha_-)v=0$, i.e.\ $v$ is homogeneous of degree $\alpha_+$
or degree $\alpha_-$. But $v$ is $\CI$ at the origin, so, unless $v\equiv 0$,
in either case the corresponding $\alpha$ must be a non-negative integer,
i.e.\ $\hat s_\pm(\lambda)-\sigma=m\in\Nat$, so $\sigma\in\hat s_\pm(\lambda)
-\Nat$, proving the lemma.
\end{proof}

\begin{cor}
The only possible poles of $R(\sigma)$ are $\sigma\in\hat s_\pm(\lambda)
-\Nat$. In particular, if $m$ is a positive integer, $R(\sigma)$ is regular
at $\sigma=\hat s_\pm(\lambda)+m$ unless $s_+(\lambda)-s_-(\lambda)
=2\sqrt{(\frac{n-1}{2})^2-\lambda}\in\Nat_+$.
\end{cor}

\begin{proof}
As noted in Corollary~\ref{cor:P_sigma-Fredholm},
$\sigma$ is a pole of $R$ if and only if there exists a non-zero
$u\in\nu^{2\sigma}\CI(\BB^{n-1})$
such that $P_\sigma u=0$. Moreover, if $\re\sigma>C$, $C$ sufficiently
large (depending on $\lambda$), then there exist no such non-trivial $u$
by Corollary~\ref{cor:P_sigma-inv}.
Correspondingly, if $\re\sigma\in (C-2,C]$ and $\sigma$ is a pole
of $R$, then the previous lemma shows that $\sigma\in\hat s_\pm(\lambda)-\Nat$.
Proceeding inductively we deduce the corollary.
\end{proof}

Now, if $\sigma$ is not a pole of $R$, then given $f\in
\dCI(\BB^{n-1}_{1/2})$, $P_\sigma v=f$ can be solved with
$v\in\nu^{2\sigma}\CI(\BB^{n-1})$.

If $\re\sigma>0$ and we extend $v$ as
$0$ to the rest of the fiber of the front face over $y'$,
$P_\sigma v$ is thus the extension of $f$.

In fact, as long as $2\sigma\notin -\Nat_+$, we can extend $v$
by expanding in Taylor series to finite order,
$v=\sum_{j=0}^N \nu^{2j}a_j+\nu^{2N+2}v'$, $v'$ $\CI$ near $\pa\BB^{n-1}$.
If we choose $N$ large enough so that $2\re\sigma+2N+2>0$, we can
extend $\nu^{2\sigma}v'$ to $\ff$ by extending it as $0$. On the
other hand, we can extend $\nu^{2\sigma+2j} a_j$ as
$(1-|Y|^2)^{\sigma+j}_+ a_j$. Thus, we obtain a distribution
$\tilde v$ on $\ff$. Now $P_\sigma$ is a second order differential operator
with $\CI$ coefficients, so $P_\sigma (1-|Y|^2)^{\sigma+j}_+ a_j$
has the form $(1-|Y|^2)^{\sigma+j-2}_+ b'_j$, with $b'_j$ smooth,
and as the principal symbol of $P_\sigma$ vanishes on the conormal
bundle of $\pa\BB^{n-1}$, it in fact has the form
$(1-|Y|^2)^{\sigma+j-1}_+ b_j$, with $b_j$ smooth, as long as
$\sigma+j$ is not a non-positive integer. In particular, we
deduce that $P_\sigma\tilde v=0$ provided that $P_\sigma v=0$.

This is the argument that requires using the analytic extension
of $R$ to $\re\sigma\leq 0$, which gives solutions
$v\in\nu^{2\sigma}\CI(\BB^{n-1})$ rather then using solutions involving
the other indicial root, $0$, which would give rise to $v\in\CI(\BB^{n-1})$,
and hence allow $P_\sigma v$ to have delta distribution terms at $\pa
\BB^{n-1}$. In particular, for $\re\sigma<0$, we cannot simply use
the conjugate (in the sense of Lemma~\ref{lemma:P_sigma-hyp})
of $L_{-\sigma}^{-1}$.

If $\sigma\in\hat s_-(\lambda)+\Nat_+$, $2\sigma\in-\Nat_+$
can hold only if $2\sqrt{(n-1)^2/4-\lambda}\in\Nat_+$; it can
never hold if $\sigma\in\hat s_+(\lambda)+\Nat_+$.
We thus deduce that with $s=\hat s_+(\lambda)$, or $s=\hat s_-(\lambda)$
under the additional assumption that $\hat s_+(\lambda)-\hat s_-(\lambda)
\notin\Nat$, we can
solve away the error in Taylor series to obtain 
\begin{equation*}
E_\pm\in x^s I^{m(s)}([X\times Y_+;\diag_{Y_+}],F),\ m(s)=-s-\frac{2n+1}{4},
\ s=\hat s_\pm(\lambda),
\end{equation*}
with
$E_\pm-E_{0,\pm}
\in x^{s+1}I^{m(s)}([X\times Y_+;\diag_{Y_+}],F)$, and $E_\pm$ supported
inside the light cone, $PE\in\dCI(X\times Y_+)$. This remaining error
can be removed using the results of Section~\ref{sec:local-solvability} to
obtain the same conclusion with $PE_\pm=0$
near $Y_+$. The standard FIO contruction
allows one to obtain $E_\pm$ with the same properties, except $PE_\pm$
supported
near $Y_-$, vanishing in a neighborhood of $Y_-$. We have thus proved:

\begin{prop}\label{prop:E(lambda)-near-Y_+}
Suppose that
$s=\hat s_+(\lambda)$, or $s=\hat s_-(\lambda)$
under the additional assumption that $\hat s_+(\lambda)-\hat s_-(\lambda)
\notin\Nat$, i.e.\ $\lambda\neq\frac{(n-1)^2-m^2}{4}$, $m\in\Nat$.
Then there
exists
\begin{equation*}
E_\pm(\lambda)
\in x^s I^{m(s)}([X\times Y_+;\diag_{Y_+}],F),\ m(s)=-s-\frac{2n+1}{4},
\end{equation*}
satisfying $PE_\pm(\lambda)\equiv 0$ near $Y_+\times Y_+$, $E_\pm(\lambda)$
supported
inside the light cone near $Y_+\times Y_+$, and
\begin{equation*}
E_\pm(\lambda)\phi=x^{s_\pm(\lambda)}v,\ v\in\CI(X),\ v|_{Y_+}=\phi
\end{equation*}
for all $\phi\in\CI(Y_+)$. Moreover,
$\sigma_{m(s)}(E_\pm)$
never vanishes.
\end{prop}

We let $E(\lambda)=E_+(\lambda)\oplus E_-(\lambda)$ be the Poisson operator
(near $Y_+$, where it solves $PE_\pm(\lambda)=0$).
However, much as it is useful in the interior of $X$ to renormalize using
powers of the Laplacian, the same holds here. The renormalization depends
on the choice of $x$ modulo $x^2\CI(X)$. So let $\Delta_h$ denote the
Laplacian of the boundary metric $h$, define $\Delta'_h$ analogously
to the case of Cauchy surfaces, i.e.\ is $\Id$ on the nullspace of $\Delta_h$,
and is $\Delta_h$ on its orthocomplement. The renormalized Poisson operator
is then
\begin{equation*}
\tilde E(\lambda)=E_+(\lambda)(\Delta'_h)^{(s_+(\lambda)-n/2)/2}\oplus
E_-(\lambda)(\Delta'_h)^{(s_-(\lambda)-n/2)/2}.
\end{equation*}
The $n/2$ in the exponent of $\Delta'_h$ is somewhat arbitrary, it is used
to normalize FIO's below to be zeroth order; any quantity differing from
$s_\pm(\lambda)$ by a constant ($s$-independent) amount would work.
By Proposition~\ref{prop:E(lambda)-near-Y_+}, the two components of
$\tilde E(\lambda)$ lie in
$x^{1-n/2} I^{-5/4}([X\times Y_+;\diag_{Y_+}],F;(\Cx^2)^*)$,
i.e.\ they have the same regularity in the interior of $X\times Y_+$
as well as the same behavior at the boundary.

\begin{prop}\label{prop:S+ep-FIO}
Suppose $\hat s_+(\lambda)-\hat s_-(\lambda)
\notin\Nat$, i.e.\ $\lambda\neq\frac{(n-1)^2-m^2}{4}$, $m\in\Nat$.
For $t_0$ sufficiently close to $1$,
the map sending scattering data at $Y_+$ to
Cauchy data at $S_{t_0}$ given by
\begin{equation*}
\cS_{+,t_0}:\CI(Y_+)^2\ni(g_+,g_-)
\mapsto(u|_{S_{t_0}},\pa_x u|_{S_{t_0}})\in\CI(S_{t_0})^2,
\end{equation*}
where $u$
is the solution of $Pu=0$ given by Theorem~\ref{thm:smooth-solns},
is the
Fourier integral operator with Schwartz kernel $E(\lambda)|_{\Sigma_+(\ep)
\times Y_+}\oplus \pa_x E(\lambda)|_{\Sigma_+(\ep)
\times Y_+}$.

Moreover, the renormalized map
\begin{equation*}\begin{split}
\tilde \cS_{+,t_0}
=R_{t_0}\tilde E(\lambda)=&R_{t_0}E_+(\lambda)
(\Delta'_h)^{(s_+(\lambda)-n/2)/2}\oplus
R_{t_0}E_-(\lambda)(\Delta'_h)^{(s_-(\lambda)-n/2)/2}\\
&\in I^0(S_{t_0}\times Y_+,F\cap (S_{t_0}\times Y_+);\cL(\Cx^2,\Cx^2)),
\end{split}\end{equation*}
with $R_{t_0}$ being the Cauchy data map at $t_0$,
$u\mapsto ((\Delta'_{t_0})^{1/2}u|_{S_{t_0}},Vu|_{S_{t_0}})$, $V$ a vector
field transversal to $S_{t_0}$,
is an invertible Fourier integral operator.
\end{prop}

\begin{proof}
Let $t_1<t_0$, but still sufficiently close to $1$.
Let $\chi\in\CI(X)$ be identically $1$ in a neighborhood of
$T\geq t_0$, supported in
$T>t_1$.
Let $u$ be the solution of $Pu=0$ given by Theorem~\ref{thm:smooth-solns},
and let
$v=\int_Y \chi E_+(\lambda) g_+\,dy+\int_Y \chi E_-(\lambda) g_-\,dy$. Then
at $Y_+$, $v$ has the asymptotics required by Theorem~\ref{thm:smooth-solns},
and $Pv=[P,\chi](\int_Y E_+ g_+\,dy+\int_Y E_- g_-\,dy)$ is supported
where $T\in(t_1,t_0)$.
For $l>\max(\frac{1}{2},l(\lambda))$,
let $v_1\in H^{1,l,-l}_0(X)$ be the solution of $Pv_1=-Pv$, As in the proof
of Theorem~\ref{thm:Cauchy-exist}, $v_1$ is identically $0$ for $T\geq t_0$,
is in $\CI(X^\circ)$, and has an asymptotic expansion at $Y_-$ as
in Theorem~\ref{thm:smooth-solns}. Thus, $v+v_1$ has all the properties
of $u$ required by the uniqueness part of Theorem~\ref{thm:smooth-solns},
so
\begin{equation*}
u=\int_Y \chi E_+ g_+\,dy+\int_Y \chi E_- g_-\,dy+v_1,
\end{equation*}
and $v_1\equiv 0$ at $S_{t_0}$. Thus, $\cS_{+,t_0}$ indeed has
Schwartz kernel
\begin{equation*}
E(\lambda)|_{S_{t_0}
\times Y_+}\oplus \pa_x E(\lambda)|_{S_{t_0}
\times Y_+},
\end{equation*}
and we have a similar expression for $\tilde E(\lambda)$.
As $F$ is transversal to $S_{t_0}$ and the restriction
map to $S_{t_0}$ is a Fourier integral operator of order $1/4$,
$\tilde E(\lambda)|_{S_{t_0}\times Y_+}$ is an FIO of order $-1$,
while $R_{t_0}\tilde E(\lambda)|_{S_{t_0}\times Y_+}$ is an FIO of order $0$.

In order to prove the invertibility of $\tilde
\cS_{+,t_0}$, it suffices to show that
it is elliptic in the sense that $\tilde \cS_{+,t_0}^*\tilde \cS_{+,t_0}$
and $\tilde \cS_{+,t_0}\tilde \cS_{+,t_0}^*$
are elliptic pseudo-differential operators,
where the adjoint is taken with respect to the Riemannian densities
on $S_{t_0}$ and $Y_+$.
Once this is shown, it follows that both the nullspace of $\cS_{+,t_0}$
and of its adjoint must lie in smooth matrix-valued functions, and
are finite dimensional. Consider
for instance $\tilde \cS_{+,t_0}$. For
such smooth Cauchy data $(g_+,g_-)$ at $Y_+$, the corresponding
solution of $\Box u=0$ is smooth in $X^\circ$, of the form given by
Theorem~\ref{thm:smooth-solns}, and the vanishing of its Cauchy data
at $S_{t_0}$ implies that in fact $u$ vanishes identically, hence
$g_\pm=0$, so $\tilde \cS_{+,t_0}$ has trivial nullspace. On the other
hand, suppose that $\tilde \cS_{+,t_0}^*$ is not injective,
i.e.\ $\tilde \cS_{+,t_0}$
is not surjective (e.g.\ on the $L^2$-spaces). Any element of the nullspace
of $\tilde\cS_{+,t_0}^*$ is smooth, so in this case there exist smooth non-zero
Cauchy data $(\psi_0,\psi_1)$
at $S_{t_0}$ which are $L^2$-orthogonal to the range of
$\tilde \cS_{+,t_0}$. Let $u$ be the solution of $Pu=0$ with these Cauchy data.
Let
$(g_+,g_-)$ be the leading coefficients of the asymptotics at $Y_+$, as in Theorem~\ref{thm:smooth-solns}.
Then $u=E_+(\lambda) g_++E_-(\lambda) g_-$ (since the right hand side has the
same asymptotics at $Y_+$ as the left hand side, so they are equal by the
uniqueness part of Theorem~\ref{thm:smooth-solns}). Therefore $(\psi_0,\psi_1)$
are in the range of $\tilde\cS_{+,t_0}$, so they vanish, which gives
a contradiction. Thus, $\tilde \cS_{+,t_0}^*$ is also injective.
This proves the invertibility of $\tilde \cS_{+,t_0}$ given its ellipticity.

In order to prove ellipticity, one needs to compute the principal
symbol of $\tilde \cS_{+,t_0}^*\tilde \cS_{+,t_0}$
and $\tilde \cS_{+,t_0}\tilde \cS_{+,t_0}^*$. Consider first the latter.
For each $\alpha=(z,\zeta)\in T^*S_{t_0}$ there are two bicharacteristics
of $\Box$ which contain a point over $z\in S_{t_0}$ whose image in
$T^* S_{t_0}=T^*_{S_{t_0}} X/N^*S_{t_0}$ is $(z,\zeta)$. Let
the corresponding points in $T^*{S_{t_0}}$ be
$\alpha_j=(t_0,z,\xi_j,\zeta)$, $j=+,-$,
where $\xi$ is the dual variable of the first coordinate, $T$.
These bicharacteristics
emanate from $S^*_\pm Y_+$ (one from $S^*_+Y_+$, one from $S^*_-Y_+$);
let $\beta_j=(y_j,\eta_j)$, $j=+,-$, be the corresponding points.
Let $\hat E_\pm=E_\pm(\Delta'_h)^{(s_\pm(\lambda)-n/2)/2}$. Then
H\"ormander's theorem on the composition of FIO's shows that the principal
symbol of $\tilde \cS_{+,t_0}\tilde \cS_{+,t_0}^*$ at $\alpha=(z,\zeta)$
is a constant times
\begin{equation*}\begin{split}
\sum_j&
\begin{bmatrix}\sigma(\Delta'_h)^{1/2}(\alpha)\sigma(\hat E_+)
(\alpha_j,\beta_j)&
\sigma(\Delta'_h)^{1/2}(\alpha)\sigma(\hat E_-)(\alpha_j,\beta_j)\\
\sigma(V)(\alpha_j)\sigma(\hat E_+)(\alpha_j,\beta_j)&
\sigma(V)(\alpha_j)
\sigma(\hat E_-)(\alpha_j,\beta_j)\end{bmatrix}\\
&\qquad\times
\begin{bmatrix}\sigma(\Delta'_h)^{1/2}(\alpha)\overline{\sigma(\hat E_+)
(\alpha_j,\beta_j)}&
\overline{\sigma(V)(\alpha_j)}\overline{\sigma(\hat E_+)(\alpha_j,\beta_j)}\\
\sigma(\Delta'_h)^{1/2}(\alpha)\overline{\sigma(\hat E_-)(\alpha_j,\beta_j)}&
\overline{\sigma(V)(\alpha_j)}
\overline{\sigma(\hat E_-)(\alpha_j,\beta_j)}\end{bmatrix}\\
&
=\sum_j (|\sigma(\hat E_+)|^2+|\sigma(\hat E_-)|^2)\sigma(\Delta')
\begin{bmatrix}1&r_j\\
\bar r_j&|r_j|^2\end{bmatrix},
\end{split}\end{equation*}
where $r_j=\frac{\overline{\sigma(V)}}
{\sigma(\Delta')^{1/2}}$, and where on the right hand
side the various principal symbols are evaluated at the same points as
on the left hand side, but suppressed in notation.
Thus, the principal symbol has the form
$\sum_j c_j\begin{bmatrix}1&r_j\\
\bar r_j&|r_j|^2\end{bmatrix}$ with $c_j>0$, and a straightforward
calculation shows that this matrix is positive definite, hence invertible,
provided $r_+\neq r_-$. But $r_+=r_-$ would imply that
$\alpha_+=\alpha_-$, which is not the case, so we conclude that
$\tilde \cS_{+,t_0}\tilde \cS_{+,t_0}^*$ is indeed elliptic.

The calculation for $\tilde \cS_{+,t_0}\tilde \cS_{+,t_0}^*$ is similar.
In this case, for each $\beta=(y,\eta)\in S^*Y_+$, there are two corresponding
bicharacteristics, again one including a point in $S^*_+Y_+$ and one
in $S^*_-Y_+$, which then cross $T^*_{S_{t_0}}X$ at
$\alpha_j=(x_j,z_j,\xi_j,\eta_j)$, $j=+,-$.
Thus, the principal symbol at $\beta=(y,\eta)$ is
\begin{equation*}\begin{split}
\sum_j&
\begin{bmatrix}\sigma(\Delta'_h)^{1/2}(\alpha_j)\overline{\sigma(\hat E_+)
(\alpha_j,\beta)}&
\overline{\sigma(V)(\alpha_j)}
\overline{\sigma(\hat E_+)(\alpha_j,\beta)}\\
\sigma(\Delta'_h)^{1/2}(\alpha_j)
\overline{\sigma(\hat E_-)(\alpha_j,\beta)}&
\overline{\sigma(V)(\alpha_j)}
\overline{\sigma(\hat E_-)(\alpha_j,\beta)}\end{bmatrix}\\
&\qquad\times
\begin{bmatrix}\sigma(\Delta'_h)^{1/2}(\alpha_j)\sigma(\hat E_+)
(\alpha_j,\beta)&
\sigma(\Delta'_h)^{1/2}(\alpha_j)\sigma(\hat E_-)(\alpha_j,\beta)\\
\sigma(V)(\alpha_j)\sigma(\hat E_+)(\alpha_j,\beta)&
\sigma(V)(\alpha_j)
\sigma(\hat E_-)(\alpha_j,\beta)\end{bmatrix}\\
&
=\sum_j (\sigma(\Delta')+|\sigma(V)|^2)|\sigma(\hat E_+)|^2\begin{bmatrix}1&r_j\\
\bar r_j&|r_j|^2\end{bmatrix},
\end{split}\end{equation*}
where now $r_j=\frac{\sigma(\hat E_-)}{\sigma(\hat E_+)}$, and where again
on the right hand
side the various principal symbols are evaluated at the same points as
on the left hand side, but suppressed in notation.

Now $\sigma_{-5/4}(\hat E_+)$ and $\sigma_{-5/4}(\hat E_-)$ satisfy the
same first order linear ODE along bicharacteristics, so their ratio along each
bicharacteristic is constant, hence are equal to the ratio evaluated
at the `initial point' at the front face of
$[X\times Y_+;\diag_{Y_+}]$ (where $\sigma(\hat E_+)$ has to be replaced
by $\sigma(x^{n/2-1}\hat E_+)$, etc.). For a given $(y',\eta')$,
the projection of the
two bicharacteristics hit the front face at $(y',Y)$, $Y=\pm\hat\eta'$, and
the bicharacteristics themselves hit the cotangent bundle over the
front face inside $N^*_{(y',Y)}F$ at $-Y d|Y|$. Thus,
\begin{equation*}
r_j=\frac{\sigma(x^{n/2-1}\hat E_-)(y',Y_j,-Y_j d|Y|)}
{\sigma(x^{n/2-1}\hat E_+)(y',Y_j,-Y_j d|Y|)}
\end{equation*}
But these can be calculated from the normal operators, which are
explicit, hence are easily evaluated as $1$, resp.\ $e^{i \pi (s_+(\lambda)
-s_-(\lambda))}$. Thus, if $s_+(\lambda)-s_-(\lambda)$ is not an even
integer, which we are assuming, the $r_j$ are unequal, so
$\tilde \cS_{+,t_0}^*\tilde \cS_{+,t_0}$ is indeed elliptic,
finishing the proof.
\end{proof}

We are now ready to prove one of our main results, that the scattering
operator is a Fourier integral operator.

\begin{thm}\label{thm:S-FIO}
Suppose that $\hat s_+(\lambda)-\hat s_-(\lambda)
\notin\Nat$, i.e.\ $\lambda\neq\frac{(n-1)^2-m^2}{4}$, $m\in\Nat$.
Then
$\cS(\lambda)$ is a Fourier integral operator with canonical relation given
by $\cS_{\cl}$, and
$\tilde\cS=\tilde\cS(\lambda)$
is an invertible elliptic $0$th order
Fourier integral operator with the same canonical relation.
\end{thm}

\begin{proof}
This is immediate from $\tilde \cS=\tilde \cS_{-,-1+\ep}^{-1}
\circ C_{1-\ep,-1+\ep}
\circ \tilde \cS_{+,1-\ep}$ for $\ep>0$ small.
Indeed, all operators are Fourier integral operators by
Proposition~\ref{prop:S+ep-FIO} (applied also at $Y_-$) and
Proposition~\ref{prop:FIO-int}, with canonical relation given
by the appropriate restriction of the bicharacteristic flow. Thus,
the projection of the canonical relation to each factor
for each of them has
surjective differential, so the composition is transversal, and H\"ormander's
theorem can be applied. As
\begin{equation*}\begin{split}
&\tilde \cS(\lambda)\\
&=
((\Delta'_h)^{-s_+(\lambda)/2+n/4}\oplus(\Delta'_h)^{-s_-(\lambda)+n/4})
\cS(\lambda)((\Delta'_h)^{s_+(\lambda)/2-n/4}
\oplus(\Delta'_h)^{s_-(\lambda)/2-n/4}),
\end{split}\end{equation*}
and the first and last operators are pseudodifferential, the theorem
follows.
\end{proof}

Theorem~\ref{thm:Cauchy-FIO} follows similarly, as the propagator mapping
Cauchy data at different $T$-slices to each other is an invertible
FIO, so it suffices to consider the case $t_0$ close to $1$,
in which case the inverse given by Proposition~\ref{prop:S+ep-FIO}
proves the claim.

\bibliographystyle{plain}
\bibliography{sm}

\end{document}